\documentclass[a4paper,11pt,reqno]{amsart}
\pdfoutput=1

\usepackage{amsmath,amsfonts,amssymb,amsthm}
\usepackage{mdwlist}
\usepackage{graphicx, tikz,subfigure,color,calc}
\usetikzlibrary{decorations.pathreplacing}

\definecolor{darkgreen}{rgb}{0.0,0.5,0.0}
\definecolor{darkblue}{rgb}{0.0,0.0,0.3}
\definecolor{nicosred}{rgb}{0.65,0.1,0.1}
\definecolor{light-gray}{gray}{0.6}
\definecolor{really-light-gray}{gray}{0.8}

\usepackage[english]{babel}
\usepackage[body={16cm,22cm},centering]{geometry} 
\usepackage{fancyhdr}
\newtheorem{theorem}{Theorem}[section]

\newtheorem{lemma}[theorem]{Lemma}
\newtheorem{proposition}[theorem]{Proposition}
\newtheorem{corollary}[theorem]{Corollary}

\theoremstyle{definition}
\newtheorem{remark}[theorem]{Remark}
\newtheorem{definition}[theorem]{Definition}

\numberwithin{equation}{section}

\def\Xint#1{\mathchoice
{\XXint\displaystyle\textstyle{#1}}%
{\XXint\textstyle\scriptstyle{#1}}%
{\XXint\scriptstyle\scriptscriptstyle{#1}}%
{\XXint\scriptscriptstyle\scriptscriptstyle{#1}}%
\!\int}
\def\XXint#1#2#3{{\setbox0=\hbox{$#1{#2#3}{\int}$ }
\vcenter{\hbox{$#2#3$ }}\kern-.6\wd0}}

\def\dashint{\Xint-}

\newcommand{\rg}{\operatorname{ran}}

\newcommand{\R}{{\mathbb R}}

\newcommand{\Esp}{{\mathcal{E}_h}}
\newcommand{\Ee}{{\mathcal{E}}}
\newcommand{\xx}{{\textbf{x}}}

\mathchardef\emptyset="001F

\DeclareMathOperator*{\argmin}{argmin}

\everymath{\displaystyle}

\title[Linearly constrained evolutions of critical points]{Linearly constrained evolutions of critical points and an application to cohesive fractures}
\author[M. Artina, F. Cagnetti, M. Fornasier, F. Solombrino]
{M. Artina, F. Cagnetti, M. Fornasier, F. Solombrino}
\address[M. Artina]{Faculty of Mathematics, Technische Universit\"at M\"unchen, Boltzmannstrasse 3, 85748, Garching, Germany}
\email{marco.artina@ma.tum.de}
\address[F. Cagnetti]{University of Sussex, Pevensey 2, Department of Mathematics, 
BN1 9QH, Brighton, United Kingdom}
\email{f.cagnetti@sussex.ac.uk}
\address[M. Fornasier]{Faculty of Mathematics, Technische Universit\"at M\"unchen, Boltzmannstrasse 3, 85748, Garching, Germany}\email{massimo.fornasier@ma.tum.de}
\address[F. Solombrino]{Faculty of Mathematics, Technische Universit\"at M\"unchen, Boltzmannstrasse 3, 85748, Garching, Germany}
\email{francesco.solombrino@ma.tum.de}
\date{\today}

\begin{document}


\begin{abstract}
We introduce a novel constructive approach to define time evolution of critical points
of an energy functional. Our procedure, which is different from other more established approaches based on viscosity approximations in infinite dimension,  is prone to efficient and consistent numerical implementations, 
and allows for an existence proof under very general assumptions. 
We consider in particular rather nonsmooth and nonconvex energy functionals, 
provided the domain of the energy is finite dimensional. 
Nevertheless, in the infinite dimensional case study of a cohesive fracture model, 
we prove a consistency theorem of a discrete-to-continuum limit. 
We show that a quasistatic evolution can be indeed recovered as a limit of evolutions 
of critical points of finite dimensional discretizations of the energy, constructed according to our scheme. 
To illustrate the results, we provide several numerical experiments both in one and two dimensions.
These agree with the crack initiation criterion, 
which states that a fracture appears only when the stress overcomes 
a certain threshold, depending on the material.
\end{abstract}

\maketitle

\vspace{6pt}

\vspace{6pt}


\bigskip

\begin{section}{Introduction}

In this paper we introduce a novel model of time evolution of physical systems 
through linearly constrained critical points of the energy functional.
In order to include all the applications we have in mind,
we consider both dissipative and nondissipative systems. 
 
Our approach is \textit{constructive}, and it can be numerically implemented, 
as we show with an application to cohesive fracture evolution. 
Since we are eventually interested in being able to perform reliable numerical simulations, 
we consider at first \textit{finite dimensional} systems. We then also 
give a concrete case study showing how our results can be adapted to 
describe infinite dimensional systems. 

Below we recall the general framework, to which we intend to contribute. Then we present and comment our results, also in comparison with related contributions appeared in recent literature.

\subsection{Critical points evolutions in the literature}

When describing the behaviour of a physical system, 
one can try may want to describe it through the time evolution of absolute minimizers of the energy. 
This modeling has been pursued, for instance, in \cite{CT, DalFrToa, DT02, DMZ, FG, Francfort-Larsen:2003, FrMa98}. From an abstract point of view, this amounts to requiring that a global stability 
condition is satisfied at every time. 
In this case, the notion of solution fits into the general scheme 
of energetic solutions to rate-independent systems (see \cite{Mielke}).
However, it is not always realistic to expect the energy to be actually minimized at every fixed time. 
In fact, global minimization may lead the system to change 
instantaneously in a very drastic way, and this is something which is 
not very often observed in nature.
For this reason, several authors recently introduced time evolutions 
only satisfying a  \textit{local stability} condition for the energy functional
\cite{A, C, DDMM, DalToa02, LT, KMZ, MRoS, Neg-Ort}.

More precisely, given a time-dependent functional $F: Y \times [0,T] \to \mathbb{R}$, 
where $Y$ is a Banach space, an evolution of critical points of $F$ 
is a function $u:[0,T] \to Y$ which satisfies
\begin{equation}\label{eq:evcrptsB}
0 \in \partial_u F(u(t),t), \quad \mbox{for a.e. } t \in [0,T],
\end{equation}
where $\partial_u F$ denotes the subdifferential of $F$ with respect to $u$. 
Typically, the existence of such an evolution is proven by a singular perturbation method. 
More precisely one first considers, for every $\varepsilon>0$,  the $\varepsilon$-gradient flow
\begin{equation}\label{eq:evcrptsBp}
-\varepsilon \dot u^\varepsilon \in  \partial_u F(u^\varepsilon(t),t)
\end{equation}
with an initial datum $ u^\varepsilon(0) = u_0$, where $u_0$ is a critical point of  $F(\cdot,0)$. 
Then, passing to the limit as $\varepsilon \to 0^+$, $u^\varepsilon$ 
converges to a function $u$ satisfying \eqref{eq:evcrptsB}. \\

We now give a detailed description of our results and we comment on them.

\subsection{Setting of the problem and main results.}

In this paper we consider the evolution of a system which is driven by a linear external constraint. 
This can model several of situations of interest, such as prescribed boundary data, integral constraints, or the coupling with a linear (partial) differential equation. We will state and discuss the problem in a discrete (finite-dimensional) setting, where our main results are obtained. Later on we will also comment on how to possibly recover solutions to problems defined in infinite dimension.

Let $\mathcal E$ and $\mathcal F$ be two Euclidean spaces
of dimension $n$ and $m$, respectively, with $n > m$.
We want to study the evolution in a time interval $[0,T]$ 
of a physical system, whose states are described by vectors $v \in \mathcal E$, 
and whose energy is given by a function $J: \mathcal E \to [0, + \infty)$.
We assume that a time dependent constraint $f: [0,T] \to \mathcal F$ is imposed.
More precisely, the only admissible states at each time $t \in [0,T]$ 
satisfy $A v = f (t)$, where $A: \mathcal E \to \mathcal F$
is a surjective linear operator. 
In addition, we consider a convex positively $1$-homogeneous
dissipation functional $\psi: \mathcal E \to [0, \infty)$.
In order to consider at the same time dissipative 
and nondissipative systems, 
we introduce a switching parameter $\alpha \in \{ 0, 1 \}$.

Before stating the main results of the paper, let us give introduce some notions, which are useful in the rest of the paper.
At any time $t \in [0,T]$ and for a fixed $\overline{v} \in \mathcal{E}$, 
one can consider the problem 
\begin{equation} \label{minimization}
\min_{Av = f (t)} \{ J (v) + \alpha \psi (v - \overline{v}) \}.
\end{equation}
If $u \in \mathcal{E}$ is a minimizer for \eqref{minimization}, then 
\begin{equation} \label{minvinc3}
A u = f (t) \quad \textnormal{ and } \quad 
\left( \partial J(u) + \partial (\alpha \psi) (u - \overline{v}) \right) \cap \rg(A^*)\neq~\emptyset,
\end{equation}
where $A^*: \mathcal{F} \to \mathcal{E}$ is the adjoint of $A$, 
$\text{ran} (A^*)$ denotes the range of $A^*$, 
and $\partial J (u)$ is the Fr\'echet subdifferential of $J$ at $u$.
A \textit{critical point of $v \mapsto J (v) + \alpha \psi (v - \overline{v})$ 
on the affine space $\mathbf{A}(f (t))$}
is any vector $u \in \mathcal{E}$ satisfying \eqref{minvinc3} where, 
for every $f \in \mathcal{F}$, we set
$$
\mathbf{A}(f):=\{v \in \mathcal E : A v = f \}.
$$
A \emph{discrete quasistatic evolution} 
with time step $\delta \in (0,1)$, initial condition $v_0 \in \mathcal{E}$,
and constraint $f$, is a right-continuous function $v_\delta:[-\delta, T]\to \Ee$ such that 
\begin{itemize}
 \item $v_\delta(t)=v_0$ for every $t \in [ -\delta, \delta)$;
 \item $v_\delta$ is constant in $[0,T] \cap [i \delta, (i + 1) \delta)$ 
 for all $i \in \mathbb{N}_0$ with $i \delta \leq T$;
 
 \item $v_\delta ( i \delta )$ is a critical point 
 of $v \mapsto J + \alpha \psi (v - v_\delta ( (i-1) \delta ))$ on the affine space $\mathbf{A} (f (i \delta))$ for every 
 $i \in \mathbb{N}_0$ with $i \delta \leq T$.
\end{itemize}

Moreover, we say that a measurable function $v:[0, T]\to \Ee$ 
is an \emph{approximable quasistatic evolution} 
with initial condition $v_0$ and constraint $f$, 
if for every $t \in [0,T]$ there exists a sequence $\delta_k \to 0^+$ 
and a sequence $( v_{\delta_k})_{k \in \mathbb{N}}$
of discrete quasistatic evolutions with time step $\delta_k$, 
initial condition $v_0$, and constraint $f$, such that
\begin{equation*} 
\lim\inf_{k\to +\infty}|v_{\delta_k}(t)-v(t)|_{\mathcal{E}} =0. 
\end{equation*}
We are now ready to state the main results of the paper (see Theorem~\ref{psiTeorexist}
and Theorem~\ref{Teorexist}).

\textbf{Dissipative systems.}
Let $v_0$ be a critical point of $v \mapsto J(v) + \psi (v - v_0)$ on the affine space $\mathbf{A} (f (0))$.
Under suitable assumptions on $J$, $\psi$, $A$, and $f$ 
(see Theorem~\ref{psiTeorexist}) we prove that there exist $v \in BV ([0,T];\mathcal{E})$ of bounded variation
and $q \in L^{\infty} ([0,T];\mathcal{F}')$ such that:

\begin{itemize}

\item[(A)] $v (\cdot)$ is an approximable quasistatic evolution 
with initial condition $v_0$ and constraint $f$;

\vspace{.2cm}

\item[(B)] 
Local stability: $A^* q (t) \in \partial J ( v (t)) + \partial \psi(0) \qquad \text{ for $\mathcal{L}^1$-a.e. } t \in [0,T]$ \footnote{More precisely, by constructing suitable representatives of $q(t)$, one can even ensure that local stability holds at all continuity points of $t\mapsto v(t)$. We do however not deal in detail with this technical aspect.};

\vspace{.2cm}

\item[(C)] Energy inequality: the function $s \mapsto \langle q (s) ,  \dot{f} (s) \rangle_{\mathcal{F}}$ 
belongs to $L^1 (0,T)$ and 
$$
J ( v (t_2) ) + \mathrm{Var}_{\psi} (v; [t_1, t_2]) 
\leq J ( v (t_1) ) +\int_{t_1}^{t_2} \langle q  (s) ,  \dot{f} (s) \rangle_{\mathcal{F}} \, \mathrm{d}s,
$$
for every $0 \leq t_1 < t_2 \leq T$, where the $\psi$-variation $\mathrm{Var}_{\psi} (v; [t_1, t_2])$ is defined by \eqref{psi-var} and  $\langle \cdot ,  \cdot \rangle_{\mathcal{F}}$ denotes the duality product in $\mathcal F$.

\end{itemize}
Any function $v \in BV ([0,T];\mathcal{E})$ satisfying (A), (B), and (C) is a
\textit{rate independent evolution}. We also remark that, as a consequence of the local stability (B), 
the energy inequality becomes an equality in all the nontrivial intervals where the solution $v(t)$ is absolutely continuous (see Theorem \ref{balance}, where also the nondissipative case is treated).
In addition, in such intervals the doubly nonlinear inclusion $A^* q (t) \in \partial J (v(t)) + \partial \psi (\dot{v}(t))$
is satisfied.
It is however a well-known fact that, due to nonconvexity, our solutions can in general develop time discontinuities, where additional dissipation appears (see \cite{MRS2}).

\noindent
\textbf{Non dissipative systems.}
When there is no dissipation we can still prove an existence result, 
although the evolution obtained is in general expected to be less regular in time.
This is due to the fact that the absence of dissipation causes loss of compactness, 
since simple estimates of the total variation of the approximating solutions are now missing. This is undoubtedly a point of great interest in the analysis of the degenerate case 
(see also \cite{ARS} for an abstract approach in the unconstrained setting).
To compensate the loss of compactness, we need to make an additional assumption on the energy functional $J$
(see condition (J4) in Section \ref{setting}).
If $v_0$ is a critical point of $J$ on the affine space $\mathbf{A} (f (0))$, 
we prove (see Theorem~\ref{Teorexist}) that 
there exists bounded and measurable functions $v: [0,T] \to \mathcal E$ and $q: [0,T] \to \mathcal F'$ such that:

\begin{itemize}

\item[(a)] $v (\cdot)$ is an approximable quasistatic evolution 
with initial condition $v_0$ and constraint $f$;

\vspace{.2cm}

\item[(b)] Local stability: $A^* q (t) \in \partial J (v (t))$ for every $t \in [0,T]$;

\vspace{.2cm}

\item[(c)] Energy inequality: the function $s \mapsto \langle q (s) ,  \dot{f} (s) \rangle_{\mathcal{F}}$ 
belongs to $L^1 (0,T)$ and 
\begin{equation}\label{en-ineq}
J (v (t)) \leq J (v_0) + \int_{0}^{t}  \langle q (s) ,  \dot{f} (s) \rangle_{\mathcal{F}} \, \mathrm{d}s,
\end{equation}
for every $t \in [0,T]$.
\vspace{.2cm}

\end{itemize}
We call any measurable and bounded function $v: [0,T] \to \mathcal E$ 
satisfying (a), (b), and (c) a
\textit{weak potential type evolution}. Evolutions of this kind (although without constraints) 
have been widely considered in literature, as limits for $\varepsilon\to 0$ of gradient flows of the type \eqref{eq:evcrptsBp} as in \cite{zanini}, or of systems with vanishing inertia (see \cite{Agostiniani, Nardini}). 
We observe that the term  $ \langle q  (t) ,  \dot{f} (t) \rangle_{\mathcal{F}}$ in (c) 
physically corresponds to the virtual power due to the external constraint.
In the case where $J$ were smooth, thanks to (b) this term could indeed be rewritten as
$\langle D J (u (t) ) ,  \dot{w}   (t) \rangle_{\mathcal{F}}$, 
for any smooth curve $w (t)$ with $A w(t) = f (t)$.

Note also that in our definition the precise value of $v$ at every point $t$ matters. 
In particular, the initial condition $v(0)=v_0$ has a meaning, and the energy inequalities
(C) and (c) need to be satisfied at every time.
Thus, we are in general not identifying functions differing on null sets, as it is usual in $L^p$ spaces.

We point out that the main novelty in our approach is not in the 
existence results per se, but rather in the constructive algorithmic procedure 
that we provide, see \eqref{algobis}.
Notice that, differently from the vanishing viscosity approach, 
the parameter $\eta$ appearing in \eqref{algobis} remains fixed, throughout all the algorithm.   
Heuristically, our inner loop aims at finding the nearest critical point through 
a sort of discretized instantaneous generalized gradient flow, see Figure \ref{GraphJ} and Figure \ref{GraphAlgorithm}.
This can be obtained by looking at the long time behavior 
of the minimizing movements of the functional, for a fixed time step $\delta$.
More details about this point are given in the next subsection. 

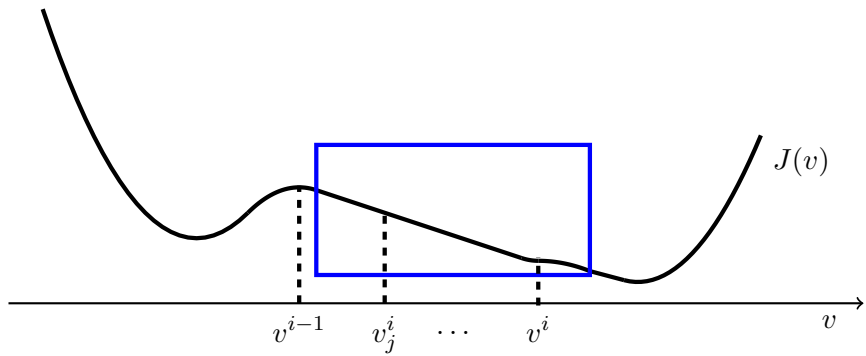
\begin{figure}[h]
\begin{tikzpicture}[domain=-10:15,xscale=0.45,yscale=0.15] 


\draw[thick, ->] (-10,-4)--(15,-4);


\draw[ultra thick, domain=-9:-3] plot (\x, {3*(\x+3) + (\x+3)^2+4}); 


\draw[ultra thick, domain=-3:-1] plot (\x, {-(\x+1) - (\x+1)^2 + 6}); 


\draw[ultra thick, domain=-1:5] plot (\x, {5 - \x}); 


\draw[ultra thick, domain=5:6-1.04/2.04] 
plot (\x, {(1.02)*(\x-6)^2+(1.04)*(\x-6)+0.02}) ;


\draw[ultra thick, domain=6-1.04/2.04:7] 
plot (\x, {(1.02)*(1.04/2.04)^2-(1.04)*(1.04/2.04)+0.02-0.4*(\x - 6+1.04/2.04)^2});


\draw[ultra thick, domain=7:8] plot (\x, {-0.8*(\x-7)
+  (1.02)*(1.04/2.04)*(1.04/2.04)-(1.04)*(1.04/2.04)+0.02-0.4*(7 - 6+1.04/2.04)*(7 - 6+1.04/2.04)}) ;


\draw[ultra thick, domain=8:12] plot (\x, {-0.8*(\x-8) + (\x-8)^2
-0.8 +  (1.02)*(1.04/2.04)*(1.04/2.04)-(1.04)*(1.04/2.04)+0.02-0.4*(7 - 6+1.04/2.04)*(7 - 6+1.04/2.04)})
node[below right] {$J (v)$} ;

\draw(14,-4.2)node[below]{$v$};


\draw[ultra thick, blue]  (-1,-1.5) rectangle (7,2*0.2*5*5);


\draw[dashed, ultra thick] (1,-4) -- (1,4);
 \draw(1,-4.2)node[below]{$v^i_j$};
 
  \draw(3,-5.5)node[below]{$\ldots$};

\draw[dashed, ultra thick] (-3/2,-4) -- (-3/2,6.4);
 \draw(-3/2,-4.2)node[below]{$v^{i-1}$};

\draw[dashed, ultra thick] (6-1.04/2.04,-4.2) -- (6-1.04/2.04,0.05);
 \draw(6-1.04/2.04,-4.2)node[below]{$v^{i}$};

\end{tikzpicture}
\caption{A pictorial idea of the algorithm when $\alpha = 0$.
Let $i \in \mathbb{N}$ be fixed, and suppose that $v^{i-1}$
is a critical point of $J$ on the affine space $\mathbf{A}(f ((i-1)\delta))$.
The algorithm allows to pass from $v^{i-1}$ to a vector $v^i$, 
which is a critical point of $J$ on the affine space $\mathbf{A}(f (i \delta))$.
With abuse of graphical notation, constrained critical points are identified
in the picture with those points where the slope of the functional $J$ is zero.
A magnified version of the area delimited by the blue rectangle can be seen in Figure \ref{GraphAlgorithm}.}
\label{GraphJ}
\end{figure}

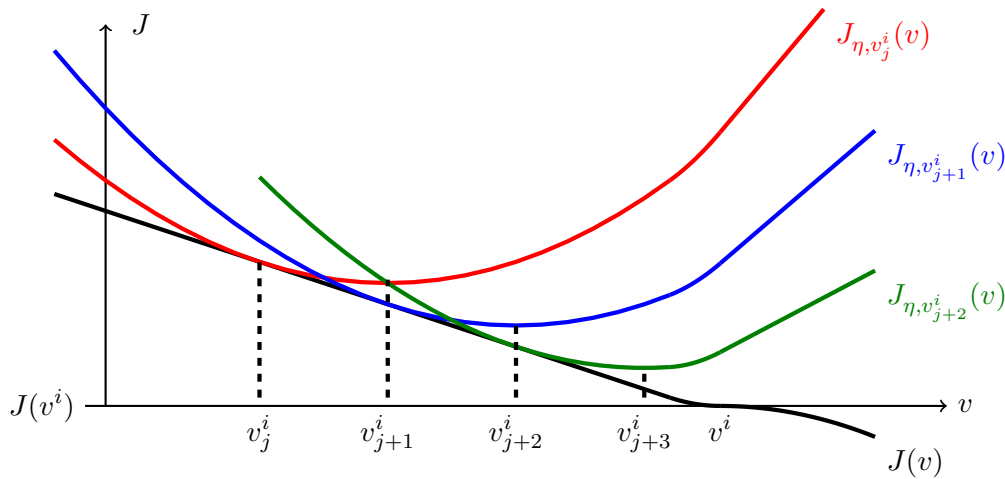
\begin{figure}[h]
\begin{tikzpicture}[domain=-1:5,xscale=1.5,yscale=0.5, scale=0.9] 


\draw[thick, ->] (-0.5,1.02*1.04^2/2.04^2-1.04*1.04/2.04+0.02) -- (-0.5,11) node[right]{\, $J$};


\draw[thick, ->] (-0.7, 1.02*1.04^2/2.04^2-1.04*1.04/2.04+0.02) -- 
(7.7, 1.02*1.04^2/2.04^2-1.04*1.04/2.04+0.02) node[right]{$v$};

 \draw(-0.7, 1.02*1.04^2/2.04^2-1.04*1.04/2.04+0.02)node[left]{$J (v^{i})$};


\draw[ultra thick, domain=-1:5] plot (\x, {5 - \x}); 
\draw[ultra thick, domain=5:6-1.04/2.04] 
plot (\x, {(1.02)*(\x-6)^2+(1.04)*(\x-6)+0.02}) ;
\draw[ultra thick, domain=6-1.04/2.04:7] 
plot (\x, {(1.02)*(1.04/2.04)^2-(1.04)*(1.04/2.04)+0.02-0.4*(\x - 6+1.04/2.04)^2}) node[below right] {$J (v)$};


\draw[red, ultra thick, domain=-1:5] plot (\x, {5 - \x + 2*0.2*(\x-1)^2}) ; 
\draw[red, ultra thick, domain=5:6-1.04/2.04] 
plot (\x, {(1.02)*(\x-6)^2+(1.04)*(\x-6)+0.02+ 2*0.2*(\x-1)^2}) ;
\draw[red, ultra thick, domain=6-1.04/2.04:6.5] 
plot (\x, {(1.02)*(1.04/2.04)^2-(1.04)*(1.04/2.04)+0.02-0.4*(\x - 6+1.04/2.04)^2+ 2*0.2*(\x-1)^2})
node[below right] {$J_{\eta, v^i_{j}} (v)$};


\draw[blue, ultra thick, domain=-1:5] plot (\x, {5 - \x + 2*0.2*(\x-1-1/(4*0.2))^2}) ;

\draw[blue, ultra thick, domain=5:6-1.04/2.04] 
plot (\x, {(1.02)*(\x-6)^2+(1.04)*(\x-6)+0.02+ 2*0.2*(\x-1-1/(4*0.2))^2}) ;

\draw[blue, ultra thick, domain=6-1.04/2.04:7] 
plot (\x, {(1.02)*(1.04/2.04)^2-(1.04)*(1.04/2.04)
+0.02-0.4*(\x - 6+1.04/2.04)^2+ 2*0.2*(\x-1-1/(4*0.2))^2})
node[below right] {$J_{\eta, v^i_{j+1}} (v)$}; 


\draw[darkgreen, ultra thick, domain=1:5] plot (\x, {5 - \x + 2*0.2*(\x-1-2/(4*0.2))^2}) ;

\draw[darkgreen, ultra thick, domain=5:6-1.04/2.04] plot (\x, {(1.02)*(\x-6)^2+(1.04)*(\x-6)+0.02+ 2*0.2*(\x-1-2/(4*0.2))^2}) ;

\draw[darkgreen, ultra thick, domain=6-1.04/2.04:7] plot (\x, {(1.02)*(1.04/2.04)^2-(1.04)*(1.04/2.04)+0.02-0.4*(\x - 6+1.04/2.04)^2 + 2*0.2*(\x-1-2/(4*0.2))^2}) node[below right] {$J_{\eta, v^i_{j+2}} (v)$}; 


\draw[dashed, ultra thick] (1,0) -- (1,4);
 \draw(1,-0.2)node[below]{$v^i_j$};


\draw[dashed, ultra thick] (1+1/0.8,0) -- (1+1/0.8,5 - 1-1/0.8 + 2*0.2/0.8^2 + 0.1) ;
\draw(1+1/0.8,-0.2)node[below]{$v^{i}_{j+1}$};


 \draw[dashed, ultra thick] (1+2/0.8,0) -- (1+2/0.8,5-1-2/0.8+0.72) ;
 \draw(1+2/0.8,-0.2)node[below]{$v^{i}_{j+2}$};
 

 \draw[dashed, ultra thick] (1+3/0.8,0) -- (1+3/0.8,0.92) ;
 \draw(1+3/0.8,-0.2)node[below]{$v^{i}_{j+3}$};


 \draw(6-1.04/2.04,-0.2)node[below]{$v^{i}$};

\end{tikzpicture}
\caption{A pictorial description of the sequence $\{ v^i_j \}_{j \in \mathbb{N}}$
provided by the algorithm in \eqref{algobis}.
For each $j \in \mathbb{N}$, the vector $v^i_j$ is the unique minimizer of the
strongly convex function $J_{\eta, v_{j-1}} (v):= J (v) + \eta |v - v^i_{j-1}|_{\mathcal{E}}^2$.
When $j \to \infty$, up to subsequences, we have that $v^i_j \to v^i$, 
where $v^i$ is a critical point of $J$ in the affine space $\mathbf{A}(f (i \delta))$.}
\label{GraphAlgorithm}
\end{figure}

As explained below, we also show how it is possible to apply our results to a model of cohesive fracture evolution.
In this specific application, we decided to work in the nondissipative setting, for several reasons.
First of all, many results are available in the literature in this case (see \cite{C}), 
and this allows us to test our methods. 
Secondly, the notions of fracture evolution which are used
to model the cohesive dissipative case are not easy to use in applications,  
since they require rather delicate tools of functional analysis (a formulation is spaces of Young measures, see \cite{CT}). 
Finally, this is a relevant  application for the degenerate case,   
where $BV$ estimates are not available.

We are aware that a more realistic model should take account of the monotonicity
of the crack growth, as done for instance in \cite{CT, Kruz,Vodicka2014}.
This issue can be dealt with in the case of brittle fracture, 
thanks to the \textit{Jump Transfer Lemma}.
Unfortunately, to date, this tool is not available in the cohesive case 
(see \cite{Francfort-Larsen:2003} for details), and even if it were existing, it would not help with the lack of $BV$ estimates.

\subsection{Comments on our result}
We would like to emphasize a few relevant and defining aspects of our results.

\begin{itemize}

\item[(i)] We prove the existence of an approximable quasistatic evolution 
for a large class of energy functions $J$ and linear operators $A$
(see \eqref{gamma} and conditions (J1)--(J4) in Section~\ref{setting}).
In particular, these include nonsmooth and nonconvex energy functionals. 

\vspace{.2cm}

\item[(ii)] We stress the fact that $v(t)$ is supposed to visit at different times 
$t$ {\it critical points} of the functional $v \mapsto J (v) + \alpha \psi (v - v (t))$ 
over the affine space $\mathbf A(f(t))$. 
This condition is rather general, if compared to the usual requirement
of focusing on global minimizers of $J$ over $\mathbf A(f(t))$. 
An evolution along critical points is in general more realistic and physically sound. Moreover, it is important to notice that the viscous approximation  usually does not provide an evolution along critical points, unless one lets the viscosity to vanish. Hence, in contrast to this well-established approximation method, our approach starts with a truly consistent approximation from the very beginning.
However, the analysis of such an evolution is usually very involved
and, in absence of dissipation, may allow for solutions $v$ that are just measurable in time.
While we shall be content with the generality of our results, we have to live with the fact
that our solutions may not be regular in the nondissipative case.

\vspace{.2cm}

\item[(iii)] Our approach is {\it constructive}. 
This is an important fact since, as we shall emphasize later, the functional 
$v\mapsto J(v) + \alpha \psi (v - v(t))$ may have multiple feasible 
critical points at the same time $t$. Hence, in order to promote uniqueness of evolutions, or even just their measurability, we need to design a proper selection principle. 
Accordingly, we shall design our selection {\it algorithm} in such a way that the selected critical point
is the closest - in terms of Euclidean distance -  to the one chosen at the previous instant of time, unless it is energetically convenient to perform a ``jump'' to another significantly different phase of the system (see formula \eqref{algobis}).
This corresponds to a rather common and well-established 
behavior of several physical (and non physical) systems (see, for instance, \cite[Section 9]{C}).  
Moreover, this method can be easily implemented 
by means of a corresponding numerical method \cite{AFS}, as we will show with the above mentioned application
to cohesive fracture evolution, see Section \ref{sec:exp}.
Thus, our evolution is the result of a constructive machinery (an algorithm), 
which is designed to emulate physical principles, according to which a critical point is selected
in terms of a balance between neighborliness (accounting the Euclidean path length between critical points) and energy convenience. 
In our view, this feature is of great relevance as we provide a black box, 
whose outputs are solutions.

\item[(iv)] As already mentioned, the proof of the main results (Theorems \ref{psiTeorexist} and \ref{Teorexist}) 
is given in a finite dimensional setting. This is due to the fact that, in the infinite dimensional case, 
the subdifferential is in general not closed with respect to the weak convergence in the domain of the energy. 
Such a difficulty could be overcome by requiring that the energy functional has compact sublevels, 
an assumption which is quite common in literature, provided the domain of the energy is suitably chosen.
On the other hand, the choice of a weaker topology for the domain may not always comply well with other conditions on the energy, that we need to prove the existence results. This is in particular the 
case of the key condition (J3) (see Section~\ref{setting}), which allows to control the virtual power, due to the external constraint, in terms of the energy. While not representing a major hurdle in the uncostrained case with time-dependent energy functionals, this issue seems particularly relevant for the problem we study (see Remark \ref{no-infinitedim} for further details).
This motivates our choice of first dealing with a finite dimensional setting, 
and then extending the results to our model case (see Theorem \ref{quasist}) with a problem-specific technique. 

\vspace{.2cm}

\item[(v)] As an important remark, we stress that in general all of the constants 
appearing in the technical assumptions in Section \ref{setting} could depend on the dimension 
of the considered Euclidean spaces.
Thus, our results can be applied to physical systems that can assume 
(a discrete or a continuum of) infinitely many states, 
provided all the relevant estimates obtained are dimension free.
For this reason, we state very clearly which are the parameters 
affecting the constants that come into play in the crucial proofs
(see Remark~\ref{bounds constants} and Remark~\ref{dependence}).
We give an important application to cohesive fracture evolution
(see Section \ref{appl} and Section \ref{hto0}), 
showing how also infinite dimensional systems can be approached with our method. 
In particular, we eventually provide an alternative proof of the existence of  evolutions of critical points 
for the cohesive fracture model firstly proven in \cite{C}.

\vspace{.2cm}

\vspace{.2cm}

\item[(vi)] The numerical simulations that we provide in Section~\ref{sec:exp} 
for the cohesive fracture evolutions agree with physically relevant requirements,
such as the \textit{crack initiation criterion} (see \cite[Theorem 4.6]{C}), 
which states that a crack appears only when the maximum sustainable stress of the material is reached.
We also mention that numerical simulations, obtained instead with the vanishing viscosity approach, have already appeared in \cite{Kne-Schr}.

\end{itemize}

To the best of our knowledge, this is the first time that 
an algorithm providing critical points evolution is introduced in such a generality, especially to treat consistently and stably also nondissipative models.

However, it is worth mentioning that our approach, although derived independently, resembles similar methods 
which have appeared recently in the literature. 
In \cite{negri13}, a related scheme has been for instance investigated in order to obtain 
a general existence result in a nonconvex but smooth setting.
The author also takes into account viscous dissipation effects,
and provides a constructive time rescaling, where the evolutions have a continuous dependence on time. 
This idea, in particular, generalises previous approaches for systems driven by nonconvex energy functionals (\cite{DM-Des-Sol, EM, MRS2, Stefanelli}). 
Moreover, the author shows an approximation result that in spirit is close to our Theorem \ref{quasist}, 
and to previous results in \cite{MRS, Serfaty}. 
However, the results in \cite{negri13} are obtained under the assumption of $C^1$ regularity of the energy functional, 
and in an \textit{unconstrained} setting. 
In particular, stability of critical points after passing to the limit is recovered through a very strong assumption 
(\cite[(8), Theorem 2.3]{negri13}), which would seem quite unnatural for a constrained nonsmooth problem.

Concerning other contributions, we also mention that 
a very general incremental minimization scheme, involving a quadratic correction with a fixed parameter $\mu>0$, has been just proposed
in \cite{Min-Sav}, even in connection with abstract dissipation distances in metric spaces. 

Another algorithm, showing some analogies to \eqref{algobis}, has furthermore been recently considered for a case study of phase field fracture coupled with damage in \cite{Neg-Knees}. In this case, the energy is nonconvex, but separately convex in the two state variables. Therefore, instead of adding a regularization, the authors define the discretized evolution through fixed points of an alternate minimization scheme. Also in this case a time reparametrization, where a full energy-dissipation balance holds, is provided.
The exploitation of similar techniques in connection with our problem is another interesting issue that we plan to pursue in the future.

\vspace{.2cm}

The plan of the paper is the following.
In Section \ref{setting} we state the main results of the paper, 
Theorem~\ref{psiTeorexist} and Theorem~\ref{Teorexist}, 
whose proofs are given in Section~\ref{proofmain1} 
and Section~\ref{secproof}, respectively.
Section \ref{appl} is devoted to the description of the cohesive fracture model
introduced in \cite{C}.
In the same section we introduce a space mesh and spatially discretize the problem. 
In Section \ref{hto0} we pass to the limit as the size of the mesh tends to $0$, 
thus obtaining a new proof of the result in \cite[Theorem 4.4]{C}.
Finally, numerical simulations are given in Section \ref{sec:exp}.
 
\end{section}

\begin{section}{Setting of the problem and main result}  \label{setting}

\subsection{Basic notation}

Throughout all the paper, we use the notation $\mathbb{N}_0:= \mathbb{N} \cup \{ 0 \}$, 
and we denote by $\mathcal{L}^1$ the standard Lebesgue $1$-dimensional measure in $\mathbb{R}$ .
Let $\mathcal E \simeq \mathbb{R}^n$ and $\mathcal F \simeq \mathbb{R}^m$ 
be two Euclidean spaces with dimension $n$ and $m$,
respectively, with $m < n$.
We consider an energy function $J: \mathcal E \to [0, + \infty)$, 
a linear operator $A: \mathcal E \to \mathcal F$, 
and a time dependent constraint $f: [0,T] \to \mathcal F$.
We will assume that $A$ is surjective.
Equivalently, we will suppose that that there exists $\gamma > 0$ such that
the adjoint operator $A^*: \mathcal F'  \to \mathcal E' $ satisfies 
\begin{equation} \label{gamma}
| A^* q |_{\mathcal{E}'} \geq \gamma | q |_{\mathcal{F}'} \quad \text{ for every }q \in \mathcal F'.
\end{equation}

We will frequently use the space $BV([a,b], X)$ of functions of bounded variation from a time interval $[a,b]$ to a Banach space $X$.

Before proceeding, let us recall some basic notions of differential calculus
that are used in the sequel.
Given $u \in \mathcal{E}$ and $S: \mathcal{E} \to [0, \infty)$, 
we recall that the Fr\'echet subdifferential 
$\partial S (u) \subset \mathcal{E}'$ of $S$ at a point $u \in \mathcal{E}$ is defined in the following way: 
$$
\xi \in \partial S (u) \quad \Longleftrightarrow \quad 
0 \leq \liminf_{v \to u} \frac{S (v) - \left( S (u) 
+ \langle \xi , v - u \rangle_{\mathcal{E}}  \right)}{|v-u|_{\mathcal{E}}},
$$
where $\langle \cdot , \cdot \rangle_{\mathcal{E}}$ denotes the 
standard scalar product in $\mathcal{E}$.
For every $f \in \mathcal{F}$, we set
$$
\mathbf{A}(f):=\{v \in \mathcal E : A v = f \}.
$$
We can now give a precise definition of critical point in our setting.
\begin{definition}
Let $f \in \mathcal F$ and $S: \mathcal{E} \to [0, \infty)$. Assume that $S$ has nonempty subdifferential at every point.
We say that $u \in \mathcal E$ is a \textit{critical point of $S$ 
on the affine space} $\mathbf{A}(f)$~if
\begin{equation}\label{minvinc1}
A u = f \quad \textnormal{ and } \quad \partial S (u) \cap \text{ran} (A^*) \neq \emptyset,
\end{equation}
where $\text{ran} (A^*)$ denotes the range of $A^*$.
\end{definition} 

\begin{remark} \label{liminf}
One can check that condition \eqref{minvinc1} implies, in turn, that
\begin{equation}\label{minvinc2}
0 \leq \liminf_{\varepsilon \to 0^+} 
\frac{S(u +\varepsilon w) - S(u)}{\varepsilon} \qquad \mbox{for every } w \in \ker(A),
\end{equation}
where $\ker(A)$ denotes the kernel of $A$.
\end{remark}

\begin{remark}
The definition above is motivated by the fact that if $u \in \mathcal{E}$ satisfies
$$
\min_{Av = f} S (v) = S (u),
$$
then \eqref{minvinc1} holds true.
\end{remark}
We also recall that a function $S: \mathcal{E} \to [0, \infty)$ 
is said to be \textit{strongly convex} if there exists $\nu > 0$ such that
whenever $v_1, v_2 \in \mathcal{E}$ we have
$$
\nu | v_1 -v_2 |_{\mathcal{E}}^2 \leq \langle \xi_1 - \xi_2 , v_1 -v_2 \rangle_{ \mathcal{E}}
\qquad  \mbox{for every } \, \xi_1 \in \partial S (v_1) \text{ and } \xi_2 \in \partial S (v_2).
$$

\subsection{The energy functional}

We can now state our assumptions on the energy functional $J$. We suppose that:
\begin{itemize}

\item[(J1)] the functional $v \longmapsto J (v) + |A v|^2_{\mathcal{F}}$ is coercive;

\vspace{.1cm}

\item[(J2)] there exists $\eta > 0$ such that $v \longmapsto J_{\eta} (v) 
:= J (v) + \eta | v|^2_{\mathcal{E}}$ is strongly convex;

\vspace{.1cm}

\item[(J3)] there exists $L > 0$ such that, for every $v \in \mathcal E$,
$\xi \in \partial J (v) \Longrightarrow | \xi |_{\mathcal{E}'} \leq L\left(J (v) + 1\right)$.

\end{itemize}
Before proceeding with the setting of the problem, we make some remarks
on the assumptions above. 
\begin{remark}\label{J-prop}
Condition (J2) implies that 
$J$ is a smooth perturbation of a convex function and, therefore, it is locally Lipschitz continuous. 
From this, it follows that $J$ is almost everywhere differentiable 
and its Fr\'echet subdifferential is nonempty at every point.
\end{remark}

\begin{remark} \label{vbar}
Let $\eta > 0$ be such that (J2) holds true, and let $\overline{v} \in \mathcal{E}$.
Then, the functional $J_{\eta, \overline{v}}: \mathcal{E} \to [0, \infty)$ defined as 
\begin{equation}\label{Jetavbar}
J_{\eta, \overline{v}} (v) := J (v) + \eta | v - \overline{v}|^2_{\mathcal{E}} \qquad 
\text{ for every }v \in \mathcal{E}, 
\end{equation}
is also strongly convex, with the same constant $\nu$. 
\end{remark}

\begin{remark} \label{sumsubdiff}
If $J$ is globally Lipschitz continuous with Lipschitz constant $L$, 
then
\begin{equation} \label{numero}
\xi \in \partial J (v) \Longrightarrow |\xi|_{\mathcal{E}'} \leq L, 
\end{equation}
for every $v \in \mathcal{E}$.
Moreover, if $J = J_1 + J_2$ where $J_1$ is lower semicontinuous 
and $J_2$ is of class $C^1$, then the decomposition  
\begin{equation} \label{numero2}
\partial J (v) = \partial J_1 (v) + D J_2 (v)
\end{equation}
holds true at every point $v \in \mathcal{E}$ such that $\partial J (v) \neq \emptyset$, 
where $D J_2 (v)$ denotes the  Fr\'echet derivative of $J_2$ at $v$.
\end{remark}

\begin{remark}
As shown in \cite[Remark 2.5]{AFS}, it suffices to check condition (J3) 
only at differentiability points of $J$, to ensure that it is satisfied at every point.
\end{remark}

\subsection{Dissipation Functional}

We introduce now the \textit{dissipation functional} $\psi: \mathcal E \to [0, \infty)$,
which measures the energy, which is lost when passing from one state to another one.
We assume that $\psi$ satisfies the following:

\begin{itemize}

\vspace{.2cm}

\item[($\Psi 1$)] $\psi (v) = 0 \mbox{ if and only if } v =0$;

\vspace{.2cm}

\item[($\Psi 2$)] $\psi (v_1 + v_2) \leq \psi (v_1) + \psi(v_2)$; 

\vspace{.2cm}

\item[($\Psi 3$)] $\psi (\lambda v) = \lambda \psi (v)$ for every $\lambda \geq 0$ 
and $v \in \mathcal E$;

\vspace{.2cm}

\end{itemize}

\begin{remark}
Note that assumptions ($\Psi 2$) and ($\Psi 3$) imply that $\psi$ is convex.
Note also that in general $\psi$ is not  a norm, unless symmetry holds
(i.e., $\psi (-v) = \psi (v)$ for every $v \in \mathcal E$).
\end{remark}

From the assumptions above it follows that there exists a constant $c > 0$ such that
\begin{equation}\label{boundsPsi}
\frac{1}{c} |v|_{\mathcal E} \leq \psi (v) \leq c |v|_{\mathcal E} \quad \text{ for every } v \in \mathcal E.
\end{equation}
Under these assumptions, it is well known 
(see, for instance, \cite{MRoS} and the references therein) that
\begin{equation} \label{sublimitato}
\emptyset \neq \partial \psi (v) \subset K^* \quad \text{ for every } v \in \mathcal E,
\end{equation}
where we set 
$$
K^*:= \partial \psi (0) = \{ \xi \in \mathcal E' :  \langle \xi, w \rangle_{\mathcal E} 
\leq \psi (w) \, \text{ for every } w \in \mathcal E \}.
$$
One can also check that, for every $v \in \mathcal E$, 
the subdifferential $\partial \psi (v)$ of $\psi$ at $v$ is characterised by
$$
\partial \psi (v) = K^* \cap \{ \xi \in \mathcal E' : \langle \xi, v \rangle_{\mathcal E} = \psi (v) \}.
$$

\begin{remark} \label{limir}
The set $K^*$ is convex (as it is a subdifferential) and bounded; indeed, the inclusion $K^*\subset B(0,c)$ (this one being the ball with radius $c$ in $\mathcal E'$) simply follows from ($\Psi 1$) and \eqref{boundsPsi}. 
\end{remark}

\begin{remark} 
Since $\Psi$ is proper convex, from (J2), the convexity of  \eqref{numero2}, and the Moreau-Rockafellar Theorem (\cite[Theorem 23.8]{Roc}), the decomposition
\[
\partial (J+\psi)(v)=\partial J(v)+\partial \psi(v)
\]
holds at every $v \in \mathcal E$.
\end{remark}

For $\psi$ as above and $u\in BV([a, b], \mathcal E)$, the $\psi$-variation of $u$ is defined as
\begin{equation}\label{psi-var}
\mathrm{Var}_{\psi}(u; [a,b])=\sup\left\{\sum_{i=0}^k \left(\psi(u(t_i))-\psi(u(t_{i-1}))\right): a=t_0<t_1<\dots<t_k=b,\, k\in \mathbb{N}\right\}\,.
\end{equation}
If one takes $\psi=|\cdot|$ in the above definition, one retrieves the usual definition for the pointwise variation of a function. For all $a<c<b$ the equality
\begin{equation}\label{additivity}
\mathrm{Var}_{\psi}(u; [a,b])=\mathrm{Var}_{\psi}(u; [a,c])+\mathrm{Var}_{\psi}(u; [c,b])
\end{equation}
immediately follows from the definition and the subadditivity of $\psi$.
If $u$ is additionally absolutely continuous, it is well known that
\begin{equation}\label{varuac}
\mathrm{Var}_{\psi}(u; [a,b])=\int_a^b \psi(\dot u(s))\,\mathrm{d}s\,. 
\end{equation}

\subsection{Main results}

In the following, together with rate independent evolutions, 
in which dissipation is present, we will also consider
\textit{weak potential type evolutions}, where there is no dissipation.
From the technical point of view, the absence of dissipation 
translates into a lack of compactness.
For this reason, we need an additional assumption to treat this case.
We shall assume

\begin{itemize}

\vspace{.1cm}

\item[(J4)] There exists a positive constant $C_{J,\eta} > 0$ such that
for every $\overline{v} \in \mathcal{E}$
\begin{equation} \label{a3}
A^* q_i \in \partial J_{\eta, \overline{v}} (v_i), \quad i = 1, 2
\quad \Longrightarrow \quad
\langle q_1 - q_2 , A v_1 - A v_2 \rangle_{\mathcal{F}} 
\leq C_{J,\eta} |v_1 - v_2 |_{\mathcal{E}} \, | A v_1 - A v_2 |_{\mathcal{F}}. 
\end{equation}

\end{itemize}

\begin{remark}
Although condition (J4) above might seem quite technical, 
it is automatically satisfied when $J \in C^{1,1}$.
Indeed, in this case $\partial J (v)$ is single valued at every $v \in \mathcal{E}$, 
and coincides with the differential $DJ (v)$.
Then, denoting by $M$ the Lipschitz constant of $D J (\cdot)$ and using \eqref{gamma}, one has
$$
| q_1 - q_2 |_{\mathcal{F}'} \leq \frac{1}{\gamma} |A^* q_1 - A^* q_2 |_{\mathcal{E}'}
= \frac{1}{\gamma} | D J(v_1) - D J(v_2) |_{\mathcal{E}'}
\leq \frac{M}{\gamma} |v_1 - v_2 |_{\mathcal{E}}.
$$
At this point, \eqref{a3} simply follows by the Cauchy-Schwarz inequality.
\end{remark}

\begin{remark}
We will show in a concrete example that condition (J4) can also be satisfied when $J \notin C^{1,1}$
(see Section~\ref{conditions a0-a3}). 
\end{remark}
Before stating our main results, we give again and in more detail the notion of discrete and 
approximable quasistatic evolution, respectively.
When this is possible, in the following we treat at the same time the cases
with and without dissipation.
To this aim, we introduce a switching parameter $\alpha \in \{ 0, 1 \}$, in such a way 
that $\alpha =0$ corresponds to the situation without dissipation, 
while in the case $\alpha = 1$ dissipation is present.

\begin{definition}\label{defdiscr}
Let $\alpha \in \{ 0, 1\}$, let $v_0 \in \Ee$ be a critical point of $J + \alpha \psi (v - v_0)$ 
on the affine space $\mathbf{A} (f (0))$, and let $\delta > 0$. 
A \emph{discrete quasistatic evolution} with time step $\delta$, initial condition $v_0$,
and constraint $f$ is a right-continuous function $v_\delta:[-\delta, T]\to \Ee$ such that 
\begin{itemize}
 \item $v_\delta(t)=v_0$ for every $t \in [ -\delta, \delta)$;
 \item $v_\delta$ is constant in $[0,T] \cap [i \delta, (i + 1) \delta)$ 
 for all $i \in \mathbb{N}_0$ with $i \delta \leq T$;
 
 \item $v_\delta ( i \delta )$ is a critical point 
 of $v \mapsto J + \alpha \psi (v - v_\delta ( (i-1) \delta ))$ on the affine space $\mathbf{A} (f (i \delta))$ for every 
 $i \in \mathbb{N}_0$ with $i \delta \leq T$.
\end{itemize}
\end{definition}

\begin{definition}\label{evolution}
Let $\alpha \in \{ 0, 1 \}$ and let $v_0 \in \Ee$ be a critical point of 
$v \mapsto J(v) + \alpha \psi (v - v_0)$ on the affine space $\mathbf{A} (f (0))$.
A bounded measurable function $v:[0, T]\to \Ee$ is said to be an
\emph{approximable quasistatic evolution} with initial condition $v_0$ 
and constraint $f$, if  there exists a sequence $\delta_k \to 0^+$ 
and a sequence $(v_{\delta_k}(t))_{k \in \mathbb{N}}$
of discrete quasistatic evolutions with time step $\delta_k$, 
initial condition $v_0$, and constraint $f$, such that, for every $t \in [0,T]$,
\begin{equation} \label{defi}
\lim\inf_{k\to +\infty}|v_{\delta_k}(t)-v(t)|_{\mathcal{E}} =0. 
\end{equation}
\end{definition}

We are now ready to state our main results.
The first one is an existence result for rate independent evolutions.
\begin{theorem}[Existence of rate independent evolutions] \label{psiTeorexist}
Let $\alpha = 1$, and suppose that \eqref{gamma}, (J1), (J2), and (J3) are fulfilled, 
and that ($\Psi$1), ($\Psi$2), and ($\Psi$3) hold true.
Let $f \in W^{1,2} ( [0,T] ; \mathcal F)$, and let 
$v_0$ be a critical point of $v \mapsto J + \psi (v - v_0)$ in the affine space $\mathbf{A} (f (0))$.
Then, there exist $v \in BV ([0,T];\mathcal{E})$ and $q \in L^{\infty} ([0,T];\mathcal{F}')$ such that:

\begin{itemize}

\item[\textit{(A)}] $v (\cdot)$ is an approximable quasistatic evolution 
with initial condition $v_0$ and constraint $f$;

\vspace{.2cm}

\item[\textit{(B)}] $A^* q (t) \in \partial J ( v (t)) + K^* \qquad \text{ for $\mathcal{L}^1$-a.e. } t \in [0,T]$;

\vspace{.2cm}

\item[\textit{(C)}] The function $s \mapsto \langle q (s) ,  \dot{f} (s) \rangle_{\mathcal{F}}$ 
belongs to $L^1 (0,T)$, and for every $0 \leq t_1 < t_2 \leq T$
$$
J ( v (t_2) ) + \mathrm{Var}_{\psi} (v; [t_1, t_2]) 
\leq J ( v (t_1) ) +\int_{t_1}^{t_2} \langle q  (s) ,  \dot{f} (s) \rangle_{\mathcal{F}} \, \mathrm{d}s.
$$

\end{itemize}

\end{theorem}
In the case without dissipation we need to add the additional assumption (J4),
and we obtain measurability, but in general no further regularity, of the evolution.
\begin{theorem}[Existence of weak potential type evolutions] \label{Teorexist}
Let $\alpha = 0$, and suppose that \eqref{gamma}, (J1), (J2), (J3), and (J4) are satisfied. 
Let $f \in W^{1,2} ( [0,T] ; \mathcal F)$, and let 
$v_0$ be a critical point of $J$ in the affine space $\mathbf{A} (f (0))$.
Then, there exist bounded and measurable functions 
$v: [0,T] \to \mathcal E$ and $q: [0,T] \to \mathcal F'$ such that:

\begin{itemize}

\item[\textit{(a)}] $v (\cdot)$ is an approximable quasistatic evolution 
with initial condition $v_0$ and constraint $f$;

\vspace{.2cm}

\item[\textit{(b)}] $A^* q (t) \in \partial J (v (t))$ for every $t \in [0,T]$;

\vspace{.2cm}

\item[\textit{(c)}] the function $s \mapsto \langle q (s) ,  \dot{f} (s) \rangle_{\mathcal{F}}$ 
belongs to $L^1 (0,T)$, and for every $0< t \in [0,T]$ we have 
\begin{equation*}
J (v (t)) \leq J (v_0) + \int_{0}^{t}  \langle q (s) ,  \dot{f} (s) \rangle_{\mathcal{F}} \, \mathrm{d}s.
\end{equation*}

\end{itemize}
\end{theorem}

\begin{remark} \label{boundevolution}
The explicit dependence of two constants $C_1$ and $C_2$ with 
$$
|q (t)|_{\mathcal{F}'} \le C_1 \quad 
\text{ and } \quad | v (t) |_{\mathcal{E}} \leq C_2 \qquad \text{ for every } t \in [0,T], 
$$ 
is given in Theorem \ref{teor18}.
\end{remark}

Notice that in the nondissipative case, the energy inequality can not be in principle stated in a proper subinterval $[t_1, t_2]$ of $[0,T]$ with $t_1>0$. This is because the measurable selection procedure we use to overcome lack of compactness is not in general enough to guarantee upper semicontinuity of the right-hand side (see Section \ref{secproof} for details). As mentioned in the Introduction, a time reparametrization technique, yielding an energy equality to hold in the rescaled time, would allow to deal with this difficulty. 

The proofs of Theorem \ref{psiTeorexist} and Theorem \ref{Teorexist} will be given in Sections \ref{proofmain1} and \ref{secproof}, respectively. Before discussing them, we conclude by showing that in both cases, as a consequence of the stability condition, the energy equality actually holds in all the subintervals where the solution happens to be absolutely continuous. On the other hand, it is well-known that solutions to problems as those we consider here are expected to be in general discontinuous because of nonconvexity of the energy.

\begin{theorem}\label{balance}
Let $\alpha=0$ or $\alpha=1$ and assume that the assumptions of Theorem \ref{Teorexist} or of Theorem \ref{psiTeorexist} are satisfied, respectively. Let $v$ be an approximable quasistatic evolution.

\begin{itemize}
 \item If $\alpha=1$ and $[t_1, t_2] \subset [0, T]$, assume additionally that $v$ is absolutely continuous in $[t_1, t_2]$. Then
\begin{equation}\label{energy-eq}
J ( v (t_2) ) + \mathrm{Var}_{\psi} (v; [t_1, t_2]) 
=J ( v (t_1) ) +\int_{t_1}^{t_2} \langle q  (s) ,  \dot{f} (s) \rangle_{\mathcal{F}} \, \mathrm{d}s.
\end{equation}
Furthermore, 
$$
A^* q (t) \in \partial J (v(t)) + \partial \psi (\dot{v}(t)) \qquad \text{ for a.e. } t \in (t_1, t_2). 
$$

\item If $\alpha=0$, assume additionally that $v$ is absolutely continuous in $[0, t]$ with $t>0$. Then
\begin{equation}\label{energy-eq2}
J ( v (t) )  
=J ( v_0 ) +\int_{0}^{t} \langle q  (s) ,  \dot{f} (s) \rangle_{\mathcal{F}} \, \mathrm{d}s.
\end{equation}
\end{itemize}
\end{theorem}

\begin{proof}
Clearly, only the ''$\geq$''-inequality in \eqref{energy-eq} or \eqref{energy-eq2} has to be shown. We begin by noticing that, since $J$ is locally Lipschitz by Remark \ref{J-prop}, under our assumption also the map $t\mapsto J(v(t))$ is absolutely continuous.
Let now $t\in [0,T]$ be a common differentiability point for $t\mapsto f(t)$, $t\mapsto v(t)$ and $t\mapsto J(v(t))$.

Now, for $\xi(t)\in \partial J(v(t))$ we have by definition of subdifferential
\begin{equation}\label{ineq}
J(v(t+h))-J(v(t))\ge \langle \xi(t), v(t+h)-v(t)\rangle_{\mathcal{E}}\,.
\end{equation}
If $\alpha=0$, we can take $\xi(t)=A^* q(t)$ we have
\[
J(v(t+h))-J(v(t))\ge \langle q(t), f(t+h)-f(t)\rangle_{\mathcal{F}}
\]
and differentiating
\[
\frac{\mathrm{d}}{\mathrm{d}t}J(v(t))\ge \langle q(t), \dot f(t)\rangle_{\mathcal{F}}\,,
\]
so that the conclusion follows by integration between $0$ and $t$.

If $\alpha=1$, we can take $\xi(t)=A^* q(t)-\zeta(t)$, with $\zeta(t)\in K^*=\partial \psi(0)$ in \eqref{ineq} to obtain
\[
J(v(t+h))-J(v(t))+ \langle \zeta(t), v(t+h)-v(t)\rangle_{\mathcal{E}}\ge \langle q(t), f(t+h)-f(t)\rangle_{\mathcal{F}}\,.
\]
Differentiating we have
\[
\frac{\mathrm{d}}{\mathrm{d}t}J(v(t))+ \langle \zeta(t), \dot v(t)\rangle_{\mathcal{E}}\ge \langle q(t), \dot f(t)\rangle_{\mathcal{F}}\,.
\]
For $\zeta(t)\in \partial \psi(0)$ it holds 
\begin{equation} \label{c'eravamo quasi}
\langle \zeta(t), \dot v(t)\rangle_{\mathcal{E}}\le \psi(\dot v(t))-\psi(0)=\psi(\dot v(t))
\end{equation} 
and therefore
\begin{equation} \label{ci siamo quasi}
\frac{\mathrm{d}}{\mathrm{d}t}J(v(t))+\psi(\dot v(t))\ge \langle q(t), \dot f(t)\rangle_{\mathcal{F}}\,.
\end{equation}
Therefore, by integration between $t_1$ and $t_2$, thanks to \eqref{varuac}, 
we get \eqref{energy-eq}.
Now, \eqref{ci siamo quasi} holds as an equality, and so does \eqref{c'eravamo quasi}.
Since $\psi (0)=0$, the inclusion $\zeta (t) \in \partial \psi (0)$ and the equality 
$\langle \zeta(t), \dot v(t)\rangle_{\mathcal{E}} =\psi(\dot v(t))$ together imply
$\zeta (t) \in \partial \psi (\dot{v} (t))$.
Since, by construction, $A^* q (t) = \xi (t) + \zeta (t)$ with $\xi (t) \in \partial J (v (t))$, 
this concludes the proof. 
\end{proof}

\section{Auxiliary Results}

In this section we prove some auxiliary results, that will be used to prove
both Theorem \ref{psiTeorexist} and Theorem \ref{Teorexist}.
We start by showing that, under suitable assumptions, an approximable quasistatic evolution 
is automatically an evolution of critical points.

\begin{proposition} \label{tre2}
Suppose that \eqref{gamma}, (J2), and (J3) are satisfied, 
and let $f \in W^{1,2} ( [0,T] ; \mathcal F)$ and $\alpha \in \{ 0, 1 \}$.
If $\alpha = 1$, suppose in addition that ($\Psi 1$), ($\Psi 2$) and ($\Psi 3$) hold true.
Let $v_0 \in \Ee$ be a critical point of $v \mapsto J + \alpha \psi (v - v_0)$ 
in the affine space $\mathbf{A} (f (0))$, 
and let $v:[0, T]\to \Ee$ be an approximable  quasistatic evolution 
with initial condition $v_0$ and constraint $f$.
Then, $v(t)$ is a critical point of $v \mapsto J + \alpha \psi (v - v(t))$ 
on the affine space $\mathbf{A} (f (t))$ for every $t \in [0,T]$.
\end{proposition}

\begin{proof}
Let $( v_{\delta_k})_{k \in \mathbb{N}}$ be as in \eqref{defi}, and let $t \in [0,T]$ be fixed.  
For every $k \in \mathbb{N}$, let $i_k \in \mathbb{N}$ be such that 
(to ease the notation, we do not stress the dependence of $i_k$ on $t$)
$$
i_k \delta_k \leq t < ( i_k + 1 ) \delta_k.
$$
From the definition of approximate quasistatic evolution we have $Av_{\delta_k}(t)=f(i_k \delta_k)$.
Then, by continuity of $f$ and \eqref{defi} we obtain $Av(t)=f(t)$.

We thus need only to show that $(\partial J(v(t)) + \alpha K^* ) \cap \text{ran} (A^*) \neq \emptyset$.
By definition of constrained critical point, there exists $q_k \in \mathcal F'$, 
$\xi_k \in \partial J(v_{\delta_k}(t))$, 
and $\zeta_k \in \partial (\alpha \psi) (v_{\delta_k}(t) - v_{\delta_k}(t-\delta_k))$ such that
\begin{equation}\label{condiz}
A^*q_k = \xi_k + \zeta_k, \qquad \text{ for every } k \in \mathbb{N}.
\end{equation}
From (J2) it follows that $J$ is locally bounded and therefore, by \eqref{defi}, we have
$$
\sup_{k \in \mathbb{N}} J(v_{\delta_k}(t))<+\infty\,.
$$
On the other hand, thanks to \eqref{sublimitato} and Remark \ref{limir}, we have 
$$
| \zeta_k |_{\mathcal{E}'} \leq c , \qquad \text{ for every } k \in \mathbb{N}.
$$
Therefore, thanks to \eqref{gamma}, \eqref{condiz}, and (J3)
$$
\sup_{k \in \mathbb{N}} |q_k|_{\mathcal{F}'} \le 
\frac1\gamma \sup_{k \in \mathbb{N}} |A^*q_k|_{\mathcal{E}'} 
\le \frac1\gamma\left(L\left( \sup_{k \in \mathbb{N}} J(v_{\delta_k}(t))+1\right)+c\right)<+\infty\,.
$$
Thus, there exists $q \in \mathcal F$ such that, up to subsequences, 
\begin{equation}\label{fatto}
\lim_{k\to +\infty}|q_k-q|_{\mathcal{F}'} =0\,.
\end{equation}
From \eqref{defi}, \eqref{condiz}, and \eqref{fatto} we get, 
by the closure property of the subdifferential, that
\begin{equation*}
A^*q \in \partial J(v(t)) + \alpha K^*\,,
\end{equation*}
as required.
\end{proof}

\subsection{A constructive approach}
In order to construct an approximate quasistatic evolution,
we first introduce an auxiliary minimum problem.
Let $\delta \in (0,1)$ be a fixed time step, and let $i \in \mathbb{N}$ with $i \delta \leq T$.
Set $v^{-1}:= v_0$, and suppose that $v^{i-1} \in \Ee$ is a critical point of 
$v \mapsto J + \alpha \psi (v - v^{i-2})$ on the affine space 
$\mathbf{A} (f ((i-1)\delta))$. 
If property (J2) is satisfied, we define the sequence 
$(v^i_{j})$, ${j \in \mathbb{N}_0}$, by setting $v^i_0 := v^{i-1}$ and
\begin{equation} \label{algobis}
v^i_j := \argmin_{ Av = f (i \delta)} \{ J (v)  + \eta |v - v^i_{j-1} |^2_{\mathcal{E}} 
+ \alpha \psi (v - v^i_0) \, : \,  v \in \mathcal E \}
\qquad \text{ for every } j \in \mathbb{N}
\end{equation}
with $\eta >0$ chosen such that (J2) holds.

\begin{remark}
Note that (J2), ($\Psi 2$) and ($\Psi 3$) guarantee that mimimizers in \eqref{algobis} are unique.
\end{remark}

The following lemma gives some properties of the sequence $( v_j^i )$, $j \in \mathbb{N}_0$. 
\begin{lemma} \label{lemmajNew}
Let \eqref{gamma}, (J1), (J2), and (J3) be satisfied, 
let $f \in W^{1,2} ( [0,T] ; \mathcal F)$, $\alpha \in \{ 0, 1 \}$,
and let $v_0 \in \Ee$ be a critical point of $v \mapsto J + \alpha \psi (v - v_0)$ 
in the affine space $\mathbf{A} (f (0))$.
If $\alpha = 1$, suppose in addition that ($\Psi 1$), ($\Psi 2$) and ($\Psi 3$) hold true.
Let $\delta \in (0,1)$ and let $i \in \mathbb{N}$ with $i \delta \leq T$.
Set $v^{-1}:= v_0$, and suppose that $v^{i-1}$ is a critical point of 
$v \mapsto J (v) + \alpha \psi (v - v^{i-2})$ on the affine space 
$\mathbf{A} (f ((i-1)\delta))$, and let $\left( v^i_{j} \right)_{j \in \mathbb{N}_0}$ be as in \eqref{algobis}.
Then:

\begin{itemize}

\vspace{.2cm}

\item[(i)] $\left(J (v^i_{j}) + \alpha \psi (v^i_{j} - v^i_0) \right)_{j \geq 2}$ is a nonincreasing converging sequence;

\vspace{.2cm}

\item[(ii)] $( v^i_{j})_{j \in \mathbb{N}_0}$ is bounded and 
\begin{equation}\label{convergenzabis}
\lim_{j \to +\infty}|v^i_j - v^i_{j-1}|_{\mathcal{E}}=0;
\end{equation}

\item[(iii)] any limit point of $\left( v^i_{j} \right)_{j \in \mathbb{N}_0}$ is a critical point of the functional 
$v \mapsto J (v)+ \alpha \psi (v - v^{i-1})$ 
on the affine space $\mathbf{A} (f (i \delta))$.

\end{itemize}
%
\end{lemma}

\begin{proof}
For every $j \geq 2$ we have $A v^i_{j} = A v^i_{j-1} = f(i \delta)$,
and therefore $v^i_{j-1}$ is a competitor for the minimum problem in \eqref{algobis}. 
Thus, 
\begin{equation} \label{jbis}
J (v^i_{j}) + \alpha \psi (v^i_{j} - v^i_0) 
\leq J (v^i_{j-1}) + \alpha \psi (v^i_{j-1} - v^i_0)
- \eta | v^i_j - v^i_{j-1} |^2_{\mathcal{E}}, \qquad \text{ for every }j \geq 2. 
\end{equation}
In particular, the sequence $\left( J (v^i_j) + \alpha \psi (v^i_{j} - v^i_0) \right)_{j \geq 2}$ is nonincreasing. 
Since $J + \alpha \psi \geq 0$, the limit
\begin{equation} \label{starbis}
\lim_{j \to \infty} \left( J (v^i_j) + \alpha \psi (v^i_{j} - v^i_0) \right) = : C \geq 0.
\end{equation}
exists and it is nonnegative, eventually showing (i).
Let now $M \in \mathbb{N}$ with $M > 2$. Summing up relation \eqref{jbis} for $j = 2, \ldots, M$ we obtain
$$
\sum_{j =2}^{M} | v^i_j - v^i_{j-1} |^2_{\mathcal{E}} \leq 
\frac{1}{\eta} (J(v^i_1) + \alpha \psi (v^i_{1} - v^i_0)
- J(v^i_M) - \alpha \psi (v^i_{M} - v^i_0)).
$$
Sending $M \to \infty$ we then have 
$$
\sum_{j =2}^{\infty} | v^i_j - v^i_{j-1} |^2_{\mathcal{E}} \leq \frac{1}{\eta} (J(v^i_1) + \alpha \psi (v^i_{1} - v^i_0) - C ) < \infty.
$$
In particular, this shows that \eqref{convergenzabis} holds true.
Note now that, by \eqref{starbis}, $\left( J (v^i_j) \right)_{j \in \mathbb{N}_0}$ is bounded.
Therefore, since $|Av^i_j|^2_{\mathcal{F}} = |f (i \delta) |^2_{\mathcal{F}}$
for every $j \geq 1$, the sequence 
$\left( J (v^i_j)+|Av^i_j|^2_{\mathcal{F}} \right)_{j \in \mathbb{N}_0}$ is bounded.
By (J1), we have that $\left( v^i_j \right)_{j \in \mathbb{N}_0}$ is also bounded, 
and this concludes the proof of (ii).

Let $v^i \in \mathcal{E}$ be a limit point of $\left( v^i_j \right)_{j \in \mathbb{N}_0}$.
Up to subsequences, we can assume that 
$$
\lim_{j \to \infty} v^i_{j} = v^i, \qquad \text{ in } \mathcal{E}.
$$
First of all, note that $Av^i= f(i \delta)$. 
By \eqref{algobis}, for every $j \in \mathbb{N}_0$ there exists $q^i_{j} \in \mathcal{F}'$ such that
$$
A^*q^i_{j} \in \partial J(v^i_{j}) + 2\eta(v^i_{j}-v^i_{j -1}) + \partial (\alpha \psi) (v^i_{j} - v^i_0), 
$$
where we used Remark \ref{sumsubdiff}. 
The previous relation can also be written as
\begin{equation} \label{ewqqbis}
A^*q^i_{j} = \xi^i_{j}+ 2\eta(v^i_{j}-v^i_{j -1}) + \zeta^i_{j}, 
\end{equation}
for some $\xi^i_{j} \in \partial J(v^i_{j})$ and $\zeta^i_{j} \in \partial (\alpha \psi) (v^i_{j} - v^i_0)$.
Note that, since $\left( v^i_j \right)_{j \in \mathbb{N}_0}$ is bounded, by (J3)
we also have that $\left( \xi^i_j \right)_{j \in \mathbb{N}_0}$ is bounded.
From \eqref{sublimitato}, $\left( \zeta^i_j \right)_{j \in \mathbb{N}_0}$ is bounded.
Thanks to \eqref{convergenzabis} and \eqref{gamma}, this implies that 
$\left(q^i_{j} \right)_{j \in \mathbb{N}_0}$ is also bounded.
Thus, up to subsequences, we can assume that
$$
\begin{cases}
\lim_{j \to \infty} \xi^i_{j} = \xi^i \\
\lim_{j \to \infty} \zeta^i_{j} = \zeta^i
\end{cases}
\quad \text{ in } \mathcal{E}' \qquad 
\text{ and } \qquad 
\lim_{j \to \infty} q^i_{j} = q^i \quad \text{ in } \mathcal{F}', 
$$
for some $\xi^i, \zeta^i \in \mathcal{E}'$ and $q^i \in \mathcal{F}'$.
Passing to the limit in \eqref{ewqqbis}, thanks to \eqref{convergenzabis} we conclude that 
$$
A^* q^i = \xi^i + \zeta^i.
$$
By the closure property of subdifferentials we have $\xi^i  \in \partial J(v^i)$ 
and $\zeta^i \in \partial (\alpha \psi) (v^i - v^i_0)$, 
and thus 
$$
\left( \partial J(v^i) + \partial (\alpha \psi) (v^i - v^i_0) \right) \cap \rg(A^*)\neq~\emptyset.
$$ 
\end{proof}

\begin{remark}
Suppose that $v^i$ and $z^i$ are two limit points 
of the sequence $\left( v^i_j \right)_{j \in \mathbb{N}_0}$.
By property (i) in  the previous lemma and the continuity of $J$, even if $v^i \neq z^i$ we have  
$$
J (v^i) + \alpha \psi (v^i - v^i_0) =  J(z^i) + \alpha \psi (z^i - v^i_0).
$$
\end{remark}
We state now a direct consequence of the previous lemma.
\begin{corollary}\label{cor}
Let \eqref{gamma}, (J1), (J2), and (J3) be satisfied, 
and let $f \in W^{1,2} ( [0,T] ; \mathcal F)$ and $\alpha \in \{ 0, 1 \}$.
If $\alpha = 1$, suppose in addition that ($\Psi 1$), ($\Psi 2$) and ($\Psi 3$) hold true.
Let $\delta \in (0,1)$ and let $v_0 \in \Ee$ be a critical point of $v \mapsto J + \alpha \psi (v - v_0)$ 
in the affine space $\mathbf{A} (f (0))$.
Set $v^0 := v_0$ and, for every $i \in \mathbb{N}$ with $i \delta \leq T$, let 
$\left( v^i_{j} \right)_{j \in \mathbb{N}_0}$ be defined by \eqref{algobis}, 
and let $v^i$ be a limit point of $\left( v^i_{j} \right)_{j \in \mathbb{N}_0}$.
Then, the function $v_{\delta} : [-\delta,T] \to \mathcal{E}$ defined as
\begin{equation}\label{approx}
v_\delta(t) : =v^i \quad \text{ for every }t \in [-\delta,T] \cap [i \delta, (i + 1) \delta), 
\quad \text{ for every } i \in \{ -1\} \cup \mathbb{N}_0 \text{ with } i \delta \leq T,
\end{equation}
is a discrete quasistatic evolution with time step $\delta$, 
initial condition $v_0$, and constraint $f$.
\end{corollary}
We now prove a uniform bound for the discrete quasistatic evolution defined above.  In the statement we write that the constant $Z_1$ depends also on $\alpha c$, with $c$ as in Remark \ref{limir}. With this we mean that this additional dependence only occurs in the dissipative case $\alpha=1$.

\begin{proposition} \label{z1}
Let \eqref{gamma}, (J1), (J2), and (J3) be satisfied, 
and let $f \in W^{1,2} ( [0,T] ; \mathcal F)$ and $\alpha \in \{ 0, 1 \}$.
If $\alpha = 1$, suppose in addition that ($\Psi 1$), ($\Psi 2$) and ($\Psi 3$) hold true.
Let $\delta \in (0,1)$ and let $v_0 \in \Ee$ be a critical point of $v \mapsto J + \alpha \psi (v - v_0)$ 
in the affine space $\mathbf{A} (f (0))$.
Let $v_{\delta} : [-\delta,T] \to \mathcal{E}$ be defined as in \eqref{approx}.
Then, there exists a positive constant 
$Z_1 = Z_1 (J, v_0, \gamma,  L, \alpha c, \eta, \| \dot{f} \|_{L^1([0,T];\mathcal{F})} , \| f \|_{L^{\infty}([0,T];\mathcal{F})})$, with $c$ as in Remark \ref{limir},
such that
$$
\sup_{\substack{ \delta \in (0,1) \vspace{.1cm} \\ t \in [0,T]} } | v_{\delta} (t) |_{\mathcal{E}} \leq Z_1.
$$
\end{proposition}

\begin{proof}
Let $i \in \mathbb{N}$ with $i \delta \leq T$ be fixed, 
and let $\left( v^i_j \right)_{j \in \mathbb N_0}$ be the sequence defined by \eqref{algobis}.
By property (J2) and Remark \ref{vbar}, the functional 
$J_{\eta, v^i_0} (\cdot) + \alpha \psi (\cdot - v^i_0)$
is strongly convex. Therefore, whenever 
$\xi \in \partial J_{\eta, v^i_0} (v^i_1) + \partial (\alpha \psi) (v^i_1 - v^i_0)$, 
we have  
$$
J_{\eta, v^i_0} (v) + \alpha \psi (v - v^i_0) 
\geq J_{\eta, v^i_0} (v^i_1) + \alpha \psi (v^i_1 - v^i_0)
+ \langle \xi , v - v^i_1 \rangle_{\mathcal{E}} 
\qquad \text{ for every } v \in \mathcal{E}.
$$
In particular, choosing 
$v = v^i_0 = v^{i-1}$ and recalling the definition of $J_{\eta, v^i_0}$ we have 
\begin{equation} \label{step1bis}
J (v^{i-1}) \geq J (v^i_1) 
+ \eta |v^i_1 - v^i_0|_{\mathcal{E}}^2
+ \alpha \psi (v^i_1 - v^i_0)
+ \langle \xi , v^i_0 - v^i_1 \rangle_{ \mathcal{E}}\,.
\end{equation}

By \eqref{algobis}, $v^i_1$ is the global minimizer of 
$J_{\eta, v^i_0} (\cdot) + \alpha \psi (\cdot - v^i_0)$ on $\mathbf{A} (f (i \delta))$.
Therefore, there exists $r^i \in \mathcal F'$ 
such that $A^* r^i \in \partial J_{\eta, v^i_0}(v^i_1) 
+ \partial (\alpha \psi) (v^i_1 - v^i_0)$ so that, by \eqref{step1bis},
\begin{equation*} 
J (v^{i-1}) \geq J_{\eta, v^i_0}(v^i_1)
+ \alpha \psi (v^i_1 - v^i_0)
+ \langle A^* r^i , v^i_0 - v^i_1 \rangle_{\mathcal{E}},
\end{equation*}
which gives
\begin{equation} \label{trg}
J_{\eta, v^i_0}(v^i_1) + \alpha \psi (v^i_1 - v^i_0) \leq J (v^{i-1}) + \langle A^* r^i , v^i_1 - v^i_0 \rangle_{\mathcal{E}}.
\end{equation}

\vspace{.2cm}

\noindent
\textbf{Step 1.} We show that there exist positive constants
$L' = L'(L, \alpha c, \eta)$ and $\overline{\delta} = \overline{\delta} 
(L, \alpha c, \eta, \gamma, \dot{f})$ with the following property:
for every $\delta \in (0, \overline{\delta})$
and $i \in \mathbb{N}$ with $i \delta \leq T$ we have 
\begin{equation}  \label{intermediate}
\left( 1- \frac{L'}{\gamma}  \int_{(i-1) \delta}^{i \delta} | \dot{f} (s) |_{ \mathcal{F}} \, \mathrm{d}s \right) J (v^i) 
\leq J (v^{i-1}) + \frac{L'}{\gamma} \int_{(i-1) \delta}^{i \delta} | \dot{f} (s) |_{ \mathcal{F}} \, \mathrm{d}s. 
\end{equation}
We start observing that
$$
A^* r^i \in \partial J_{\eta, v^i_0}(v^i_1) 
+ \partial (\alpha \psi) (v^i_1 - v^i_0)
\subset \partial J_{\eta, v^i_0}(v^i_1) + \alpha K^*.
$$
Therefore, thanks to condition (J3) we have 
$$
| A^* r^i |_{\mathcal{E}'} \leq L' \left(J_{\eta, v^i_0}(v^i_1) + 1\right), 
$$
for some positive constant $L'= L' (L, \alpha c, \eta)$, 
where $c$ is given by Remark~\ref{limir}.
Thus, from \eqref{gamma}
$$
| r^i |_{\mathcal{F}'} \leq \frac{L'}{\gamma} (J_{\eta, v^i_0}(v^i_1) + 1)\,.
$$
We also have
\begin{align*}
&\left| \langle A^* r^i ,  v^i_1 - v^i_0 \rangle_{ \mathcal{E}} \right|
= \left| \langle r^i , A v^i_1 - A v^i_0 \rangle_{\mathcal{F}} \right| 
\leq  \int_{(i-1) \delta}^{i \delta}   | r^i |_{ \mathcal{F}'} | \dot{f} (s) |_{ \mathcal{F}} \, \mathrm{d}s  \\
& \leq \left( \frac{L'}{\gamma} J_{\eta, v^i_0}(v^i_1) + \frac{L'}{\gamma} \right) 
\int_{(i-1) \delta}^{i \delta} | \dot{f} (s) |_{ \mathcal{F}} \, \mathrm{d}s \\
&\leq \left[ \frac{L'}{\gamma} \left( J_{\eta, v^i_0}(v^i_1) + \alpha \psi (v^i_1 - v^i_0) \right)
+ \frac{L'}{\gamma} \right]
\int_{(i-1) \delta}^{i \delta} | \dot{f} (s) |_{ \mathcal{F}} \, \mathrm{d}s.
\end{align*}
Using last inequality, \eqref{trg} gives
\begin{equation}  \label{asdf}
 \left( 1- \frac{L'}{\gamma}  \int_{(i-1) \delta}^{i \delta} | \dot{f} (s) |_{ \mathcal{F}} \, \mathrm{d}s \right)  
 \left( J_{\eta, v^i_0}(v^i_1) + \alpha \psi (v^i_1 - v^i_0) \right)
\leq J (v^{i-1}) + \frac{L'}{\gamma} \int_{(i-1) \delta}^{i \delta} | \dot{f} (s) |_{ \mathcal{F}} \, \mathrm{d}s. 
\end{equation}
By the absolutely continuity of the integral, there exists a positive 
constant $\overline{\delta} = \overline{\delta} 
(L, \alpha c, \eta, \gamma, \dot{f})$ such that
\begin{equation} \label{asbcont}
\mathcal{L}^1 (G) < \overline{\delta} \quad \Longrightarrow \quad 
\frac{L'}{\gamma} \int_{G} | \dot{f} (s) |_{ \mathcal{F}} \, \mathrm{d}s < \frac{1}{2}.
\end{equation}
Note now that, for every $j \geq 2$, by the minimality of $v^i_j$ we have
\begin{align*}
J (v^i_j) &\leq J (v^i_j)  + \eta |v^i_j - v^i_{j-1} |^2_{\mathcal{E}} 
+ \alpha \psi (v^i_j - v^i_0)
\leq J (v^i_{j-1}) + \alpha \psi (v^i_{j-1} - v^i_0) \\
&\leq  J (v^i_1) + \alpha \psi (v^i_1 - v^i_0)
\leq J_{\eta, v^i_0} (v^i_1) + \alpha \psi (v^i_1 - v^i_0), 
\end{align*}
where we also took into account that 
$\left( J (v^i_{j}) + \alpha \psi (v^i_{j} - v^i_0) \right)_{j \geq 2}$ is nonincreasing.
Passing to the limit when $j \to \infty$, up to subsequences, we obtain
$$
J (v^i) \leq J_{\eta, v^i_0} (v^i_1) + \alpha \psi (v^i_1 - v^i_0).
$$
Combining last relation with \eqref{asdf}, we get \eqref{intermediate}. 

\vspace{.2cm}

\noindent
\textbf{Step 2.} We conclude. We start by proving that
\begin{equation} \label{laksdj}
\sup_{\substack{ \delta \in (0,1) \vspace{.1cm} \\ t \in [0,T]} } | J (v_{\delta} (t))| \leq Z_2,
\end{equation}
for some positive constant $Z_2 =Z_2 (\gamma, L', v_0, \| \dot{f} \|_{L^1 ([0,T];\mathcal{F})})$.
Note that it is not restrictive to assume $\delta \in (0, \overline{\delta})$.
We now set, for every $i \in \mathbb{N}$ such that $i \delta \leq T$,
$$
a_{\delta, i} := J (v^i) \quad \text{ and }
\quad b_{\delta, i} := \frac{L'}{\gamma} \int_{(i-1) \delta}^{i \delta} | \dot{f} (s) |_{ \mathcal{F}} \, \mathrm{d}s\,.
$$
Notice that $ b_{\delta,i}<\frac12$ by \eqref{asbcont}, and that relation \eqref{intermediate} gives
$$
a_{\delta,i} \leq \frac{b_{\delta,i}}{1 - b_{\delta,i}} + \frac{a_{\delta,i-1}}{1 - b_{\delta,i}}.
$$
Iterating the previous inequality we obtain
\begin{align*}
a_{\delta, i} &\leq \frac{b_{\delta,i}}{1 - b_{\delta,i}} + \frac{a_{\delta,i-1}}{1 - b_{\delta,i}} 
\leq \frac{b_{\delta,i}}{1 - b_{\delta,i}} + \frac{1}{1 - b_{\delta,i}} 
\left[ \frac{b_{\delta,i-1}}{1 - b_{\delta,i-1}} + \frac{a_{\delta,i-2}}{1 - b_{\delta,i-1}} \right] \\
&= \frac{b_{\delta,i}}{1 - b_{\delta,i}} + \frac{b_{\delta,i-1}}{(1 - b_{\delta,i})(1 - b_{\delta,i-1})}
+ \frac{a_{\delta,i-2}}{(1 - b_{\delta,i})(1 - b_{\delta,i-1})} \\
&\leq \ldots \leq \frac{a_{\delta,0}}{(1 - b_{\delta,i}) (1 - b_{\delta,i-1}) \ldots(1 - b_{\delta,1})} 
+ \sum_{k=0}^{i-1} \frac{b_{\delta,i-k}}{(1 - b_{\delta,i}) (1 - b_{\delta,i-1}) \ldots(1 - b_{\delta,i-k})} \\
&\leq \frac{a_{\delta,0}}{(1 - b_{\delta,i}) (1 - b_{\delta,i-1}) \ldots(1 - b_{\delta,1})} 
+ \frac{1}{(1 - b_{\delta,i}) (1 - b_{\delta,i-1}) \ldots(1 - b_{\delta,1})}  
\sum_{k=0}^{i-1} b_{\delta,i-k} \\
&\leq \frac{1}{(1 - b_{\delta,i}) (1 - b_{\delta,i-1}) \ldots(1 - b_{\delta,1})}  
\left[ J (v_0) + \int_{0}^{T} | \dot{f} (s) |_{ \mathcal{F}} \, \mathrm{d}s \right].
\end{align*}

We therefore only need to find a bound for the quantity
$
 \frac{1}{(1 - b_{\delta,i}) (1 - b_{\delta,i-1}) \ldots(1 - b_{\delta,1})}.
$

\noindent
Since $ b_{\delta,i}<\frac12$, with the elementary inequality $0<-\ln(1-x)\le 2x$ for all $0\le x \le \frac12$ we eventually get
\begin{align*}
\ln \left[ \frac{1}{(1 - b_{\delta,i}) (1 - b_{\delta,i-1}) \ldots(1 - b_{\delta,1})} \right]
= - \sum_{l=1}^i  \ln (1 - b_{\delta, l} ) \leq 2\sum_{l=1}^i  b_{\delta, l} \le \frac{2L'}{\gamma}\int_{0}^{T} | \dot{f} (s) |_{ \mathcal{F}} \, \mathrm{d}s
\end{align*}
so that \eqref{laksdj} follows.
We now have 
$$
\sup_{\substack{ \delta \in (0,1) \vspace{.1cm} \\ t \in [0,T]} } 
\left( J (v_{\delta} (t))+|Av_{\delta} (t)|^2_{\mathcal{F}} \right)
\leq \sup_{\substack{ \delta \in (0,1) \vspace{.1cm} \\ t \in [0,T]} } 
\left( J (v_{\delta} (t))+ \| f \|^2_{L^{\infty} ((0,T) ; \mathcal{F} )} \right) < Z_3,
$$
for some positive constant 
$Z_3 = Z_3 (\gamma, L', \eta, \| \dot{f} \|_{L^2 ((0,T); \mathcal{F})}, 
\| f \|_{L^{\infty} ((0,T) ; \mathcal{F} )})$.
Taking into account (J1), last inequality implies that 
$$
| v_{\delta} (t) |_{\mathcal{E}} \leq Z_1, \qquad \text{ for every } \delta \in (0,1) \text{ and } t \in [0,T],
$$
with a constant $Z_1$ that also depends on the coercivity of the function 
$v \mapsto J (v )+|Av|^2_{\mathcal{E}}$. 
\end{proof}

\section{Proof of Theorem \ref{psiTeorexist}}\label{proofmain1}
We start by stating a lemma that will be used later.
\begin{lemma} \label{lemmadelta}
Let $X$ be a Banach space, let $T > 0$, and let $v \in BV ([0,T]; X)$.
Let $\left( v_{\delta_{k}} \right)_{k \in \mathbb{N}} \subset L^{\infty} ([0,T]; X) \cap BV ([0,T]; X)$
have equibounded variation. Then it exists an at most countable set $N\subset [0,T]$ and a subsequence $\delta_{{k}_j}$, independent of $t$, with
$$
\lim_{j \to \infty} \|v_{\delta_{{k}_j}} (t-\delta_{{k}_j})-v_{\delta_{{k}_j}} (t) \|_X=0\quad \text{ for every }t \in (0,T] \setminus N.
$$
\end{lemma}

\begin{proof}
We set $V_k(t)=\mathrm{Var}(v_{\delta_{k_j}}; [0,t])$ and we observe that, by the assumption, $V_{k}$ is an equibounded sequence of monotone nondecreasing functions. By Helly's Theorem it exists a subsequence $\delta_{{k}_j}$, independent of $t$, and a monotone nondecreasing function $V$ such that
$
V_{{k}_j}(t)\to V(t)
$
for every $t\in [0,T]$. Let now $N$ be the (at most countable) set of discontinuity points of $V$. For $t\in (0,T] \setminus N$ and an arbitrary $\varepsilon >0$, let $\delta_0>0$ be such that
\begin{equation}\label{delta0}
V(t)-V(t-\delta_0)\le \varepsilon\,.
\end{equation}
Using the definition of pointwise variation, \eqref{additivity} (with $\psi=\|\cdot\|_X$), the monotonicity of $V_k$ and \eqref{delta0} we get
\begin{eqnarray*}
&\displaystyle
\lim\sup_{j \to \infty} \|v_{\delta_{{k}_j}} (t-\delta_{{k}_j})-v_{\delta_{{k}_j}} (t) \|_X \le \lim\sup_{j \to \infty}\left(V_{k_j}(t)- V_{{k}_j}(t-\delta_{{k}_j})\right)\\
&\displaystyle
\le \lim\sup_{j \to \infty}\left(V_{k_j}(t)- V_{{k}_j}(t-\delta_0)\right)=V(t)-V(t-\delta_0)\le \varepsilon\,.
\end{eqnarray*}
This proves the statement by the arbitrariness of $\varepsilon$.
\end{proof}

The following a-priori estimates will be needed to prove the Theorem.

\begin{theorem} \label{erfd}
Let \eqref{gamma}, (J1), (J2), (J3), ($\Psi 1$), ($\Psi 2$) and ($\Psi 3$) be satisfied, 
and let $f \in W^{1,2} ( [0,T] ; \mathcal F)$. 
Let $\delta \in (0,1)$ and let $v_0 \in \Ee$ be a critical point of $v \mapsto J +  \psi (v - v_0)$ 
in the affine space $\mathbf{A} (f (0))$.
Let $v_{\delta} : [-\delta,T] \to \mathcal{E}$ be defined as in \eqref{approx} with $\alpha = 1$.
Then, there exist $q_{\delta} \in L^{\infty} ([0,T] ; \mathcal{F}')$, 
$\overline{v}_{\delta} : [0,T] \to \mathcal{E}$, two positive constants $Z_1, Z_5$, 
and $r_1 : (0, 1) \to [0, \infty)$ and $r_2 : (0, 1) \to [0, \infty)$ with
$$
\lim_{\delta \to 0^+} r_1 (\delta) = \lim_{\delta \to 0^+} r_2 (\delta)  = 0, 
$$
such that
\begin{itemize}

\item[(i)] $|q_\delta (t)|_{\mathcal{F}'} \le Z_4$ and $|v_\delta (t)|_{\mathcal{E}} \leq Z_1$ for every $t \in [0,T]$;

\vspace{.2cm}

\item[(ii)] $\textnormal{dist} \left( A^* q_{\delta} (t) - \partial J (\overline{v}_{\delta} (t)) , K^*  \right) \leq 2\eta r_1(\delta)$, with $\eta>0$ as in (J2);

\vspace{.2cm}

\item[(iii)] $
\sup_{t \in [0,T]} | v_{\delta} (t-\delta) - \overline{v}_{\delta} (t) |_{\mathcal{E}} 
\leq r_1 (\delta)$;

\vspace{.2cm}

\item[(iv)] For every $0 \leq t_1 < t_2 \leq T$
$$
J ( v_{\delta} (t_2) ) + \mathrm{Var}_{\psi} (v_{\delta}; [t_1, t_2]) 
\leq J ( v_{\delta} (t_1) ) +\int_{t_1}^{t_2} \langle q_\delta (s) ,  \dot{f} (s) \rangle_{\mathcal{F}} \, \mathrm{d}s + r_2 (\delta).
$$

\end{itemize}

\end{theorem}

\begin{remark} \label{bounds constants}
The constants $Z_1$ is given in Proposition \ref{z1}, while $Z_4$ depends on the following quantities
$$
Z_4 = Z_4 (L, c, \eta, \gamma, v_0, \| \dot{f}\|_{L^1 ([0,T];\mathcal{F})}).
$$
\end{remark}

\begin{proof}
First of all, note that inequality $|v_\delta (t)|_{\mathcal{E}} \leq Z_1$ was proven in Proposition \ref{z1}.
Let now $\eta$ be given by (J2), let $\nu$ be the constant of strong convexity of 
$v \mapsto J (v) + \eta | v |^2_{\mathcal{E}}$, and let 
$\overline{\delta}$ 
be given by \eqref{asbcont}.

\vspace{.2cm}

\noindent
\textbf{Step 1.} We show that there exists a positive constant 
$Z_4= Z_4 (L, c, \gamma, v_0, \| \dot{f}\|_{L^1 ([0,T];\mathcal{F})})$ 
with the following property.
If $r^i_0, r^i_1 \in \mathcal{F}'$ are such that
$$
A^* r^i_0 \in \partial J (v^{i}_0) + K^*
\quad \text{ and } \quad 
A^* r^i_1 \in \partial J_{\eta, v^{i}_0} (v^{i}_1) +  \partial \psi (v^{i}_1 - v^{i}_0),
$$
then we have 
$$
| r^i_0 |_{\mathcal{F}'} + | r^i_1 |_{\mathcal{F}'} \leq Z_4.
$$
By \eqref{asdf} and \eqref{laksdj} we have
$$
J_{\eta, v^i_0} (v^i_1) + \psi (v^i_1-v^i_0) \leq Z_2 + \frac{L'}{\gamma} \| \dot{f}\|_{L^1 ([0,T];\mathcal{F})}
 \quad \text{ for every } \delta \in (0, \overline{\delta})
\text{ and } i \in \mathbb{N} \text{ with } i \delta \leq T,
$$
where $Z_2$ does not depend on $\delta$ and $i$.
Then, by (J3) and \eqref{gamma} the estimate $| r^i_1 |_{\mathcal{F}'} \leq Z_4$ follows. Since by \eqref{laksdj} it also holds $J (v^{i}_0)\le Z_2$, while $K^*\subset B(0,c)$ by Remark \ref{limir}, again by (J3) and \eqref{gamma} the claim is proved.

\vspace{.2cm}

\noindent
\textbf{Step 2.} We prove (iii). 
First of all, we define $\overline{v}_{\delta} (t) = v_0$
for every $t \in [0, \delta)$, and 
$$
\overline{v}_{\delta} (t) := v^{i}_1 \quad \text{ for } i \delta \leq t < (i+1) \delta,
\quad i \in \mathbb{N} \text{ with } i \delta \leq T.
$$
Let now $i \in \mathbb{N}$ with $i \delta \leq T$.
We need to show that there exists $r_1: [0, \overline{\delta}) \to [0, \infty)$ with
$r_1 (\delta) \to 0$ as $\delta \to 0^+$ such that
 $$
| v^{i}_1 - v^{i}_0 |_{\mathcal{E}} \leq r_1 (\delta).
$$
By construction, we have that $v^{i}_0 = v^{i-1}$ is a critical point of the functional
$v \mapsto J (v) +  \psi (v - v^{i-1}_0)$ on the affine space $\mathbf{A} (f ((i-1)\delta))$.
Therefore, there exists $r^i_0$ such that
\begin{align*}
A^* r^i_0 \in \partial J (v^{i}_0) +   \partial \psi (v^{i}_0 - v^{i-1}_0)
\subset \partial J (v^{i}_0) +  K^*.
\end{align*}
Using the fact that $\partial J (v^{i}_0) = \partial J_{\eta, v^{i}_0} (v^{i}_0)$ we can also write
$$
A^* r^i_0 \in \partial J_{\eta, v^{i}_0} (v^{i}_0) +  \partial \psi (v^{i}_0 -v^{i}_0).
$$
Note now that, by construction, $v^i_1$ minimizes the functional
$v \mapsto J_{\eta, v^{i}_0} (v) +  \psi (v - v^{i}_0)$ on the affine space $\mathbf{A} (f (i \delta))$.
Therefore, there exists $r^i_1$ such that
\begin{equation} \label{r1i}
A^* r^i_1 \in \partial J_{\eta, v^{i}_0} (v^{i}_1) +  \partial \psi (v^{i}_1 - v^{i}_0).
\end{equation}
Now, $A^* r^i_0$ and $A^* r^i_1$ are both in the subdifferential of the strongly convex functional $J_{\eta, v^{i}_0} (v) +  \psi (v - v^{i}_0)$ at $v^{i}_0$. The constant $\nu$ of strong convexity is furthermore indepedendent of $v^{i}_0$ by Remark \ref{vbar}. With this and step 1 we have 
\begin{align*}
\nu | v^{i}_1 - v^{i}_0 |^2_{\mathcal{E}} 
&\leq \langle A^* r^i_1 - A^* r^i_0 , v^{i}_1 - v^{i}_0 \rangle_{\mathcal{E}}
= \langle r^i_1 - r^i_0 , A v^{i}_1 - A v^{i}_0 \rangle_{\mathcal{F}} \\
& = \langle r^i_1 - r^i_0 , f (i \delta) - f ((i-1) \delta) \rangle_{ \mathcal{F}}
\leq Z_4 | f (i \delta) - f ((i-1) \delta) |_{\mathcal{F}} \\
&\leq Z_4 \int_{(i-1) \delta}^{i \delta} | \dot{f} (s) |_{\mathcal{F}} \, \mathrm{d}s.
\end{align*}
Then,
$$
| \overline{v}_{\delta} (t) - v_{\delta} (t-\delta) |^2_{\mathcal{E}} 
= | v^{i}_1 - v^{i}_0 |^2_{\mathcal{E}} \leq \frac{Z_4}{\nu} 
\int_{(i-1) \delta}^{i \delta} | \dot{f} (s) |_{\mathcal{F}} \, \mathrm{d}s
\leq \frac{Z_4}{\nu} \sup_{i \delta \leq T} \int_{(i-1) \delta}^{i \delta} | \dot{f} (s) |_{\mathcal{F}} \, \mathrm{d}s. 
$$
Setting 
$$
r_1 (\delta)  := \sqrt{ \frac{Z_4}{\nu}} 
\left( \sup_{i \delta \leq T} \int_{(i-1) \delta}^{i \delta} | \dot{f} (s) |_{\mathcal{F}} \, \mathrm{d}s \right)^{1/2},
$$
by the absolute continuity of the integral we have that $r_1 (\delta) \to 0$ as $\delta \to 0^+$ and we conclude.

\vspace{.2cm}

\noindent
\textbf{Step 3.} 
We prove (i) and (ii).
We start by defining the function 
$q_{\delta}: [0, T] \to \mathcal F'$.
Since $v_0$ is a critical point of $v \mapsto J + \psi (v - v_0)$
in the affine space $\mathbf{A} (f (0))$, there exists $q_0 \in \mathcal F'$
such that 
$$
A^* q_0 \in \partial J (v_0) + K^*.
$$
We set $q_{\delta} (t) = q_0$ for $0 \leq t < \delta$. Moreover, we set
$$
q_{\delta} (t) := r^i_1  \quad \text{ for } i \delta \leq t < (i+1) \delta,
\quad i \in \mathbb{N} \text{ with } i \delta \leq T.
$$
Then, by construction and step 1 we have $| q_{\delta} (i \delta) |_{\mathcal{F}} \leq Z_4$, 
so that (i) is satisfied. Moreover, 
\begin{align*}
A^* q_{\delta} (i \delta) &\in \partial J_{\eta, v^{i}_0} (v^{i}_1) +  \partial \psi (v^{i}_1 - v^{i}_0) \\
&= \partial J (v^{i}_1) + 2 \eta (v^{i}_1-v^{i}_0) +  \partial \psi (v^{i}_1 - v^{i}_0) \\
&\subset \partial J (v^{i}_1) + 2 \eta (v^{i}_1-v^{i}_0) +  K^*,
\end{align*}
so that
$$
\textnormal{dist} \left( A^* q_{\delta} (t) - \partial J (\overline{v}_{\delta} (t)) , K^*  \right) 
\leq 2 \eta | v^{i}_1-v^{i}_0 |_{\mathcal{E}} \leq 2 \eta r_1 (\delta).
$$

\vspace{.2cm}

\noindent
\textbf{Step 4.} We show (iv). Let $0 \leq t_1 < t_2 \leq T$.
By \eqref{r1i} we have that for every $v \in \mathcal{E}$
$$
J_{\eta, v^{i}_0} (v) +  \psi (v - v^{i}_0) 
\geq J_{\eta, v^{i}_0} (v^{i}_1) +  \psi (v^{i}_1 - v^{i}_0) +
\langle A^* r^i_1 , v - v^{i}_1 \rangle_{\mathcal{E}},
$$
for every $i \delta \leq t < (i+1) \delta$, and $i \in \mathbb{N}$ with $i \delta \leq T$.
In particular, choosing $v = v^{i}_0$ we obtain
\begin{align*}
J (v^{i}_0)  
\geq J (v^{i}_1) +  \psi (v^{i}_1 - v^{i}_0) +
\langle A^* r^i_1 , v^{i}_0 - v^{i}_1 \rangle_{ \mathcal{E}},
\end{align*}
which gives
\begin{align*}
&J (v^{i}_1) + \psi (v^{i}_1 - v^{i}_0) 
\leq J (v^{i}_0)  + \langle A^* r^i_1 , v^{i}_1 - v^{i}_0 \rangle_{ \mathcal{E}} \\
&\hspace{.2cm}= J (v_{\delta} ((i-1)\delta ) ) + \int^{i \delta}_{(i-1)\delta} \langle q_{\delta} (s), \dot{f} (s) \rangle_{\mathcal{F}} \, \mathrm{d}s.
\end{align*}
Thanks to (i) in  Lemma~\ref{lemmajNew}, we have that for every $j \geq 2$:
\begin{align*}
J (v^{i}_j) + \psi (v^{i}_j - v^{i}_0) 
\leq J (v_{\delta} ((i-1)\delta ) ) + \int^{i \delta}_{(i-1)\delta} 
\langle q_{\delta} (s), \dot{f} (s) \rangle_{\mathcal{F}} \, \mathrm{d}s.
\end{align*}
Passing to the limit as $j \to \infty$ we then obtain
\begin{equation} \label{i}
J (v_{\delta} (i \delta)) +  \psi (v_{\delta} (i \delta) - v_{\delta} ((i-1)\delta )) 
\leq J (v_{\delta} ((i-1)\delta ) ) + \int^{i \delta}_{(i-1)\delta} 
\langle q_{\delta} (s), \dot{f} (s) \rangle_{\mathcal{F}} \, \mathrm{d}s
\end{equation}
Let now $l , k \in \mathbb{N}$ be such that
$l \delta \leq t_1 < (l+1) \delta$ and $k \delta \leq t_2 < (k+1) \delta$.
By summing up relation \eqref{i} for $i = l+1, \ldots, k$ we obtain
\begin{align*}
&J (v_{\delta} (k \delta)) +  \sum_{i = l}^k \psi (v_{\delta} (i \delta) - v_{\delta} ((i-1)\delta )) 
\leq J (v_{\delta} (l \delta ) ) + \int^{k \delta}_{l \delta} 
\langle q_{\delta} (s), \dot{f} (s) \rangle_{\mathcal{F}} \, \mathrm{d}s \\
&\hspace{.2cm}= J (v_{\delta} (l \delta ) ) 
+ \int^{t_2}_{t_1} \langle q_{\delta} (s), \dot{f} (s) \rangle_{ \mathcal{F}} \, \mathrm{d}s
+ \int^{t_1}_{l \delta} \langle q_{\delta} (s), \dot{f} (s) \rangle_{\mathcal{F}} \, \mathrm{d}s
- \int_{k \delta}^{t_2} \langle q_{\delta} (s), \dot{f} (s) \rangle_{\mathcal{F}} \, \mathrm{d}s \\
&\leq J (v_{\delta} (l \delta ) ) 
+ \int^{t_2}_{t_1} \langle q_{\delta} (s), \dot{f} (s) \rangle_{ \mathcal{F}} \, \mathrm{d}s 
+ 2Z_4\sup_{i \delta \leq T} \int_{(i-1) \delta}^{i \delta} | \dot{f} (s) |_{\mathcal{F}} \, \mathrm{d}s.
\end{align*}
Since $v_{\delta}$ is piecewise constant, we have
$$
\sum_{i = l}^k \psi (v_{\delta} (i \delta) - v_{\delta} ((i-1)\delta ))  = \mathrm{Var}_{\psi} (v_{\delta}; [t_1, t_2])\,.
$$
With this, and since $v_{\delta} (k \delta)=v_\delta(t_2)$ and $v_{\delta} (l \delta)=v_\delta(t_1)$ by construction, (iv) follows by setting
\[
r_2(\delta):=2Z_4\sup_{i \delta \leq T} \int_{(i-1) \delta}^{i \delta} | \dot{f} (s) |_{\mathcal{F}} \, \mathrm{d}s\,.
\]

\end{proof}
We can finally give the proof of Theorem \ref{psiTeorexist}.

\begin{proof}[Proof of Theorem \ref{psiTeorexist}]
We divide the proof into several steps.

\vspace{.2cm}

\noindent
\textbf{Step 1.} We prove that there exists a sequence $\left( \delta_k \right)_{k \in \mathbb{N}}$,
and functions $v \in L^1 ( [0,T] ; \mathcal{E})$, and $q \in L^{\infty} ( [0,T] ; \mathcal{F}')$,
such that
$$
v_{\delta_k} \to v \text{ in } L^1 ( [0,T] ; \mathcal{E})
\qquad \text{ and } \qquad
q_{\delta_k} \stackrel{*}{\rightharpoonup} q \quad \text{ weakly* in } L^{\infty} ( [0,T] ; \mathcal{F}').
$$

From properties (i) and (iv) of Theorem~\ref{erfd} we have 
$$
\| v_{\delta} \|_{L^1 ( [0,T] ; \mathcal{E})} + \mathrm{Var}_{\psi} (v_{\delta}; [t_1, t_2]) \leq C,
$$
for some constant $C$ that is independent on $\delta$. 
By Helly's theorem (and since $\mathcal E$ has finite dimension), there exists a subsequence
$\left( v_{\delta_k} \right)_{k \in \mathbb{N}}$ and a function $v \in BV ([0,T];\mathcal{E})$ such that
$$
v_{\delta_k} \to v \text{ in } L^1 ( [0,T] ; \mathcal{E})
$$
and, in addition, 
$$
\lim_{k \to \infty} v_{\delta_k} (t) \to v (t) \quad \text{in $\mathcal{E}$ for every } t \in [0,T]
$$
Up to extracting a further subsequence we can assume that Lemma \ref{lemmadelta} holds. Furthermore, since by (i) in  Theorem~\ref{erfd} we also have 
$$
\| q_{\delta} \|_{L^{\infty} ( [0,T] ; \mathcal{F}')}  \leq C,
$$
without any loss of generality we can assume that for the same subsequence 
$\left( \delta_k\right)_{k \in \mathbb{N}}$ we have 
$$
q_{\delta_k} \stackrel{*}{\rightharpoonup} q \quad
\text{ weakly* in } L^{\infty} ( [0,T] ; \mathcal{F}'),
$$
for some function $q \in L^{\infty} ( [0,T] ; \mathcal{F}')$.

\vspace{.2cm}

\noindent
\textbf{Step 2.} We are going to show that
$$
\overline{v}_{\delta_k} \to v \quad \text{ strongly in } L^1 ((0,T); \mathcal{E}).
$$
Indeed, we have 
\begin{align*}
\int_0^T | \overline{v}_{\delta_k} (s) - v (s) |_{\mathcal{E}} \, \mathrm{d}s
\leq \int_0^T | \overline{v}_{\delta_k} (s) - v_{\delta_k} (s-\delta_k) |_{\mathcal{E}} \, \mathrm{d}s
+ \int_0^T | v_{\delta_k} (s-\delta_k) - v (s) |_{\mathcal{E}} \, \mathrm{d}s.
\end{align*}
By (iii) in  Theorem~\ref{erfd} we have 
$$
\lim_{k \to \infty} \int_0^T | \overline{v}_{\delta_k} (s) - v_{\delta_k} (s-\delta_k) |_{\mathcal{E}} \, \mathrm{d}s = 0.
$$
Concerning the second term, by Lemma~\ref{lemmadelta}, the convergence of $v_{\delta_k}$ to $v$ in $L^1$, and dominated convergence we have 
$$
\lim_{k \to \infty} \int_0^T | v_{\delta_k} (s-\delta_k) - v (s) |_{\mathcal{E}} \, \mathrm{d}s = 0.
$$
By (iii) in  Theorem~\ref{erfd}, Lemma~\ref{lemmadelta}, and the pointwise convergence of $v_{\delta_k}$ to $v$, we also have 
$$
\lim_{k \to \infty} \overline{v}_{\delta_k} (t) \to v (t) \quad \text{in $\mathcal{E}$ for every } t \in (0,T] \setminus N\,, 
$$
where $N$ is at most countable.
By (ii) in  Theorem~\ref{erfd}, for every $t \in [0,T]$ 
there exists $\xi_{\delta} (t) \in \partial J (\overline{v}_{\delta_k} (t))$ such that 
$$
\textnormal{dist} \left( A^* q_{\delta} (t) - \xi_{\delta} (t) , K^*  \right) \leq 2\eta r_1 (\delta).
$$
In particular, since $\left( q_{\delta_k} \right)_{k \in \mathbb{N}}$ is bounded in $L^{\infty} ( [0,T] ; \mathcal{F}')$, 
this implies that $\left( \xi_{\delta_k} \right)_{k \in \mathbb{N}}$ is bounded in $L^{\infty} ( [0,T] ; \mathcal{E}')$.
Without any loss of generality, we can assume that 
$$
\xi_{\delta_k} \stackrel{*}{\rightharpoonup} \xi \quad \text{ weakly* in } L^{\infty} ( [0,T] ; \mathcal{E}')
$$
for some $\xi \in L^{\infty} ( [0,T] ; \mathcal{E}')$.

\vspace{.2cm}

\noindent
\textbf{Step 3.}
We are now going to prove that
$$
\xi (t) \in \partial J (v (t)) \quad \text{ for $\mathcal{L}^1$-a.e. } t \in [0,T].
$$
First of all note that there exists a set $\Lambda \subset [0,T]$ with $\mathcal{L}^1 (\Lambda) = 0$, such that
for every $t \in [0,T] \setminus \Lambda$ the following properties are satisfied: 
$$
\begin{cases}
\lim_{h \to 0^+} \frac{1}{h} \int_t^{t + h} J (v (s)) \, \mathrm{d}s = J (v (t)) &\text{ in }\R, \\
\lim_{h \to 0^+} \frac{1}{h} \int_t^{t + h} \langle \xi (s), v (s) \rangle_{\mathcal{E}} \, \mathrm{d}s 
= \langle \xi (t), v (t) \rangle_{\mathcal{E}} &\text{ in }\R, \\
\lim_{h \to 0^+} \frac{1}{h} \int_t^{t + h} \xi (s) \, \mathrm{d}s = \xi (t) &\text{ in }\mathcal{E}', \\
\lim_{h \to 0^+} \frac{1}{h} \int_t^{t + h} | v (s) |^2_{\mathcal{E}} \, \mathrm{d}s = | v (t) |^2_{\mathcal{E}} &\text{ in }\R, \\
\lim_{h \to 0^+} \frac{1}{h} \int_t^{t + h} v (s) \, \mathrm{d}s = v (t) &\text{ in }\mathcal{E}.
\end{cases}
$$

Let now $s \in [0,T]$.
Since for every $k \in \mathbb{N}$ we have $\xi_{\delta_k} (s) \in \partial J (\overline{v}_{\delta_k} (s))
= \partial J_{\eta, \overline{v}_{\delta_k} (s)} (\overline{v}_{\delta_k} (s))$, 
and recalling that $J_{\eta, \overline{v}_{\delta_k} (s)}$ is convex, for every $w \in \mathcal{E}$ 
$$
J (w) \geq J (\overline{v}_{\delta_k} (s)) - \eta | w - \overline{v}_{\delta_k} (s) |_{\mathcal{E}}^2
+ \langle \xi_{\delta_k} (s) , w - \overline{v}_{\delta_k} (s) \rangle_{\mathcal{E}}.
$$
Let now $t \in [0,T] \setminus \Lambda$.
Averaging the previous inequality between $t$ and $t+h$ we obtain
$$
J (w) \geq \frac{1}{h} \int_t^{t + h} \left[ J (\overline{v}_{\delta_k} (s)) - \eta | w - \overline{v}_{\delta_k} (s) |_{\mathcal{E}}^2
+ \langle \xi_{\delta_k} (s) , w - \overline{v}_{\delta_k} (s) \rangle_{\mathcal{E}} \right] \, \mathrm{d}s.
$$
Passing to the limit as $k \to \infty$ we obtain (recalling that $\overline{v}_{\delta_k} \to v$ strongly in $L^1$)
\begin{align*}
J (w) &\geq \lim_{k \to \infty} \frac{1}{h} \int_t^{t + h} \left[ J (\overline{v}_{\delta_k} (s)) - \eta | w - \overline{v}_{\delta_k} (s) |_{\mathcal{E}}^2
+ \langle \xi_{\delta_k} (s) , w - \overline{v}_{\delta_k} (s) \rangle_{\mathcal{E}} \right] \, \mathrm{d}s \\
&= \frac{1}{h} \int_t^{t + h} 
\left[ J (v (s)) + \langle \xi (s) , w - v (s) \rangle_{\mathcal{E}} \right] \, \mathrm{d}s
- \frac{1}{h} \int_t^{t + h} \eta | w - v (s) |_{\mathcal{E}}^2 \, \mathrm{d}s.
\end{align*}
Then, passing to the limit as $h \to 0^+$ we obtain
$$
J (w) + \eta | w - v (t) |_{\mathcal{E}}^2 \geq J (v (t)) + \langle \xi (t) , w - v (t) \rangle_{\mathcal{E}} 
\qquad \forall \, w \in \mathcal{E}.
$$
The inequality above shows that $\xi (t) \in \partial J_{\eta, v (t)} (v (t)) = \partial J (v (t))$.

\vspace{.2cm}

\noindent
\textbf{Step 4.} We prove that
$$
A^* q (t) \in \partial J ( v (t)) + K^* \qquad \text{ for $\mathcal{L}^1$-a.e. } t \in [0,T].
$$
Using (ii) in  Theorem~\ref{erfd} we have 
\begin{align*}
&\int_0^T \textnormal{dist} \left( A^* q (t) - \partial J ( v (t)) , K^*  \right)  \, \mathrm{d}s \\
&\leq \liminf_{k \to \infty} 
\int_0^T \textnormal{dist} \left( A^* q_{\delta_k} (t) - \partial J (\overline{v}_{\delta_k} (t)) , K^*  \right) \, \mathrm{d}s \\
& \leq 2\eta T \liminf_{k \to \infty} r_1 (\delta_k) = 0,
\end{align*}
thus giving that
$$
A^* q (t) \in \partial J ( v (t)) + K^* \qquad \text{ for $\mathcal{L}^1$-a.e. } t \in [0,T].
$$

\vspace{.2cm}

\noindent
\textbf{Step 5.} We prove the energy inequality.
Let now $0 \leq t_1 < t_2 \leq T$. 
From (iv) in  Theorem \ref{erfd} we have
$$
J ( v_{\delta_k} (t_2) ) + \mathrm{Var}_{\psi} (v_{\delta_k}; [t_1, t_2]) 
\leq J ( v_{\delta_k} (t_1) ) +\int_{t_1}^{t_2} \langle q_{\delta_k} (s) ,  \dot{f} (s) \rangle_{\mathcal{F}} \, \mathrm{d}s + r_2 (\delta_k).
$$
Passing to the limit as $k \to \infty$, and using the fact that the total variation
is lower semicontinuous with respect to the $L^1$ topology
\begin{align*}
&J ( v (t_2) ) + \mathrm{Var}_{\psi} (v; [t_1, t_2]) \leq 
\liminf_{k \to \infty} \left( J ( v_{\delta_k} (t_2) ) + \mathrm{Var}_{\psi} (v_{\delta_k}; [t_1, t_2]) \right) \\
&\leq \liminf_{k \to \infty} \left( J ( v_{\delta_k} (t_1) ) 
+ \int_{t_1}^{t_2} \langle q_{\delta_k} (s) ,  \dot{f} (s) \rangle_{\mathcal{F}} \, \mathrm{d}s + r_2 (\delta_k) \right) \\
&= J ( v (t_1) ) + \int_{t_1}^{t_2} \langle q (s) ,  \dot{f} (s) \rangle_{\mathcal{F}} \, \mathrm{d}s.
\end{align*}
\end{proof}

\section{Proof of Theorem \ref{Teorexist}} \label{secproof}

We now prove Theorem \ref{Teorexist}. A relevant difference with the previous section is that in the discrete approximate energy inequality we need to choose $q_\delta(t)$ in a different way to cope later with the lack fo $BV$ compactness. In particular we will need the inclusion $A^* q_\delta (t) \in \partial J (v_\delta (t))$ to hold already at this level. In order to do this, we need the additional assumption (J4).

\begin{theorem} \label{teor18}
Let $\alpha = 0$, let \eqref{gamma}, (J1), (J2), (J3), and (J4) be satisfied, 
and let $f \in W^{1,2} ( [0,T] ; \mathcal F)$.
Let $\delta \in (0,1)$ and let $v_{0}$ be a critical point of $J$ 
on the affine space $\mathbf{A} (f (0))$.
Let $v_{\delta} : [0,T] \to \mathcal{E}$ be the discrete quasistatic evolution 
with time step $\delta$, initial condition $v_0$, and constraint $f$
given by \eqref{approx}. 
Then, there exist a piecewise constant right-continuous function
$q_\delta: [0,T] \to \mathcal F'$ 
and positive constants $C_1$, $C_2$ and $C_3$, independent of $\delta$, such that
\begin{itemize}
 \item[\textit{(i)}] $A^* q_\delta (t) \in \partial J (v_\delta (t))$ for every $t \in [0,T]$;

\vspace{.2cm}

\item[\textit{(ii)}] $|q_\delta (t)|_{\mathcal{F}'} \le C_1$ 
and $|v_\delta (t)|_{\mathcal{E}} \leq C_2$ for every $t \in [0,T]$;

\vspace{.2cm}

\item[\textit{(iii)}] for every $t_1<t_2 \in [0,T]$ 
$$
J (v_\delta(t_2)) \leq J (v_\delta(t_1)) 
+ \int_{t_1}^{t_2}  \langle q_\delta (s) ,  \dot{f} (s) \rangle_{\mathcal{F}} \, \mathrm{d}s+ C_3 \sqrt{ \delta}.
$$ 
\end{itemize}
\end{theorem}

\begin{remark} \label{dependence}
More precisely, as it appears by a careful reading of the proof 
of Theorem \ref{teor18}, we have
$$
C_1 = C_1 (J,\gamma, L, \eta, \| \dot{f} \|_{L^2 ((0,T); \mathcal{F})})
\quad \text{ and } \quad 
C_3 = C_3 (J, \gamma, L, \eta, C_{J, \eta}, \| \dot{f} \|_{L^2 ((0,T); \mathcal{F})}).
$$
The boundedness of $v(t)$ has been already proved in Proposition \ref{z1} and one can take for $C_2$ the constant $Z_1$ provided there. The dependence on $J$ has again to be understood in the sense of the coercivity of the function 
$v \mapsto J (v )+|Av|^2_{\mathcal{E}}$, see the proof of Theorem~\ref{erfd}.   
\end{remark}

\begin{proof}
Since $v_{\delta}$ is a discrete quasistatic evolution, for every 
$i \in \mathbb{N}$ with $i \delta \leq T$ there exists $q^i \in \mathcal F'$ 
such that $A^* q^i \in \partial J (v_\delta (i \delta))$. 
Then, if we define $q_{\delta} : [0,T] \to \mathcal{F}'$ as 
\begin{equation*}
q_\delta(t) : =q^i \quad \hbox{ for all } t \in [0,T] \cap [i \delta, (i + 1) \delta), 
\quad \hbox{ for all } i \in \mathbb{N}_0 \text{ with } i \delta \leq T,
\end{equation*}
property (i) is satisfied. Since $|v_\delta (t)|_{\mathcal{F}}$ is bounded by Proposition \ref{z1}, with (J3) and \eqref{gamma} we obtain (ii)
We now divide the remaining part of the proof into two steps.

\vspace{.2cm}

\noindent
\textbf{Step 1.} We show that there exists  a constant $M$, 
depending only on $\eta$ and $C_{J, \eta}$, such that
\begin{equation} \label{try}
J (v_{\delta} (i \delta)) \leq J (v_\delta((i-1) \delta) )
+ \int_{(i-1) \delta}^{i \delta}  \langle q^{\delta} (s) ,  \dot{f} (s) \rangle_{\mathcal{F}} \, \mathrm{d}s
+M \delta  \int_{(i-1) \delta}^{i \delta} | \dot{f} (s)|^2_{\mathcal{F}} \,\mathrm{d}s,
\end{equation}
for every $i \in \mathbb{N}$ with $i \delta \leq T$.

To this aim, let $i \in \mathbb{N}$ with $i \delta \leq T$ be fixed, 
and let $\left( v^i_j \right)_{j \in \mathbb N_0}$ be the sequence defined by \eqref{algobis}.
By property (J2) and Remark \ref{vbar}, the functional $J_{\eta, v^i_0}$
is strongly convex. Therefore, whenever $\xi \in \partial J_{\eta, v^i_0} (v^i_1)$, 
we have  
$$
J_{\eta, v^i_0} (v) \geq J_{\eta, v^i_0} (v^i_1) 
+ \langle \xi , v - v^i_1 \rangle_{\mathcal{E}} 
\qquad \text{ for every } v \in \mathcal{E}.
$$
In particular, choosing 
$v = v^i_0$ and recalling the definition of $J_{\eta, v^i_0}$ we have 
\begin{equation} \label{step1}
J (v^i_0) \geq J (v^i_1) + \eta |v^i_1 - v^i_0|_{\mathcal{E}}^2
+ \langle \xi , v^i_0 - v^i_1 \rangle_{\mathcal{E}} \qquad 
\text{ for every } \xi \in \partial J_{\eta, v^i_0} (v^i_1).
\end{equation}
By \eqref{algobis}, $v^i_1$ is the global minimizer of $J_{\eta, v^i_0}$ on $\mathbf{A} (f (i \delta))$.
Therefore, there exists $r^i \in \mathcal F'$ 
such that $A^* r^i \in \partial J_{\eta, v^i_0}(v^i_1)$ so that, by \eqref{step1},
$$
J (v^i_0) \geq J (v^i_1) + \eta |v^i_1 - v^i_0|_{\mathcal{E}}^2
+ \langle A^* r^i , v^i_0 - v^i_1 \rangle_{\mathcal{E}}.
$$
Therefore, by the absolute continuity of $f$, and
recalling that $q_\delta$ is constant in the interval $[ (i-1) \delta , i \delta)$, we have  
\begin{align}
&\eta |v^i_1 - v^i_0 |^2_{\mathcal{E}} 
+ J (v^i_1) - J (v^i_0) 
\leq \langle A^* r^i , v^i_1 - v^i_0 \rangle_{\mathcal{E}} \nonumber \\
&\hspace{.2cm}= \langle A^* r^i - A^*q_\delta((i-1) \delta), 
v^i_1 - v^i_0 \rangle_{\mathcal{E}}
+ \langle A^*q_\delta((i-1) \delta), v^i_1 - v^i_0 \rangle_{\mathcal{E}} \nonumber \\
&\hspace{.2cm}= \langle r^i - q_\delta((i-1) \delta) , Av^i_1 - Av^i_0 \rangle_{\mathcal{F}} 
+  \langle q^{\delta} ((i-1) \delta) ,  Av^i_1 - Av^i_0 \rangle_{\mathcal{F}} \label{conti} \\
&\hspace{.2cm}= \langle r^i - q_\delta((i-1) \delta) , Av^i_1 - Av^i_0 \rangle_{\mathcal{F}} 
+ \int_{(i-1) \delta}^{i \delta}  
\langle q^{\delta} ((i-1) \delta) ,  \dot{f} (s) \rangle_{\mathcal{F}} \, \mathrm{d}s, \nonumber
\end{align}
where we used the fact that $Av^i_1 = f (i \delta)$ and $Av^i_0 = f ((i-1) \delta)$.
Observe now that, by definition of $q_{\delta}$, 
we have $A^* q_{\delta} ((i-1) \delta) \in \partial J (v_{\delta} ((i-1) \delta))$.
Thus, recalling that $v_{\delta} ((i-1) \delta) = v^{i-1} = v^i_0$, we obtain 
$$
A^* q_{\delta} ((i-1) \delta) \in \partial J (v^i_0) = \partial J_{\eta, v^i_0}(v^i_0).
$$
Thus, recalling that $A^* r^i \in \partial J_{\eta, v^i_0}(v^i_1)$, by property (J4) we obtain
$$
\langle r^i - q_\delta((i-1) \delta) , Av^i_1 - Av^i_0 \rangle_{\mathcal{F}}
\leq C_{J,\eta}|v^i_1 - v^i_0 |_{\mathcal{E}}
 | Av^i_1 - Av^i_0 |_{\mathcal{F}}
$$
which, together with \eqref{conti},  gives
\begin{align*}
\eta |v^i_1 - v^i_0 |^2_{\mathcal{E}} + J (v^i_1) - J (v^i_0)
\leq  C_{J,\eta}|v^i_1 - v^i_0 |_{\mathcal{E}}
 | f(i \delta) - f((i-1) \delta) |_{\mathcal{F}}
+ \int_{(i-1) \delta}^{i \delta}  \langle q^{\delta} (s) ,  \dot{f} (s) \rangle_{\mathcal{F}} \, \mathrm{d}s.
\end{align*}
Using Young's and H\"{o}lder's inequality, we get
\begin{align} 
J (v^i_1) &\leq J (v^i_0) + M | f(i \delta) - f((i-1) \delta) |^2_{\mathcal{F}} 
+ \int_{(i-1) \delta}^{i \delta}  \langle q^{\delta} (s) ,  \dot{f} (s) \rangle_{\mathcal{F}} \, \mathrm{d}s \nonumber \\
 &= J (v^i_0) + M \left| \int_{(i-1) \delta}^{i \delta} \dot{f} (s) \, \mathrm{d}s \right|^2_{\mathcal{F}}  
 + \int_{(i-1) \delta}^{i \delta}  \langle q^{\delta} (s) ,  \dot{f} (s) \rangle_{\mathcal{F}} \, \mathrm{d}s \label{poli}\\
&\leq J (v^i_0) + M \delta \int_{(i-1) \delta}^{ i \delta} |\dot{f} (s)|^2_{\mathcal{F}} \, \mathrm{d}s 
+ \int_{(i-1) \delta}^{i \delta}  \langle q^{\delta} (s) ,  \dot{f} (s) \rangle_{\mathcal{F}} \, \mathrm{d}s, \nonumber
\end{align}
for a suitable constant $M>0$ (depending only on $\eta$ and $C_{J, \eta}$), 
where we also used the fact that $f \in W^{1,2}([0,T];\mathcal F)$.
Recalling that $J(v^i_j)\le J(v^i_1)$ for all $j\ge 2$, 
by \eqref{approx} and property (ii) in  Lemma \ref{lemmajNew}, we have 
$$
J (v_{\delta} (i \delta)) = J (v^i) = \lim_{j \to \infty} J (v^i_j) \le J(v^i_1).
$$
Taking into account last inequality, and recalling that $v^i_0 = v^{i-1} = v_\delta( (i-1) \delta)$, 
relation \eqref{poli} gives \eqref{try}.

\vspace{.2cm}

\noindent
\textbf{Step 2.} We show (iii). Let $l , k \in \mathbb{N}$ be such that
$l \delta \leq t_1 < (l+1) \delta$ and $k \delta \leq t_2 < (k+1) \delta$. Recall that $v_{\delta} (k \delta)=v_\delta(t_2)$ and $v_{\delta} (l \delta)=v_\delta(t_1)$ by construction. By summing up relation \eqref{try} for $i = l+1, \ldots, k$, using H\"{o}lder's  inequality and taking into account that $\delta \in (0,1)$ we have
\begin{eqnarray*}
&\displaystyle
J (v_{\delta} (t_2)) \leq J (v_{\delta} (t_1) )
+ \int_{l\delta}^{k\delta}  \langle q_{\delta} (s) ,  \dot{f} (s) \rangle_{\mathcal{F}} \, \mathrm{d}s
+ M \delta \| \dot{f} \|^2_{L^2 ((0,T); \mathcal{F})} \\
&\displaystyle
= J(v_{\delta} (t_1) )
+ \int_{t_1}^{t_2}  \langle q_{\delta} (s) ,  \dot{f} (s) \rangle_{\mathcal{F}} \, \mathrm{d}s 
+ \int_{l\delta}^{t_1}  \langle q_{\delta} (s) ,  \dot{f} (s) \rangle_{\mathcal{F}} \, \mathrm{d}s- \int_{t_2}^{k\delta}  \langle q^{\delta} (s) ,  \dot{f} (s) \rangle_{\mathcal{F}} \, \mathrm{d}s
+M  \delta \| \dot{f} \|^2_{L^2 ((0,T); \mathcal{F})} \\
&\displaystyle
\leq J(v_{\delta} (t_1))+ \int_{t_1}^{t_2}  \langle q^{\delta} (s) ,  \dot{f} (s) \rangle_{\mathcal{F}}\,\mathrm{d}s
+ C_1 \sqrt{ \delta} \| \dot{f} \|_{L^2 ((0,T); \mathcal{F})}
+ M \sqrt{\delta} \| \dot{f} \|^2_{L^2 ((0,T); \mathcal{F})} \\
&= J(v_{\delta} (t_1))+ \int_{t_1}^{t_2}  \langle q^{\delta} (s) ,  \dot{f} (s) \rangle_{\mathcal{F}}\,\mathrm{d}s
+ \sqrt{ \delta} \left( C_1  \| \dot{f} \|_{L^2 ((0,T); \mathcal{F})}
+ M \| \dot{f} \|^2_{L^2 ((0,T); \mathcal{F})} \right),
\end{eqnarray*}
which gives (iii).
\end{proof}
Before giving the proof of Theorem \ref{Teorexist}, 
we need the following result (see \cite[Lemma 3.6]{dmgipo09}).
\begin{lemma} \label{DalGP}   
Let $X$ be a compact metric space.
Let $p:[0,T] \to \mathbb{R}$,
$p_k:[0,T] \to \mathbb{R}$ and
$f_k: [0,T] \rightarrow X$ be measurable functions,
for every $k \in \mathbb{N}$.
For every $t \in [0,T]$ let us set
\begin{align*}
\mathcal{I} (t):= \{ x \in X : \, \exists \, k_j \rightarrow + \infty
\textnormal{ such that }x = \lim_{j \rightarrow + \infty} f_{k_j} (t)
\text{ and } p (t) = \lim_{j \rightarrow \infty} p_{k_j} (t) \}.
\end{align*}
Then
\begin{itemize}
\item $\mathcal{I}(t)$ is closed for all $t \in [0,T]$; 
\vspace{.1cm}
\item for every open set $U \subseteq X$ the set 
$\{ t \in [0,T]: \mathcal{I}(t) \cap U \neq \emptyset \}$
is measurable.
\end{itemize}
\end{lemma}
We conclude this section with the proof of Theorem \ref{Teorexist}.

\begin{proof}[Proof of Theorem \ref{Teorexist}]

We divide the proof into several steps.

\vspace{.2cm}

\noindent
\textbf{Step 1: proof of (a) and (b).}

\vspace{.2cm}
\noindent
For every $\delta \in (0, 1)$, let $q_{\delta}: [0,T] \to \mathcal F'$ 
and $v_{\delta}: [0,T] \to \mathcal E$ be given by Theorem \ref{teor18}. 
Let $\Lambda \subset [0,T]$ be such that $\mathcal{L}^1 (\Lambda) = 0$ 
and $\dot{f} (t)$ is well defined for every $t \in [0,T] \setminus \Lambda$.
We fix a sequence $\left( \delta_k \right)_{k \in \mathbb{N}}$ such that $\delta_k \to 0^+$
and define
$$
\theta_k (t) : = 
\begin{cases}
\langle q_{\delta_k} (t) ,  \dot{f} (t) \rangle_{\mathcal{F}} & \text{ for every } 
t \in [0,T] \setminus \Lambda, \vspace{.1cm} \\
0 & \text{ for every } t \in \Lambda,
\end{cases}
$$
and 
$$
\theta (t):= \limsup_{k \to \infty} \theta_k (t) \quad \text{ for every } t \in [0,T].
$$
By definition of $\theta$, for every $t \in [0,T]$ we can extract 
a subsequence $\left( \delta_{k_j} \right)_{j \in \mathbb{N}}$ (possibly depending on $t$) such that
$$
\theta (t) = \lim_{j \to \infty} \theta_{k_j} (t) \quad \text{ for every } t \in [0,T].
$$
By (ii) of Theorem \ref{teor18}, we have 
$$
| q_{\delta} (t) |_{\mathcal{F}'} \leq C_1,  
\quad | v_{\delta} (t) |_{\mathcal{E}} \leq C_2, \quad \text{ for every } 
\delta \in (0, 1) \text{ and } t \in [0,T].
$$ 
Thus, for every $t \in [0,T]$ we can extract a further subsequence
(not relabelled) such that
$$
\lim_{j \to \infty} v_{\delta_{k_j}} (t) = v (t) \quad \text{ in } \mathcal E
\quad \text{ and } \quad 
\lim_{j \to \infty} q_{\delta_{k_j}} (t) = q (t)\quad \text{ in } \mathcal F',
$$
for some $v (t) \in \mathcal E$ and $q (t) \in \mathcal F'$
with $| q (t) |_{\mathcal{F}'} \leq C_1$ and $| v (t) |_{\mathcal{E}} \leq C_2$.
Let us now show that, for every $t \in [0,T]$, we can choose 
the subsequence $\left( k_j \right)_{j \in \mathbb{N}}$ in such a way that the maps 
$q:[0,T] \to \mathcal F'$ and $v:[0,T] \to \mathcal E$
are measurable.

Let us denote by $B^{\mathcal{F}'}_{C_1}$ ($B^{\mathcal{E}}_{C_2}$)
the closed ball of $\mathcal{F}$ ($\mathcal{E}$) with
center at the origin and radius $C_1$ ($C_2$).
Applying Lemma~\ref{DalGP} with 
$X = B^{\mathcal{F}'}_{C_1} \times B^{\mathcal{E}}_{C_2}$, 
$f_k = (q_{\delta_k} , v_{\delta_k})$ and $p_k = \theta_k$, 
we have that
\begin{itemize}
\item $\mathcal{I}(t)$ is closed for all $t \in [0,T]$,
\vspace{.1cm}
\item for every open set $U \subseteq X$ the set 
$\{ t \in [0,T]: \mathcal{I}(t) \cap U \neq \emptyset \}$
is measurable,
\end{itemize}
where the set $\mathcal{I} (t)$ is given by 
\begin{align*}
\mathcal{I} (t) &:= \{ (q (t),v(t)) \in 
B^{\mathcal{F}'}_{C_1} \times B^{\mathcal{E}}_{C_2} : 
\, \exists \, k_j \rightarrow + \infty \textnormal{ such that } \\
&(q (t),v(t)) = \lim_{j \rightarrow + \infty} (q_{\delta_{k_j}} (t), v_{\delta_{k_j}} (t))
\text{ and } \theta (t) = \lim_{j \rightarrow \infty} \theta_{k_j} (t) \}.
\end{align*}
Thanks to \cite[Theorem III.6]{cava77}, for every $t \in [0,T]$ 
we can select $(q (t),v(t)) \in B^{\mathcal{F}'}_{C_1} \times B^{\mathcal{E}}_{C_2}$ such that
$t \to (q (t),v(t))$ is measurable.
Thus, (a) is proven.
Finally, by repeating the arguments used in the proof of Proposition~\ref{tre2} we obtain (b).

\vspace{.2cm}

\noindent
\textbf{Step 2: Proof of (c).}
\vspace{.2cm}
\noindent
Observe that, for every $t \in [0,T] \setminus \Lambda$, 
$$
\theta (t) = \limsup_{k \to \infty} \theta_k (t) = \lim_{j \to \infty} \theta_{k_j} (t)
= \lim_{j \to \infty} \langle q_{\delta_{k_j}} (t) ,  \dot{f} (t) \rangle_{\mathcal F}
= \langle q (t) ,  \dot{f} (t) \rangle_{\mathcal F}.
$$
Let us now show that $\theta \in L^1 (0,T)$.
Since $\theta$ is the $\limsup$ of measurable functions, we deduce that it is measurable.
Moreover, we have 
$$
\int_0^T | \theta (t) | \, dt = \int_0^T | \langle q (t) ,  \dot{f} (t) \rangle_{\mathcal F} | \, dt
\leq C_1 \int_0^T |\dot{f} (t)|_{\mathcal F} \, dt \leq C_1 \sqrt{T} \, \| \dot{f} \|_{L^2 ((0,T); \mathcal F)}.
$$
In order to get the energy inequality, recall that
by (iii) of Theorem \ref{teor18} (for $t_2=t$ and $t_1=0$) we have, for every $j \in \mathbb{N}$,
$$
J (v_{\delta_{k_j}}(t)) \leq J (v_0) 
+ \int_{0}^{t}  \langle q_{\delta_{k_j}} (s) ,  \dot{f} (s) \rangle_{\mathcal F} \, \mathrm{d}s
+ C_3 \, \delta_{k_j}^{1/2}.
$$ 
Taking the limsup in $j$ of the previous expression, using Fatou's Lemma
\begin{align*}
J (v (t)) &= \lim_{j \to \infty}
J (v_{\delta_{k_j}}(t)) \leq J (v_0) + \limsup_{j \to \infty} 
\int_{0}^{t}  \langle q_{\delta_{k_j}} (s) ,  \dot{f} (s) \rangle \, \mathrm{d}s \\
&\leq J (v_0) + \limsup_{k \to \infty} 
\int_{0}^{t}  \langle q_{\delta_{k}} (s) ,  \dot{f} (s) \rangle_{\mathcal F} \, \mathrm{d}s
\leq J (v_0) + \int_{0}^{t}  \limsup_{k \to \infty} \, 
\langle q_{\delta_{k}} (s) ,  \dot{f} (s) \rangle_{\mathcal F} \, \mathrm{d}s \\
&= J (v_0) + \int_{0}^{t}  \langle q (s) ,  \dot{f} (s) \rangle_{\mathcal F} \, \mathrm{d}s,
\end{align*}
so that (c) follows.
\end{proof}

\end{section}

\section{A discrete version of fracture evolution for cohesive zone models}\label{appl}

In the remaining part of the paper, we show how to apply Theorem \ref{Teorexist}
to cohesive fracture evolution.
We start this section by recalling the model introduced in~\cite{C}, 
where a critical points evolution is obtained by following the scheme described 
in the first part of the Introduction.
We conclude the section performing a finite dimensional discretization.
In the next section we will then show how it is possible to pass to the limit, 
thus obtaining a different proof of the existence result in~\cite{C}.

\subsection{A previous model for the time evolution of cohesive fractures} \label{recallfract}

Let $\Omega \subset \mathbb{R}^d$ be a bounded 
open set with Lipschitz boundary, with $d \in \{ 1, 2 \}$.
We assume that the reference configuration is the infinite cylinder
$\Omega \times \mathbb{R}\subset \mathbb{R}^{d+1}$,
and that the displacement $U : \Omega \times \mathbb{R} \to \mathbb{R}^{d+1}$
has the special form 
$$
U (x_1, \ldots, x_d, x_{d+1}) = (0, \ldots, 0, u (x_1, \ldots, x_d)), 
$$ 
where $u: \Omega \to \mathbb{R}$.
This situation is referred to in the literature as \textit{generalized antiplanar shear}.

We assume that the crack path in the reference
configuration is contained in $\Gamma \cap \overline{\Omega}$, 
where $\Gamma \subset \mathbb R^d$ is a Lipschitz closed set such that 
$0 < \mathcal H^{d-1}(\Gamma \cap \overline{\Omega}) < \infty$ 
and $\Omega \setminus \Gamma = \Omega^+ \cup \Omega^-$,
where $\Omega^+$ and $\Omega^-$ are disjoint open connected sets with Lipschitz boundary. 

We will study the energy of a finite portion of the cylinder, 
obtained by intersection with two horizontal hyperplanes separated by a unit distance.
Given a time interval $[0,T]$, with $T > 0$, we assume  
that the evolution is driven by a time-dependent displacement
$\omega \in H^1 ([0,T] ; H^1 (\Omega))$, imposed on a fixed portion
$\partial_D \Omega$ of $\partial \Omega$.
We make the assumption that $\partial_D \Omega$ is well separated from $\Gamma$, 
and that $\mathcal{H}^{d-1} ( \partial_D \Omega \cap \partial \Omega^{\pm} ) > 0$.

In the framework of linearized elasticity, the stored elastic energy associated 
with a displacement $u \in H^1 (\Omega \setminus \Gamma)$ is given by 
\begin{equation} \label{elastic}
W (u):= \frac{1}{2} \int_{\Omega \setminus \Gamma} |\nabla u|^2 \, \mathrm{d}x.
\end{equation}
The crack in the reference configuration can be identified with the set
$$
J_u:= \{ x \in \Gamma : u^+ (x) \neq u^- (x) \},
$$ 
where $u^{\pm}$ denotes the trace on $\Gamma$ of the restriction of $u$ to $\Omega^{\pm}$. 
In order to take into account the cohesive forces acting 
between the lips of the crack, according to Barenblatt's model \cite{Bar62} 
we consider a fracture energy of the following form:
$$
\kappa\int_\Gamma g (|[u]|) d \mathcal H^{d-1}, 
$$
where $[u]:= u^+ - u^-$ is the jump of $u$ across $\Gamma$, 
and the energy density $g: [0, \infty) \to [0, \infty)$ is a $C^1$, 
nondecreasing, bounded, concave function with $g(0) = 0$ and 
$\sigma:=g'(0^+) \in (0,+\infty)$.
We will consider here the special case (see also \cite[Section 9]{C}) 
\begin{equation}\label{g}
g (s) = 
\begin{cases}
 \vspace{.1cm}
 - \frac{s^2}{2 R} + s &\text{ if } 0 \leq s < R, \\
\frac{R}{2} &\text{ if }  s \geq R,
\end{cases}
\end{equation}
where the parameter $R > 0$ represents the range of the cohesive interactions 
between the lips of the crack.
The parameter $\kappa$ is the stiffness material constant and, for sake of simplicity, it will be taken equal to $1$ through all the theoretical analysis.
Summarizing, the energy functional $E: H^1 (\Omega \setminus \Gamma) \to [0, \infty)$ 
is given by
\begin{equation}\label{eq:Energy2}
E (u) = \frac{1}{2} \int_{\Omega \setminus \Gamma} |\nabla u|^2 \, \mathrm{d}x 
+ \int_\Gamma g (|[u]|) d \mathcal H^{d-1}.
\end{equation}

Let $t \in [0,T]$, and let $u(t)$ be a  minimizer for the problem 
$$
\min \{ E (v) : v \in H^1 (\Omega \setminus \Gamma), v = \omega (t) \text{ on } \partial_D \Omega \}. 
$$
Then \cite[Proposition 3.1]{C}, 
\begin{equation}\label{eq:crtpcond3}
\int_{\Omega \setminus \Gamma} \nabla u (t) \cdot \nabla \psi \, \mathrm{d}x + \int_\Gamma \left ([\psi] g'(|[u(t)]|)\operatorname{sign}([u(t)])1_{J_{u(t)}} + |[\psi]| 1_{J_{u(t)}^c} \right ) d \mathcal H^{d-1} \geq 0,
\end{equation}
for all $\psi \in H^1 (\Omega\setminus \Gamma)$ with $\psi = 0$ on $\partial_D \Omega$.
One can see \cite[Proposition 3.2]{C} that $u(t)$ satisfies \eqref{eq:crtpcond3} if and only if 
it is a weak solution of 
\begin{equation}\label{eq:crptcond}
\begin{cases}
\Delta u(t) = 0 & \mbox{ in } \Omega \setminus \Gamma, \\
u(t) = \omega(t) & \mbox{ on } \partial_D \Omega,\\
\partial_\nu u(t) = 0 & \mbox{ on } \partial \Omega \setminus \partial_D \Omega, \\
\partial_\nu u^+(t) = \partial_\nu u^-(t) & \mbox{ on } \Gamma, \\
|\partial_\nu u (t)| \leq 1 &\mbox{ on } \Gamma \setminus J_{u(t)}, \\
\partial_\nu u (t)= g'(|[u(t)]|)\operatorname{sign}([u(t)]) & \mbox{ on } J_{u(t)},
\end{cases}
\end{equation}
where $\operatorname{sign}( \cdot )$ denotes the sign function.
Any function $u (t)$ satisfying \eqref{eq:crtpcond3} or \eqref{eq:crptcond} 
will be referred to as \textit{critical point of $E$ at time $t$}.

In \cite{C}, a critical points evolution is obtained by following the general ideas given in the Introduction, 
by setting $Y = L^2(\Omega)$ and
$$
F(u,t) = 
\begin{cases}
E(u) & \mbox { for } u \in H^1(\Omega \setminus \Gamma), u = w(t) \mbox{ on } \partial_D \Omega, \\
+\infty & \mbox { otherwise in } L^2(\Omega).
\end{cases}
$$
Given a critical point $u_0$ of $E$ at time $0$ and $\varepsilon \in (0,1)$, 
the author shows \cite[Theorem 4.1]{C} the existence of a function 
$u^\varepsilon: [0,T] \to H^1 (\Omega \setminus \Gamma)$ 
satisfying $u^{\varepsilon}(0) = u_0$ and \eqref{eq:evcrptsBp}, which in this case reads as
\begin{equation}\label{eq:crptcond2}
\begin{cases}
\Delta u^\varepsilon(t) =  \varepsilon \dot u^\varepsilon(t) & \mbox{ in } \Omega \setminus \Gamma, \\
u^\varepsilon(t) = \omega(t) & \mbox{ on } \partial_D \Omega,\\
\partial_\nu u^\varepsilon(t) = 0 & \mbox{ on } \partial \Omega \setminus \partial_D \Omega, \\
\partial_\nu u^\varepsilon(t)_{| \Omega^-} 
= \partial_\nu u^\varepsilon_{| \Omega^+}(t)&  \mbox{ on } \Gamma, \\
|\partial_\nu u^\varepsilon (t)| \leq 1 &\mbox{ on } \Gamma \setminus J_{u^\varepsilon(t)}, \\
\partial_\nu u^\varepsilon = g'(|[u^\varepsilon (t)]|)
\operatorname{sign}([u^\varepsilon (t)]) & \mbox{ on } S_{u^\varepsilon(t)}.
\end{cases}
\end{equation}
Under the additional assumption $g \in C^{1,1}$ (which is fulfilled by the function in \eqref{g}), 
uniqueness for the above problem also holds true \cite[Theorem 4.2]{C}. 
Finally \cite[Theorem 4.4]{C}, there exists a bounded measurable function 
$u:[0,T] \to H^1(\Omega \setminus \Gamma)$ with $u(0)=u_0$ such that the following properties are satisfied:

\vspace{.2cm}

\begin{itemize}
\item {\it approximability}: for every $t \in [0,T]$ there exists a sequence $\varepsilon_n \to 0^+$ 
such that
$$
u^{\varepsilon_n}(t) \rightharpoonup u(t) \mbox{ weakly in } H^1(\Omega \setminus \Gamma);
$$

\vspace{.2cm}

\item {\it stationarity}: for a.e. $t \in [0,T]$ the function $u(t)$ is a critical point of $E$ at time $t$;

\vspace{.2cm}

\item {\it energy inequality}: for every $t \in [0,T]$
\begin{equation}\label{eq:enineq}
E(u(t)) \leq E(u(0)) + \int_0^t \int_{\Omega \setminus \Gamma} \nabla u(s) \cdot \nabla \dot \omega(s) \, \mathrm{d}x \mathrm{d}s.
\end{equation}
\end{itemize}

\subsection{Discrete Setting} \label{sec:discrete}

In view of the applications of the results of this paper to the model just introduced, 
we need a finite dimensional version of the energy functional $E$ in \eqref{eq:Energy2}. 
In order to focus on the main ideas of our approach, 
we keep the formulation as clear as possible, considering a very simple geometry.

Let $d=2$ and $\ell > 0$ be fixed, and define
$$
\Omega:= (0,2\ell)^2, \qquad 
\Omega^-:= (0, \ell) \times (0,2\ell),
\qquad
\Omega^+:= (\ell, 2 \ell) \times (0,2\ell),
\qquad
\Gamma:= \{ \ell \} \times [0,2\ell].
$$
We will study a fracture evolution where the deformation is imposed on the set 
$$
\partial_D\Omega:= \left( \{0\}\times[0,2\ell] \right) \cup \left( \{2\ell\}\times [0,2\ell] \right),
$$
see Figure \ref{mesh0}.
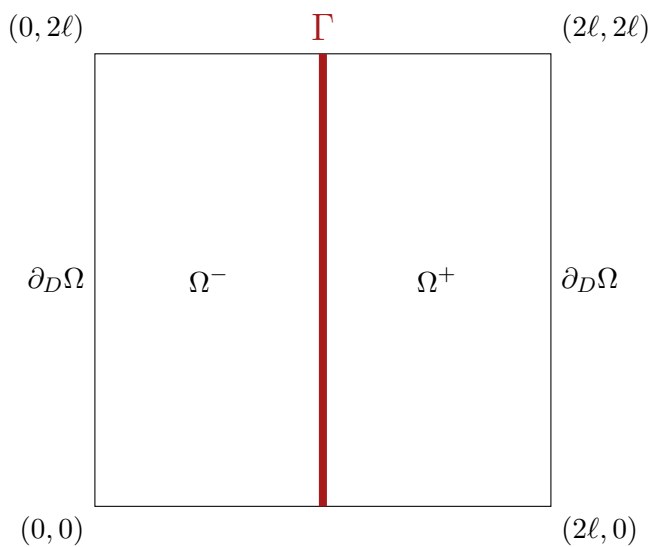
\begin{figure}[h]
	\begin{center}
		\begin{tikzpicture}[>=latex, scale=0.6]
			\draw(0,0)rectangle(10,10);
						
			\draw[line width=3pt, color=nicosred] (5,0)--(5,10)node[above]{\Large{$\Gamma$}};
			
			\draw(0,0)node[below left]{$(0,0)$};
			\draw(10,10)node[above right]{$(2 \ell, 2 \ell)$};
			\draw(0,10)node[above left]{$(0, 2 \ell)$};		
			\draw(10,0)node[below right]{$(2 \ell, 0)$};
			\draw(10, 5)node[right]{$\partial_D \Omega$};
			\draw(0, 5)node[left]{$\partial_D \Omega$};
			\draw(2.5, 5)node{$\Omega^-$};
			\draw(7.5, 5)node{$\Omega^+$};
		\end{tikzpicture}
	\end{center}
	\caption{Geometry of the problem.}\label{mesh0}
\end{figure}





For every $h > 0$, we assume we are given a triangulation
$\mathcal T_h$ of the set $\Omega \setminus \Gamma$,
and we define $\mathcal{E}_h$ as the finite dimensional space of continuous functions 
that are affine on each triangle belonging to $\mathcal T_h$.
More precisely, we set
$$
\mathcal{E}_h := \{u \in C(\overline{\Omega}\setminus \Gamma) \cap H^1 (\Omega\setminus \Gamma) 
:\quad \nabla u=\hbox{ const. a.e. on $T$, \quad for every }T\in \mathcal T_h\}\, ,
$$
and we define $\mathcal{E}_h^{reg}$ as the set of functions of $\mathcal{E}_h$ 
that do not jump across $\Gamma$:
$$
\mathcal{E}_h^{reg}:= \mathcal{E}_h \cap H^1 (\Omega).
$$

We endow $\mathcal{E}_h$ with the induced norm of $H^1 (\Omega \setminus \Gamma)$
$$
| u |^2_{\mathcal{E}_h} := \int_{\Omega}  u^2 \,\mathrm{d}x
+ \int_{\Omega \setminus \Gamma} |\nabla u |^2  \,\mathrm{d}x \qquad \qquad u \in \Esp, 
$$
With a slight abuse of notation, we will denote with
$\langle \cdot , \cdot \rangle_{\Esp}$ both the duality pairing between $\Esp'$ and $\Esp$ (whenever the two are not identified), as well as the scalar product in $\Esp$. In particular,
throughout this section, we convene that the equality
$$
\xi  = v
$$
where $\xi \in \Esp'$ and $v \in \Esp$, is meant in sense of the Riesz isometry.
We will denote the restriction of the energy functionals $E$ and $W$ to the space $\mathcal{E}_h$ by
$$
E_h:= E \mid_{\mathcal{E}_h}, \qquad  W_h:= W \mid_{\mathcal{E}_h}.
$$
We denote by $A_h$ the operator which associates to every function of
$\mathcal{E}_h$ its trace on $\partial_D \Omega$, and we set $\mathcal{F}_h:= A_h (\mathcal{E}_h)$.
Note that $\mathcal{F}_h$ is closed, since $\mathcal{E}_h$ is finite dimensional.
Therefore, $\mathcal{F}_h$ endowed with the induced scalar product 
$\langle \cdot, \cdot \rangle_{\mathcal{F}_h}$ is a Hilbert subspace of $H^{1/2} (\partial_D \Omega)$.
We also notice that, since the operators $A_h$ are the restrictions to $\mathcal{E}_h$ of the surjective trace operator $A: H^{1}(\Omega \setminus \Gamma)\to H^{1/2} (\partial_D \Omega)$, by the open mapping theorem we can assume that \eqref{gamma} is satisfied for all $h$ with a constant $\gamma$ independent of $h$. In the sequel this will be used tacitely.

We finally observe that, applying \cite[Lemma 4.1.3]{Z} to our setting, 
we obtain the following version of Poincar\'e inequality: 
$$
\| u - u_{D} \|_{L^2 (\Omega)} \leq C \| \nabla u \|_{L^2 (\Omega \setminus \Gamma)}
\qquad \text{ for every } u \in \Esp,
$$
where the constant $C$ depends on $\Omega$ and $\partial_D \Omega$, and
$$
u_{D} :=   \left( \, \dashint_{\partial_D \Omega} u^2 \, d \mathcal{H}^1 \right)^{1/2}.
$$
The previous inequality in particular implies that
\begin{equation} \label{poincare}
\| u \|_{L^2 (\Omega)} \leq u_{D} + C \| \nabla u \|_{L^2 (\Omega \setminus \Gamma)}
\leq C \left( \left| A_h u \right|_{\mathcal{F}_h} + \| \nabla u \|_{L^2 (\Omega \setminus \Gamma)} \right) \qquad \text{ for every } u \in \Esp,
\end{equation} 
where with $C$ we denote different constants, all depending on $\Omega$ and $\partial_D \Omega$.
We conclude this subsection with an important remark that will be used later.

\begin{remark} \label{remarkregular}
Let $v \in \Esp$, $w \in \Esp^{reg}$, and let $\xi \in \partial E_h (v)$.
Then, from the definition of subdifferential and direct computation, one can check that the action of 
$\xi$ on $w$ coincides with the action of the Fr\'echet differential $\partial W_h (v)$ on $w$.
In formulas:
$$
\langle \xi, w \rangle_{\Esp} = \langle \partial W_h (v), w \rangle_{\Esp}.
$$
\end{remark}

We show now that the functional $E_h$ satisfies the assumptions of Theorem \ref{Teorexist}.

\subsection{Assumptions (J1)--(J4) are satisfied by $E_h$ and $A_h$} \label{conditions a0-a3}

First of all, we start by observing that condition (J1) is satisfied, 
by using standard arguments of calculus of variations.

\begin{proposition} 
The functional
$$
\mathcal{E}_h \ni v \longmapsto E_h (v) + |A_h v|^2_{\mathcal{F}_h}
$$
is coercive.
\end{proposition} 

\begin{proof}
Let $C > 0$ be fixed, and let $\left( v_k \right)_{k \in \mathbb{N}} \subset \mathcal{E}_h$ 
be a sequence such that
$$
E_h (v_k) + |A_h v_k|^2_{\mathcal{F}_h} \leq C.
$$
Then, recalling the expression of $E_h$ and thanks to Poincar\'e inequality \eqref{poincare}, we have 
$$
|v_k|^2_{\Esp} \leq C, 
$$
for some new constant, still denoted by $C$, depending on $\Omega, \partial_D \Omega$.
Then, there exists a subsequence $\left( v_{k_j} \right)_{j \in \mathbb{N}}$ and a function 
$v \in \mathcal{E}_h$ such that
$$
v_{k_j} \rightharpoonup v \qquad \text{ weakly in } \mathcal{E}_h.
$$
Since $\mathcal{E}_h$ is finite dimensional, this implies that 
$$
v_{k_j} \to v \qquad \text{ in } \mathcal{E}_h,
$$
and this concludes the proof.
\end{proof}
We now show that condition (J2) is satisfied.
\begin{proposition} \label{ConditionA1}
There exists $\eta > 0$ such that the function
$$
u \longmapsto E_h (u) + \eta |u|_{\Esp}^2
$$
is strictly convex.   
\end{proposition}

\begin{proof}
We divide the proof into several steps.

\vspace{.2cm}

\noindent
\textbf{Step 1}. We show that there exists $\mu > 0$ such that the function 
$p_{\mu} : \mathbb{R} \to \mathbb{R}$ given by
\begin{equation}\label{econv}
p_{\mu} (s) := g (|s|) + \mu s^2, \qquad s \in \mathbb{R},
\end{equation}
is strictly convex. To this aim, we need to find $\mu$ such that the second 
distributional derivative $p_{\mu}''$ of $p_{\mu}$ is a positive Radon measure.
Recalling the definition of $g$, we have 
$$
p_{\mu} (s) = 
\begin{cases}
 \vspace{.1cm}
|s| + \left( \mu - \frac{1}{2 R} \right) s^2 &\text{ if } 0 \leq |s| < R, \\
\frac{R}{2} + \mu s^2 &\text{ if }  |s| \geq R,
\end{cases}
$$
The distributional derivative $p_{\mu}'$ of $p_{\mu}$ is given by
$$
p_{\mu}' (s) =  
\begin{cases}
\vspace{.2cm}
-1 + \left( 2 \mu - \frac{1}{R} \right) s &\text{ if } - R < s < 0, \\
\vspace{.2cm}
1 + \left( 2 \mu - \frac{1}{R} \right) s&\text{ if } 0 < s < R, \\
\vspace{.2cm}
2 \mu s &\text{ if } |s| \geq R.
\end{cases}
$$
Note that $p_{\mu}' \in L^1_{loc} (\mathbb{R})$.
We can then calculate the second distributional derivative $p_{\mu}''$ of $p_{\mu}$, 
which is the Radon measure in $\mathbb{R}$ given by
$$
p_{\mu}'' = \left( 2 \mu - \frac{1}{R} \right) \mathcal{L}^1 \lfloor_{\left(- R , R \right)} 
+ 2 \mu  \mathcal{L}^1 \lfloor_{\left(- \infty, - R \right) \cup \left( R , \infty \right)}
+ 2 \delta_{0},
$$
where $\delta_0$ represents the Dirac measure concentrated at the origin.
Note that 
$$
p_{\mu}'' (B) \geq \left( 2 \mu - \frac{1}{R} \right) \mathcal{L}^1 (B), \qquad \text{ for every Borel set } B \subset \mathbb{R}. 
$$
Thus, if we choose $\mu$ such that
\begin{equation} \label{red}
\mu > \frac{1}{2R},
\end{equation}
$p_{\mu}''$ is a positive Radon measure on $\mathbb{R}$, 
and $p_{\mu}$ is convex.

\vspace{.2cm}

\noindent
\textbf{Step 2}. We show that the functional $\overline{E}_h: \mathcal{E}_h \to [0, \infty)$ given by
$$
\overline{E}_h (u):= E_h (u) + \mu \int_{\Gamma} |[u]|^2 \, d \mathcal{H}^1,
$$
is convex.
By the previous step, the function
$r_{\mu} : \mathcal{G}_h \to [0, \infty)$ defined as
$$
r_{\mu} ([u]) := \int_{\Gamma} g (|[u]| ) \, d \mathcal{H}^1 + \mu \int_{\Gamma} |[u]|^2 \, d \mathcal{H}^1
$$
is convex, where $\mathcal{G}_h$ is the subset of $L^2 (\Gamma)$ given by
$$
\mathcal{G}_h:= \{ [u] : u \in \mathcal{E}_h \}.
$$
Note now that
$$
\overline{E}_h (u) =  W_h (u) + r_{\mu} (|[u]|),
$$
where $W$ was defined in \eqref{elastic}.
From the fact that $W_h : \mathcal{E}_h \to [0, \infty)$ is convex, 
we then obtain that also $\overline{E}_h : \mathcal{E}_h \to [0, \infty)$ is convex.

\vspace{.2cm}

\noindent
\textbf{Step 3: conclusion.} 
By \cite[Lemma 5.3]{C}, there exists a constant $\overline{C} > 0$ such that
\begin{equation} \label{semdefpos}
\int_{\Gamma} |[u]|^2 \, d \mathcal{H}^1 \leq \overline{C} \, | u |^2_{\Esp}. 
\end{equation}
Taking $\eta > \mu \, \overline{C}$ we have 
$$
E_h (u) + \eta |u|^2_{\Esp} = \underbrace{E_h (u) + \mu \int_{\Gamma} |[u]|^2 \, d x_2}_{\overline{E}_h (u)}
+ \mu \underbrace{\left( \overline{C} \, | u |^2_{\Esp} 
- \int_{\Gamma} |[u]|^2 \, d x_2 \right)}_{\widetilde{E}_h (u)}
+ ( \eta - \mu \, \overline{C} ) |u|^2_{\Esp}.
$$
We have already proven that $\overline{E}_h$ is convex.
On the other hand, $\widetilde{E}_h$ is a quadratic form which is positive semidefinite
by \eqref{semdefpos}, and thus is convex.
Since the remaining term $( \eta - \mu \, \overline{C} ) |u|^2_{\Esp}$
is strictly convex, this concludes the proof of (J2).
\end{proof}

Before passing to the proof of (J3) we need some preliminary results.
First, we make a few remarks on the regularity of the elastic part 
and on the crack part of the energy functional.
\begin{remark} \label{lip1}
Note that $W_h \in C^{1,1} (\mathcal{E}_h)$. 
In particular, $\partial W_h (\cdot): \Esp \to \mathcal{E}_h'$ 
is a single-valued Lipschitz function with Lipschitz constant $1$.
Indeed, we have
\begin{equation}\label{isoequiv}
\langle \partial W_h(w), v \rangle_{\mathcal{E}_h} =
\int_{\Omega \setminus \Gamma} \nabla w \cdot \nabla v \,\mathrm{d}x \qquad 
\text{ for every } w, v \in \mathcal{E}_h.
\end{equation}
Then, for every $w_1,w_2,v \in \mathcal{E}_h$ 
\begin{align*}
\left|  \partial W_h(w_1) - \partial W_h (w_2) \right|_{\mathcal{E}_h'}
&= \sup \left\{  \int_{\Omega \setminus \Gamma} \left( \nabla w_1 - \nabla w_2 \right) \cdot \nabla v  \,\mathrm{d}x, 
\quad v \in \mathcal{E}_h \text{ with } |v|_{\mathcal{E}_h} = 1 \right\} \\
&\leq \sup \left\{ 
\| \nabla w_1 - \nabla w_2 \|_{L^2 (\Omega \setminus \Gamma)}
\| \nabla v \|_{L^2 (\Omega \setminus \Gamma)}, 
\quad v \in \mathcal{E}_h \text{ with } |v|_{\mathcal{E}_h} = 1 \right\} \\
&\leq \sup \left\{ |w_1 - w_2|_{\mathcal{E}_h} |v|_{\mathcal{E}_h},
\quad v \in \mathcal{E}_h \text{ with } |v|_{\mathcal{E}_h} = 1 \right\} \\
&\leq |w_1 - w_2|_{\mathcal{E}_h}.
\end{align*}
Since $\partial W_h (0) = 0$, this implies
\begin{equation} \label{riso}
\left|  \partial W(w) \right|_{\mathcal{E}_h'} \leq \left|  w \right|_{\mathcal{E}_h}, 
\quad \text{ for every } w \in \Esp.
\end{equation}

\end{remark}

\begin{remark} \label{lipb}
From the previous remark, it also follows that
$$
\left|  \partial W_h(w_1) - \partial W_h (w_2) \right|_{\mathcal{E}_h'}
\leq \| \nabla w_1 - \nabla w_2 \|_{L^2 (\Omega \setminus \Gamma)}
\qquad \text{ for every } w_1, w_2 \in \mathcal{E}_h.
$$
\end{remark}

\begin{remark} \label{gh}
The functional $G_h : \mathcal{E}_h \to [0, \infty)$ defined as
$$
G_h (v):= \int_{\Gamma} g (|[v]|) \, d \mathcal{H}^1. 
$$
is globally Lipschitz continuous.
Indeed, for every $v_1, v_2 \in \mathcal{E}_h$ we have 
\begin{align*}
\left| G_h (v_1) - G_h (v_2) \right| 
&\leq \int_{\Gamma} \left| g (|[v_1]|) - g (|[v_2]|) \right| \, d \mathcal{H}^1
\leq \int_{\Gamma} \big| |[v_1]|  - |[v_2]| \big| \, d \mathcal{H}^1 \\
& \leq \int_{\Gamma} \big| [v_1]  - [v_2] \big| \, d \mathcal{H}^1
= \int_{\Gamma} \big| [v_1  - v_2] \big| \, d \mathcal{H}^1 \\
& \leq \left( H^1 (\Gamma) \right)^{1/2} \| [v_1  - v_2] \|_{L^2 (\Gamma)}
\leq \overline{C} \left( H^1 (\Gamma) \right)^{1/2} | v_1  - v_2 |_{\mathcal{E}_h}, 
\end{align*}
where $\overline{C}$ is given by \eqref{semdefpos}, 
and we used the fact that $\| g' \|_{L^{\infty} ([0, \infty))} = 1$.
\end{remark}

Next proposition shows condition (J3).
\begin{proposition}
$E_h$ satisfies condition (J3).
\end{proposition}

\begin{proof}
Note that 
$$
E_h (u) = W_h (u) + G_h (u), 
$$
where $G_h: \Esp \to [0, \infty)$ is defined in Remark~\ref{gh}.
Let now $v \in \mathcal{E}$. 
By \eqref{numero2}, every $\xi \in \partial J (v)$ can be written as
$$
\xi = \xi_1 + \xi_2, 
$$
where $\xi_1 \in \partial W_h (v)$, and $\xi_2 \in \partial G_h (v)$.
By Remark~\ref{lipb} we have 
$$
\left|  \xi_1 \right|_{\mathcal{E}_h'}
\leq \| \nabla v \|_{L^2 (\Omega \setminus \Gamma)}
\leq 1+ \| \nabla v \|^2_{L^2 (\Omega \setminus \Gamma)}
\leq 2 (1 + E_h (v) ).
$$
On the other hand, thanks to Remark~\ref{gh} $G_h$ is globally Lipschitz continuous
with Lipschitz constant $\overline{C} \left( H^1 (\Gamma) \right)^{1/2}$. 
Therefore, by \eqref{numero}
$$
|\xi_2|_{\mathcal{E}'_h} \leq \overline{C} \left( H^1 (\Gamma) \right)^{1/2}.
$$
Thus, 
$$
| \xi |_{\mathcal{E}'_h} \leq | \xi_1 |_{\mathcal{E}'_h} + | \xi_2 |_{\mathcal{E}'_h} 
\leq 2 E_h (v) + 2 + \overline{C} \left( H^1 (\Gamma) \right)^{1/2}.
$$
\end{proof}
Next lemma will be used to prove (J4), and gives a bound on the norm of a regular critical point, 
in terms of its trace on $\partial_D \Omega$.
\begin{lemma} \label{lemmaw}
Let $w \in \Esp^{reg}$, let $f \in \mathcal{F}_h$ be such that $A_h w = f$, 
and suppose $\partial W_h (w) \in \textnormal{ran} (A^*)$.
Then, there exists a positive constant $C = C (\Omega, \gamma)$ such that
\begin{equation}\label{vleqf1}
|w|_\Esp \leq C |f|_{\mathcal{F}_h}.
\end{equation}
\end{lemma}
\begin{proof}
By \eqref{poincare} and \eqref{isoequiv}
\begin{align} 
|w|_{\Esp}^2 &= \| w \|_{L^2 (\Omega)}^2 + \| \nabla w \|_{L^2 (\Omega \setminus \Gamma)}^2
\leq C \left( \left| A_h w \right|_{\mathcal{F}_h}^2 
+ \| \nabla w \|_{L^2 (\Omega \setminus \Gamma)}^2 \right) \label{amaro} \\
&= C \left[ | f |^2_{\mathcal{F}_h} + \langle \partial W_h (w) , w \rangle_{\Esp} \right], \label{amaro2}
\end{align}
where $C$ denotes different constants, depending only on $\Omega$ and $\partial_D \Omega$.
Let now $q \in \mathcal{F}_h$ be such that $A^*_h q = \partial W_h (w)$.
Then,
$$
\left|  \partial W_h (w) \right|_{\mathcal{E}_h'} = | A^*_h q |_{\Esp'} \geq \gamma | q |_{\mathcal{F}_h}.
$$
Thus, taking into account \eqref{riso} we have 
\begin{align*}
\langle \partial W_h (w) , w \rangle_{\Esp} 
&= \langle A^*_h q , w \rangle_{\Esp}
= \langle q , A_h w \rangle_{\mathcal{F}_h}
= \langle q , f \rangle_{\mathcal{F}_h}
\leq | q |_{\mathcal{F}_h} |f|_{\mathcal{F}_h} \\
&\leq \frac{1}{\gamma} \left|  \partial W_h (w) \right|_{\mathcal{E}_h'} |f|_{\mathcal{F}_h}
\leq \frac{1}{\gamma} \left|  w  \right|_{\mathcal{E}_h} |f|_{\mathcal{F}_h}.
\end{align*}
Using \eqref{amaro2}, we obtain
\begin{align*}
|w|_{\Esp}^2 \leq C 
\left[ | f |^2_{\mathcal{F}_h} + \frac{1}{\gamma} \left|  w  \right|_{\mathcal{E}_h} |f|_{\mathcal{F}_h} \right].
\end{align*}
From the previous relation, the thesis follows using Young inequality.
\end{proof}
We can finally prove (J4). 
\begin{proposition} 
Let $v_1,v_2, \overline{v} \in \Esp$ and $q_1,q_2, f_1, f_2 \in \mathcal{F}_h$ be such that
$$
A_h v_i = f_i, \qquad A^*_h q_i \in \partial  (E_h )_{\eta, \overline{v}} (v_i), \qquad i = 1, 2,
$$
where $\eta$ is given by Proposition \ref{ConditionA1}.
Then, there exists $C >0$, depending only on $\Omega$ and $\partial_D \Omega$ such that
$$
\langle q_1 - q_2 , A_h v_1 - A_h v_2 \rangle_{\mathcal{F}_h} 
\leq C |v_1 - v_2 |_{\mathcal{E}_h}  | A_h v_1 - A_h v_2 |_{\mathcal{F}_h}.
$$
\end{proposition} 

\begin{proof}
Let $w$ be the unique solution of the following minimization problem:
\begin{equation}\label{minW}
w=\argmin_{v \in \mathcal{E}_h^{reg}}\left\{W_h (v):A_h (v)=f_1-f_2\right\}.
\end{equation}
By Remark~\ref{remarkregular} we have
$$
\langle \xi_i, w \rangle_{\Esp} = \langle \partial W_h(v_i), w \rangle_{\Esp}
\qquad i = 1, 2.
$$
Now, by definition of $( E_h )_{\eta, \overline{v}}$ and Remark~\ref{sumsubdiff}, 
there exist $\xi_i \in\partial E_h (v_i)$, with $i=1,2$ 
such that
$$
A^*_h q_i - 2 \eta (v_i - \overline{v})= \xi_i, \qquad i = 1,2.
$$
Subtracting term by term we obtain
$$
A^*_h (q_1 - q_2) - 2 \eta (v_1 - v_2)= \xi_1 - \xi_2.
$$
Thus, thanks to Remark~\ref{lip1}
\begin{align*}
&\langle q_1 - q_2, A_h v_1 - A_h v_2 \rangle_{\mathcal{F}_h}
= \langle q_1 - q_2, f_1 - f_2 \rangle_{\mathcal{F}_h} 
=\langle q_1 - q_2, A_h (w) \rangle_{\mathcal{F}_h} \\
&\hspace{.2cm}= \langle A^*_h (q_1 - q_2), w\rangle_{\mathcal{E}_h}
= \langle \xi_1 - \xi_2 , w\rangle_{\Esp}
+ 2 \eta \langle  v_1 - v_2 , w\rangle_{\Esp} \\
&\hspace{.2cm}= \langle \partial W_h (v_1)- \partial W_h (v_2) ,w \rangle_{\Esp} 
+ 2 \eta \langle v_1 - v_2,w \rangle_{\Esp} \\
&\hspace{.2cm}\leq (1+2\eta) |w|_\Esp|v_1-v_2|_\Esp 
\leq C ( 1+2\eta )|f_1-f_2|_{\mathcal{F}_h}|v_1-v_2|_\Esp,
\end{align*}
where we also used the fact that $w$ satisfies the assumptions
of Lemma \ref{lemmaw} with $f = f_1 - f_2$.
\end{proof}

We conclude this section with a remark, which clarifies why we didn't directly prove the main theorem
of the paper in the infinite dimensional setting.

\begin{remark} \label{no-infinitedim}
As already explained in the Introduction, the proof of the main result of the paper (Theorem \ref{Teorexist}) 
is given only in a finite dimensional setting. 
This is due to the fact that, in the infinite dimensional case, 
the subdifferential is in general not closed with respect to the weak convergence in the domain of the energy. 
Such a difficulty could be overcome by requiring that the energy functional has compact sublevels, 
an assumption which is quite common in literature. 
In the cohesive fracture model case, this would amount to choosing
$L^2(\Omega)$ as domain of the energy , and considering the Dirichlet linear constraint 
as being encoded by an {\it unbounded} densely defined surjective linear operator 
$A:L^2(\Omega) \to H^{\frac12}(\partial_D\Omega)$. This does not affect neither \eqref{gamma} nor hypothesis (J1), which are still satisfied with minor modifications. 
Even condition (J2) can be proved with a little bit more of effort. 
On the other hand, the key conditions (J3) and  (J4), 
which we need in order to cope with the cohesive fracture energy in our model, would no longer hold true. 
This motivates our choice of first dealing with a finite dimensional setting, 
and then extend the results with a problem-specific technique. 
\end{remark}

\section{Recovering an approximable quasistatic evolution} \label{hto0}

We show now that the existence of a quasistatic evolution for the functional $E$, 
in the sense of \cite{C} and Section \ref{recallfract}, 
can be recovered from a discrete quasistatic evolution for $E_h$, when the parameter 
$h$ controlling the mesh size tends to $0$.
Before stating the main theorem of the section we need some notation.
We set
\begin{equation*} 
D := \left\{ h > 0 : h = \frac{\ell}{N} \text{ for some } N \in \mathbb{N} \right\}
\end{equation*}
with $\ell$ as in \ref{sec:discrete}.
Let $\omega \in W^{1,2} ([0,T], H^1(\Omega))$, 
and let $u_0$ be a contrained critical point of $E$ at time $0$, under the constraint
$u_0 = \omega (0)$ on $\partial_D \Omega$.
By \cite{quarteroni08}, there exists a sequence 
$\left( \omega_h \right)_{h \in D} \subset W^{1,2} ([0,T], H^1(\Omega))$ 
such that $\omega_h \in W^{1,2} ([0,T], \Esp)$ for every $h \in D$ and
\begin{equation} \label{approxbound}
\omega_h \stackrel{h \to 0^+}{\longrightarrow} \omega \qquad \text{ in } W^{1,2} ([0,T], H^1(\Omega)).
\end{equation}
For every $t \in [0,T]$, we define $f_h (t) := A_h \omega_h (t)$, 
we will denote by $f(t)$ the trace of $\omega (t)$ on $\partial_D \Omega$.
Again by \cite{quarteroni08}, there exists a sequence 
$\left( u_{0,h} \right)_{h \in D} \subset H^1(\Omega \setminus \Gamma)$ 
such that 
$$
u_{0,h} \in \Esp \qquad  \text{ with } \qquad A_h u_{0,h} = f_h (0) \qquad  \text{ for every } h \in D,
$$
and
\begin{equation} \label{approxbound2}
u_{0,h} \stackrel{h \to 0^+}{\longrightarrow} u_0 
\qquad \text{ in } H^1(\Omega \setminus \Gamma).
\end{equation}

\begin{remark}
Let $h \in D$ be fixed.
Since in general $u_{0,h}$ is not a critical point of $E_h$ at time $0$, 
it is not possible to consider an approximable quasistatic evolution 
with initial condition $u_{0,h}$ and constraint $f_h$.
We can, however, modify Definition \ref{evolution} in such a way 
that no critical point condition is required at the initial time $0$, 
as clarified below.
For our purposes, this is still sufficient.
Indeed, as stated in Theorem \ref{quasist} below, the critical point condition at $t=0$ 
is recovered when passing to the limit $h \to 0$.
\end{remark}

\begin{definition} \label{newquasist}
Let $h \in D$, $\delta \in (0,1)$, and let $u_{0,h}$ and $f_h$ be given above.
A \emph{discrete quasistatic evolution} with time step $\delta$, 
initial condition $u_{0,h}$, and constraint $f_h$ is a right-continuous function 
$u_{\delta, h}: [0, T] \to \Esp$ such that 
\begin{itemize}
 \item $u_{\delta, h} (0) = u_{0,h}$;
 \item $u_\delta$ is constant in $[0,T] \cap [i \delta, (i + 1) \delta)$ 
 for all $i \in \mathbb{N}_0$ with $i \delta \leq T$;
 
 \item $u_\delta ( i \delta )$ is a critical point of $E_h$ 
 on the affine space $\mathbf{A} (f_h (i \delta))$ for every 
 $i \in \mathbb{N}$ with $i \delta \leq T$.
\end{itemize}

\end{definition}

\noindent

\begin{remark}
Note that, in the definitions above, we do not require $u_{\delta, h} (0)$ 
to be a critical point of $E_h$ at $t=0$.
\end{remark}

\begin{definition} \label{newquasist2}
Let $h \in D$, and let $u_{0,h}$ and $f_h$ be given above.
We say that a measurable function $u_h:[0, T]\to \Esp$ 
is an \emph{approximable quasistatic evolution} 
with initial condition $u_{0,h}$ and constraint $f_h$, 
if for every $t \in [0,T]$ there exists a sequence $\delta_k \to 0^+$ (possibly depending on $t$)
and a sequence $\left( u_{\delta_k, h} \right)_{k \in \mathbb{N}}$
of discrete quasistatic evolutions with time step $\delta_k$, 
initial condition $u_{0,h}$, and constraint $f_h$, such that
\begin{equation*} 
\lim_{k\to +\infty} | u_{\delta_k, h}(t)- u_h(t) |_{\mathcal{E}_h} =0. 
\end{equation*}
\end{definition}
With this choice, an approximable quasistatic evolution still satisfies $u_h(0)=u_{0,h}$. On the other hand (see Theorem \ref{theoremh} below) in this case we will simply require that the stationarity condition for $u_h$ holds for every $t\in (0,T]$, while stationarity at $0$ will be recovered in the limit passage $h \to 0$.

Our goal is now proving the following version of \cite[Theorem 4.4]{C},
which is the main result of this section.

\begin{theorem} \label{quasist}
Let $\omega \in W^{1,2} ([0,T], H^1(\Omega))$, and let $u_0$ be a critical point 
of $E$ at time $0$ with $u_0 = \omega (0)$ on $\partial_D \Omega$. 
For every $h \in D$, let $\omega_h$ and $u_{0,h}$ be defined as above. 
Then, there exists a bounded measurable function $u:[0,T] \to H^1(\Omega \setminus \Gamma)$ 
with $u(0)=u_0$ such that the following properties are satisfied:

\vspace{.2cm}

\begin{itemize}
\item[(a)] {\it approximability}: for every $t \in [0,T]$ 
there exists a sequence $h_j \to 0^+$ such that
$$
u_{h_j} (t) \rightharpoonup u(t) \mbox{ weakly in } H^1(\Omega \setminus \Gamma)
$$
where, for every $j \in \mathbb{N}$, $u_{h_j}$ is an approximable quasistatic evolution
of $E_{h_j}$ with initial condition $u_{0, h_j}$ and constraint $f_{h_j}$, 
see Definition \ref{newquasist};

\vspace{.2cm}

\item[(b)] {\it stationarity}: for every $t \in [0,T]$ the function $u(t)$ is a critical point of $E$ at time $t$ under the constraint $u(t)=\omega(t)$ on $\partial_D \Omega$;

\vspace{.2cm}

\item[(c)] {\it energy inequality}: for every $t \in [0,T]$
\begin{equation*}
E(u(t)) \leq E(u(0)) + \int_0^t \int_{\Omega \setminus \Gamma} \nabla u(s) \cdot \nabla \dot \omega(s) \, \mathrm{d}x \mathrm{d}s.
\end{equation*}
\end{itemize}
\end{theorem}

Before proving Theorem \ref{quasist}, we need the following result, 
which is obtained by applying Theorem \ref{Teorexist} to $E_h$.
\begin{theorem}  \label{theoremh}
Let $h \in D$ be fixed, and let $u_{0,h}$ and $f_h$ given above. 
Then, there exists a measurable bounded mapping 
$u_h \colon [0,T] \to \mathcal{E}_h$ such that 

\begin{itemize}

\item[\textit{(a')}] $u_h (\cdot)$ is an approximable quasistatic evolution 
with initial condition $u_{0,h}$ and constraint~$f_h$;

\vspace{.2cm}

\item[\textit{(b')}] stationarity: for every $t \in (0,T]$ we have 

\begin{equation}\label{eq:crtpconddiscr+}
\int_{\Omega \setminus \Gamma} \nabla u_h (t) \cdot \nabla \psi \, \mathrm{d}x 
+ \int_\Gamma \left ([\psi] g'(|[u_h(t)]|)\operatorname{sign}([u_h(t)])1_{J_{u_h(t)}} 
+ |[\psi]| 1_{J_{u_h(t)}^c} \right ) d \mathcal H^{d-1} \geq 0,
\end{equation}
for all $\psi \in \Esp$ with $\psi = 0$ on $\partial_D \Omega$.

\vspace{.2cm}

\item[\textit{(c')}] {\it energy inequality}: The function 
$s \longmapsto \int_{\Omega \setminus \Gamma} 
\nabla u_h (s) \cdot \nabla \dot \omega_h(s) \, \mathrm{d}x$ belongs to $L^1 (0,T)$ and
\begin{equation*} 
E_h (u_h(t)) \leq E (u_{0,h}) + \int_0^t \int_{\Omega \setminus \Gamma} 
\nabla u_h (s) \cdot \nabla \dot \omega_h (s) \, \mathrm{d}x \mathrm{d}s \qquad \textnormal{ for every } t \in [0,T].
\end{equation*}

\item[\textit{(d')}] {\it Uniform bound}: There exists a constant $\overline{C}_2$, 
independent of $h$, such that
\begin{equation}\label{compattezza}
\|u_h (t)\|_{\mathcal{E}_h}\le \overline{C}_2 \qquad \text{ for every } t \in [0,T].
\end{equation}

\end{itemize}
\end{theorem}

\begin{remark}
As already observed, property (a') has to be intended 
in the sense of Definition~\ref{newquasist2}.
\end{remark}

\begin{proof}
As proven in the previous subsection, \eqref{gamma} and assumptions (J1)--(J4) are satisfied.
We now need to check that the proof of Theorem~\ref{Teorexist} can be adapted
to the present situation, where Definition~\ref{newquasist} and Definition~\ref{newquasist2}
substitute Definition~\ref{defdiscr} and Definition~\ref{evolution}, respectively.

\vspace{.2cm}

\noindent
\textbf{Step 1:} We construct a discrete quasistatic evolution.
Let $\delta \in (0,1)$ be a fixed time step, and let $i \in \mathbb{N}$ with $i \delta \leq T$.
We set $u^0_h := u_{0,h}$ while, for $i \geq 2$,   
we suppose that $u^{i-1}_h \in \Esp$ is a critical point of $E_h$ on the affine space 
$\mathbf{A} (f_h ((i-1)\delta))$. 
Then, analogously to \eqref{algobis}, we define the sequence 
$\left( u^i_{h,j} \right)_{j \in \mathbb{N}_0}$ by setting $u^i_{h,0} := u^{i-1}_h$ and
\begin{equation*} 
u^i_{h,j} := \argmin_{ A_h v = f_h (i \delta)} \{ E_h (v)  + \eta |v - u^i_{h,j-1} |^2_{\Esp} 
\, : \,  v \in \Esp \} \qquad \text{ for every } j \in \mathbb{N}.
\end{equation*}
Let now $i \in \mathbb{N}$ with $i \delta \leq T$.
One can check that Lemma \ref{lemmajNew} and Corollary \ref{cor} still hold true. 
In particular, for every $i \in \mathbb{N}$ with $i \delta \leq T$ we can find 
a function $u^i_{h} \in \Esp$ with $A_h u^i_{h} = f_h (i \delta)$ such that, up to subsequences, 
$$
\lim_{j \to \infty} u^i_{h,j} = u^i_{h} \qquad \text{ in } \Esp.
$$
Moreover, the function $u_{h,\delta} : [0,T] \to \mathcal{E}_h$ defined as
\begin{equation*} 
u_{h,\delta} (t) : = u^i_{h} \quad \text{ for every }t \in [0,T] \cap [i \delta, (i + 1) \delta), 
\quad \text{ for every } i \in \mathbb{N}_0 \text{ with } i \delta \leq T,
\end{equation*}
is a discrete quasistatic evolution with time step $\delta$, 
initial condition $(u_0 )_{h}$, and constraint $f_h$, in the sense of Definition~\ref{newquasist}.

At this point, we need to check that Theorem~\ref{teor18} can still be proven.
In particular, we want to define a function $q_{h,\delta} : [0,T] \to \mathcal{F}_h$
such that an approximate energy inequality (as (iii) of Theorem~\ref{teor18}) holds true.
The main problem consists in defining $q_{h,\delta}$ in the interval $[0, \delta)$. 
Indeed, since $u_{h,\delta} (0)$ is not a critical point, 
we do not have a natural choice available.
We then modify the proof of Theorem~\ref{teor18} in the following way.

Let $i = 1$. Since $(E_h)_{\eta, (u_0 )_{h}}$ is strictly convex, 
for every $\xi \in \partial (E_h)_{\eta, (u_0 )_{h}} (u^1_{h,1})$, we have  
$$
(E_h)_{\eta, (u_0 )_{h}} (v) \geq (E_h)_{\eta, (u_0 )_{h}} (u^1_{h,1}) 
+ \langle \xi , v - u^1_{h,1} \rangle_{\mathcal{E}_h}
\qquad \text{ for every } v \in \mathcal{E}_h.
$$
In particular, choosing 
$v = (u_0 )_{h}$ and recalling the definition of $(E_h)_{\eta, (u_0 )_{h}}$ we have 
\begin{equation} \label{step1b}
E_h ((u_0 )_{h}) \geq E_h (u^1_{h,1}) + \eta |u^1_{h,1} - (u_0 )_{h} |_{\mathcal{E}_h}^2
+ \langle \xi , (u_0 )_{h} - u^1_{h,1} \rangle_{\mathcal{E}_h} \quad
\forall \, \xi \in \partial (E_h)_{\eta, (u_0 )_{h}} (u^1_{h,1}).
\end{equation}
Recall now that $u^1_{h,1}$ is the global minimizer of $(E_h)_{\eta, (u_0 )_{h}}$ 
on $\mathbf{A} (f_h (\delta))$.
Therefore, there exists $r^1 \in \mathcal{F}_h$ such that 
$A^*_h r^1 \in \partial (E_h)_{\eta, (u_0 )_{h}} (u^1_{h,1})$.
Therefore, by \eqref{step1b},
$$
E_h ((u_0 )_{h}) \geq E_h (u^1_{h,1}) + \eta |u^1_{h,1} - (u_0 )_{h} |_{\mathcal{E}_h}^2
+ \langle A^*_h r^1 , (u_0 )_{h} - u^1_{h,1} \rangle_{\mathcal{E}_h}.
$$
At this point we can finally define the function $q_{h,\delta}: [0,T] \to \mathcal{F}_h$.
Since $u_{h,\delta}$ is a discrete quasistatic evolution, for every $i \in \mathbb{N}$
with $i \delta \leq T$ there exists $q^i \in \mathcal{F}_h$ 
such that $A^* q^i \in \partial (E_h ) (u_{h,\delta} (i \delta))$. 
Then, we define $q_{\delta} : [0,T] \to \mathcal{F}_h$ as 
\begin{equation*}
q_\delta(t) : =
\begin{cases}
r^1 & t \in [0, \delta), \\
q^i & t \in [0,T] \cap [i \delta, (i + 1) \delta), 
\quad \hbox{ for } i \in \mathbb{N} \text{ with } i \delta \leq T.
\end{cases}
\end{equation*}
Thus, repeating the proof of Theorem~\ref{teor18} for $i =1$ we have
\begin{align*}
&\eta |u^1_{h,1} - (u_0 )_{h} |_{\mathcal{E}_h}^2 + E_h (u^1_{h,1}) - E_h ((u_0 )_{h}) \\
&\leq \langle r^1 - q_\delta(0) , A_h u^1_{h,1} - A_h (u_0 )_{h} \rangle_{\mathcal{F}_h} 
+ \int_{0}^{\delta}  \langle q^{\delta} (0) ,  \dot{f}_h (s) \rangle_{\mathcal{F}_h} \, \mathrm{d}s \\
&= \int_{0}^{\delta}  \langle q^{\delta} (0) ,  \dot{f}_h (s) \rangle_{\mathcal{F}_h} \, \mathrm{d}s.
\end{align*}
Recalling that Lemma \ref{lemmajNew} holds true also in this case, we have
\begin{align*} 
E_h (u_{h, \delta} (\delta)) &\leq E_h (u^1_{h,1}) \leq E_h ((u_0 )_{h}) 
+ \int_{0}^{\delta}  \langle q^{\delta} (s) ,  \dot{f}_h (s) \rangle_{\mathcal{F}_h} \, \mathrm{d}s.
\end{align*}
Since for $i \geq 2$ the proof of \eqref{try} can be repeated with no modifications, 
this shows Step 1 of the proof of Theorem~\ref{teor18}.

In order to prove condition (ii) of Theorem~\ref{teor18}, we proceed in the following way.
By the minimality property of $u^1_{h,1}$, we have 
$$
E_h (u^1_{h,1}) \leq E_h (u^1_{h,1})  + \eta |u^1_{h,1} - (u_0 )_{h} |_{\mathcal{E}_h}^2
\leq  E_h ( \omega_h (\delta) )  + \eta |\omega_h (\delta) - (u_0 )_{h} |_{\mathcal{E}_h}^2 \leq C, 
$$
where, by \eqref{approxbound} and \eqref{approxbound2}, $C$ is a constant depending only on 
$\eta$, $\omega$ and $u_0$ (and not on $h$ and $\delta$).
Then, thanks to \eqref{approxbound}
$$
E_h (u^1_{h,1}) + |A_h u^1_{h,1}|^2_{\mathcal{F}_h}
= E_h (u^1_{h,1}) + |f_h (\delta)|^2_{\mathcal{F}_h} \leq C, 
$$
where $C$ is again a (possibly different) constant, depending only on 
$\eta$, $\omega$ and $u_0$ (and not on $h$ and $\delta$).
Thus, by condition (J1) and by equicoercivity of the family of functionals 
$\left( E_h \right)_{h \in D}$ we have that 
$$
| u^1_{h,1} |_{\mathcal{E}_h} \leq C
$$
for some constant independent of $h$.
Therefore, thanks to property (J3) and 
recalling that $A^*_h r^1 \in \partial (E_h)_{\eta, (u_0 )_{h}} (u^1_{h,1})$
\begin{align*}
|r^1|_{\mathcal{F}_h} \leq \frac{1}{\gamma} | A^*_h r^1 |_{\mathcal{E}_h}
\leq \frac{L}{\gamma}\left( E_h (u^1_{h,1}) + 1\right) 
+ \frac{2 \eta}{\gamma} | u^1_{h,1} - (u_0 )_{h}|_{\mathcal{E}_h} \leq C, 
\end{align*}
for some constant $C$ independent of $h$.
From this, in particular, we obtain that 
$$
| q_{\delta} (0) |_{\mathcal{F}_h} = |r^1|_{\mathcal{F}_h} \leq C
\leq  \frac{L}{\gamma} E_h (u_{h, \delta} (0)) + C, 
$$
which gives the equiboundedness of $| q_{\delta} (t) |_{\mathcal{F}_h}$ in the interval $[0, \delta)$.
At this point, the proof of Theorem~\ref{teor18} can be repeated without any modifications.

\vspace{.2cm}

\noindent
\textbf{Step 2:} We apply Theorem~\ref{Teorexist}.

By Step 1, properties (a)--(c) of Theorem~\ref{Teorexist} hold true.
In particular, (a) implies (a').
By (b) of Theorem~\ref{Teorexist}, there exists 
a bounded measurable function $q_h : (0,T] \to \mathcal{F}_h$ such that
\begin{equation} \label{aqh}
A^* q_h (t) \in \partial E_h (u_h (t)) \qquad \text{ for every }t \in (0,T].
\end{equation}
Let now $t \in (0,T]$ be fixed. Thanks to Remark~\ref{liminf}, we have
$$
0 \leq \liminf_{\varepsilon \to 0^+} \frac{E_h (u_h (t) + \varepsilon w) - E_h (u_h (t))}{\varepsilon}
\qquad \text{ for every } w \in \text{ker} ( A^*_h).
$$
A careful inspection of the proof of \cite[Proposition 3.1]{C} 
shows that last inequality implies (b').
Let us now show (c'). 
From (c) of Theorem~\ref{Teorexist}, the function $s \mapsto \langle q_h (s) ,  
\dot{f}_h (s) \rangle_{\mathcal{F}_h}$ belongs to $L^1 (0,T)$,
and for every $t \in [0,T]$ we have 
\begin{equation*} 
E_h (u_h (t)) \leq E ((u_0)_h) 
+ \int_{0}^{t}  \langle q_h (s) ,  \dot{f}_h (s) \rangle_{\mathcal{F}_h} \, \mathrm{d}s.
\end{equation*}
Recalling that $f_h (s) = A_h \omega_h (s)$ and that the linear operator $A_h$ 
is independent of time, we have 
\begin{align*}
E_h (u_h (t)) &\leq E ((u_0)_h) 
+ \int_{0}^{t}  \langle q_h (s) ,  \dot{f}_h (s) \rangle_{\mathcal{F}_h} \, \mathrm{d}s \\
&= E ((u_0)_h) 
+ \int_{0}^{t}  \langle q_h (s) ,  A_h \dot{\omega}_h (s) \rangle_{\mathcal{F}_h} \, \mathrm{d}s \\
&= E ((u_0)_h) 
+ \int_{0}^{t}  \langle A^*_h q_h (s) ,  \dot{\omega}_h (s) \rangle_{\mathcal{E}_h} \, \mathrm{d}s.
\end{align*} 
By \eqref{aqh}, for every $s \in (0,T)$ we have $A^* q_h (s) \in \partial E_h (u_h (s))$. 
Since $\dot{\omega}_h (s) \in \mathcal{E}_h^{\text{reg}}$ for every $s \in (0,T)$, 
by Remark \ref{remarkregular} we have
$$
\langle A^*_h q_h (s) ,  \dot{\omega}_h (s) \rangle_{\mathcal{E}_h} 
= \int_{\Omega \setminus \Gamma} \nabla u_h (s) \cdot \nabla \dot{\omega}_h (s) \, \mathrm{d}x
\qquad \text{ for every } s \in (0,T).
$$
Therefore, 
\begin{align*}
E_h (u_h (t)) &\leq E ((u_0)_h)
+ \int_{0}^{t}  \int_{\Omega \setminus \Gamma} \nabla u_h (s) 
\cdot \nabla \dot{\omega}_h (s) \, \mathrm{d}x \, \mathrm{d}s, 
\end{align*} 
which gives (c).
Finally, property (d') directly follows from Remark~\ref{boundevolution}
and Remark~\ref{dependence}.
\end{proof}
We can now pass to the limit as $h \to 0^+$.

\begin{proof}[Proof of Theorem \ref{quasist}]
Let $t \in [0,T]$.
We will use an argument similar to the one  in the proof of Theorem \ref{Teorexist}.

\vspace{.2cm}

\noindent
\textbf{Step 1: Proof of (a) and (c).}

\vspace{.2cm}
\noindent
First of all, we fix the subsequence $\left( h_{k} \right)_{k \in \mathbb{N}}$ given by
$$
h_{k} := \frac{\ell}{2^k}, \qquad \qquad k \in \mathbb{N},
$$
so that 
\begin{equation} \label{inclusion}
\mathcal{E}_{h_{l}} \subset \mathcal{E}_{h_{m}} \quad \text{ for every } m > l .
\end{equation}
By \eqref{approxbound}, we have 
$$
\dot{\omega}_{h_k} \to  \dot{\omega}
\quad \text{ strongly in } L^2 ([0,T]; H^1 (\Omega)) \quad \text{ as } k \to \infty.
$$
Thus, there exists a set $\Lambda_2 \subset [0,T]$ 
with $\mathcal{L}^1 (\Lambda_2) = 0$ such that  
$\dot{\omega} (t)$ is well defined for every $t \in [0,T] \setminus \Lambda_2$ and
\begin{equation} \label{strongtime}
\dot{\omega}_{h_{k}} (t) \to  \dot{\omega} (t)
\quad \text{ strongly in } H^1 (\Omega) \quad \text{for every } t \in [0,T] \setminus \Lambda_2
\quad \text{ as } k \to \infty.
\end{equation}
For every $h \in D$, let $u_h: [0,T] \to \mathcal{E}_h$ be given
by Theorem~\ref{theoremh}. We define
$$
\theta_k (t) : = 
\begin{cases}
\int_{\Omega \setminus \Gamma} \nabla u_{h_k} (t) 
\cdot \nabla \dot{\omega}_{h_k} (t) \, \mathrm{d}x & \text{ for every } 
t \in [0,T] \setminus \Lambda_2, \vspace{.1cm} \\
0 & \text{ for every } t \in \Lambda_2,
\end{cases}
$$
and 
$$
\theta (t):= \limsup_{k \to \infty} \theta_k (t) \quad \text{ for every } t \in [0,T].
$$
By definition of $\theta$, for every $t \in [0,T]$ we can extract 
a subsequence $\left( h_{k_j} \right)_{j \in \mathbb{N}}$ (possibly depending on $t$) such that
$$
\theta (t) = \lim_{j \to \infty} \theta_{k_j} (t) \quad \text{ for every } t \in [0,T].
$$
By \eqref{compattezza}, for every $t \in [0,T]$ we can extract a further subsequence
(not relabelled) such that
\begin{equation} \label{weakly}
u_{h_{k_j}} (t) \rightharpoonup u (t) \quad \text{weakly in } H^1 (\Omega \setminus \Gamma) \quad
\text{ as } j \to \infty.
\end{equation}
for some $u (t) \in H^1 (\Omega \setminus \Gamma)$ 
with $\| u (t) \|_{H^1 (\Omega \setminus \Gamma)} \leq \overline{C}_2$.
By repeating what was done in the proof of Theorem \ref{Teorexist}, we can show that  
the subsequence $\left( k_j \right)_{j \in \mathbb{N}}$ can be chosen in such a way that the map 
$u :[0,T] \to \mathcal H^1 (\Omega \setminus \Gamma)$ is measurable, 
and this shows (a).

Let us now show the energy inequality.
By \eqref{strongtime} and \eqref{weakly} we have that, for every $t \in [0,T] \setminus \Lambda_2$, 
$$
\theta (t) = \limsup_{k \to \infty} \theta_k (t) = \lim_{j \to \infty} \theta_{k_j} (t)
= \lim_{j \to \infty} \int_{\Omega \setminus \Gamma} \nabla u_{h_{k_j}} (t) 
\cdot \nabla \dot{\omega}_{h_{k_j}} (t) \, \mathrm{d}x
= \int_{\Omega \setminus \Gamma} \nabla u (t) 
\cdot \nabla \dot{\omega} (t) \, \mathrm{d}x.
$$
In order to prove that $\theta \in L^1 (0,T)$ we first observe that 
$\theta$ is measurable, since it is the $\limsup$ of a sequence of measurable functions.
Moreover, we have 
\begin{align*}
\int_0^T | \theta (t) | \, dt 
&= \int_0^T \left| \int_{\Omega \setminus \Gamma} \nabla u (t) 
\cdot \nabla \dot{\omega} (t) \, \mathrm{d}x \right| \, dt
\leq \int_0^T \| \nabla u (t) \|_{L^2 (\Omega \setminus \Gamma)} 
\| \nabla \dot{\omega} (t) \|_{L^2 (\Omega \setminus \Gamma)} \, dt \\
&\leq \overline{C}_2 \int_0^T \| \nabla \dot{\omega} (t) \|_{L^2 (\Omega \setminus \Gamma)} \, dt 
\leq \overline{C}_2 \sqrt{T} \, \| \dot{\omega} \|_{L^2 ((0,T); H^1 (\Omega))}.
\end{align*}
By (c') of Theorem \ref{theoremh} we have, for every $j \in \mathbb{N}$ and for every $t \in [0,T]$
\begin{align} 
E (u_{h_{k_j}} (t)) &= E_{h_{k_j}} (u_{h_{k_j}} (t)) \leq E (u_{0,h_{k_j}}) 
+ \int_0^t \int_{\Omega \setminus \Gamma} 
\nabla u_{h_{k_j}} (s) \cdot \nabla \dot \omega_{h_{k_j}} (s) \, \mathrm{d}x \mathrm{d}s .  \label{eq:enineq4}
\end{align}
Note now that the energy $E (\cdot)$ is lower semicontinuos w.r.t. 
weak convergence in $H^1 (\Omega \setminus \Gamma)$.
Moreover, since $u_{0,h_{k_j}} \to u_0$ 
strongly in $H^1 (\Omega \setminus \Gamma)$, we have 
$$
\lim_{j \to \infty} E (u_{0,h_{k_j}}) = E (u_0).
$$
Therefore, taking the limsup in $j$ of the \eqref{eq:enineq4}, and using Fatou's Lemma
\begin{align*}
E (u (t)) &\leq \liminf_{j \to \infty}
E (u_{h_{k_j}} (t)) \leq E (u_0) + \limsup_{j \to \infty} 
\int_0^t \int_{\Omega \setminus \Gamma} 
\nabla u_{h_{k_j}} (s) \cdot \nabla \dot \omega_{h_{k_j}} (s) \, \mathrm{d}x \mathrm{d}s \\
&\leq E (u_0) + \limsup_{k \to \infty} 
\int_0^t \int_{\Omega \setminus \Gamma} 
\nabla u_{h_{k}} (s) \cdot \nabla \dot \omega_{h_{k}} (s) \, \mathrm{d}x \mathrm{d}s \\
&\leq E (u_0) + \int_{0}^{t}  \limsup_{k \to \infty} \, 
\int_{\Omega \setminus \Gamma} \nabla u_{h_{k}} (s) \cdot \nabla \dot \omega_{h_{k}} (s) \, \mathrm{d}x \, \mathrm{d}s \\
&= E (u_0) + \int_{0}^{t} 
\int_{\Omega \setminus \Gamma} \nabla u (s) \cdot \nabla \dot \omega (s) \, \mathrm{d}x \, \mathrm{d}s,
\end{align*}
so that (c) follows.

We finally prove the stationarity. Since $u_h(0)=u_{0,h}\to u_0$ as $h \to 0$, we only have to prove the condition at a point $t\in (0,T]$.
Let $\psi \in H^1 (\Omega \setminus \Gamma)$ with $\psi = 0$ on $\partial_D \Omega$.
Then (see for instance \cite{quarteroni08}), we can find a sequence $\left( \psi_{h_{k_j}} \right)_{j \in \mathbb{N}}$ such that 
$$
\psi_{h_{k_j}} \to \psi \quad \text{ strongly in } H^1 (\Omega \setminus \Gamma)
\quad \text{ as } j \to \infty 
$$
and $\psi_{h_{k_j}} \in \mathcal{E}_{h_{k_j}}$  with $\psi_{h_{k_j}} = 0$ 
on $\partial_D \Omega$, for every $j \in \mathbb{N}$.
Note that, by \eqref{inclusion}, we have
$$
\psi_{h_{k_l}} \in \mathcal{E}_{h_{k_j}} \quad \text{with }\psi_{h_{k_l}} = 0 \text{ on } 
\partial_D \Omega \quad \text{ for every } j > l .
$$
Therefore, by \eqref{eq:crtpconddiscr+}  
\begin{align} 
&\int_{\Omega \setminus \Gamma} \nabla u_{h_{k_j}} (t) \cdot \nabla \psi_{h_{k_l}} \, \mathrm{d}x \nonumber \\
&\hspace{1cm}\geq \int_\Gamma \left ( - [\psi_{h_{k_l}}] g'(|[u_{h_{k_j}}(t)]|)
\operatorname{sign}([u_{h_{k_j}}(t)])1_{J_{u_{h_{k_j}}(t)}} 
- |[\psi_{h_{k_l}}]| 1_{J_{u_{h_{k_j}}(t)}^c} \right ) d \mathcal H^{d-1}, \label{eq:crtpconddiscr+2}
\end{align}
for every $j > l$.
By \eqref{weakly} we have 
\begin{equation} \label{one}
\lim_{j \to \infty} \int_{\Omega \setminus \Gamma} 
\nabla u_{h_{k_j}} (t) \cdot \nabla \psi_{h_{k_l}} \, \mathrm{d}x
= \int_{\Omega \setminus \Gamma} 
\nabla u (t) \cdot \nabla \psi_{h_{k_l}} \, \mathrm{d}x.
\end{equation}
Define now, for every $t \in [0,T]$ and for every $j > l$, 
the function $f_j (t): \Gamma \to \mathbb{R}$ as
$$
f_j (t) := - [\psi_{h_{k_l}}] g'(|[u_{h_{k_j}}(t)]|)
\operatorname{sign}([u_{h_{k_j}}(t)])1_{J_{u_{h_{k_j}}(t)}} 
- |[\psi_{h_{k_l}}]| 1_{J_{u_{h_{k_j}}(t)}^c}.
$$
We want to prove that for every $t \in [0,T]$
\begin{equation} \label{uno}
\liminf_{j \to \infty} f_j (t) \geq - [\psi_{h_{k_l}}] g'(|[u(t)]|)
\operatorname{sign}([u(t)])1_{J_{u(t)}} 
- |[\psi_{h_{k_l}}]| 1_{J_{u(t)}^c} \qquad \mathcal{H}^1\text{-a.e. in } \Gamma.
\end{equation}
Up to extracting a further subsequence, we can assume that
\begin{equation} \label{1l}
\liminf_{j \to \infty} f_j (t) = \lim_{j \to \infty} f_j (t) \qquad \mathcal{H}^1\text{-a.e. in } \Gamma,
\end{equation}
and 
\begin{equation} \label{2l}
\lim_{j \to \infty} [u_{h_{k_j}}(t)] = [u (t)] 
\qquad \mathcal{H}^1\text{-a.e. in } \Gamma.
\end{equation}
Now, let us fix $x \in J_{u(t)}$ such that \eqref{1l} and \eqref{2l} hold true.
Then, for $j \in \mathbb{N}$ large enough we have 
$$
x \in J_{u_{h_{k_j}}(t)} \qquad \text{ and } \qquad \operatorname{sign}([u_{h_{k_j}}(t)] (x)) 
=  \operatorname{sign}([u(t)] (x)).
$$
Therefore, 
\begin{align}
&\liminf_{j \to \infty} f_j (t) (x) = \lim_{j \to \infty} f_j (t) (x) \nonumber \\
&= \lim_{j \to \infty} - [\psi_{h_{k_l}}] (x) g'(|[u_{h_{k_j}}(t)] (x)| )
\operatorname{sign}([u_{h_{k_j}}(t)] (x)) 1_{J_{u_{h_{k_j}}(t)}} (x)
- |[\psi_{h_{k_l}}] (x)| 1_{J_{u_{h_{k_j}}(t)}^c} (x) \nonumber \\
&= - [\psi_{h_{k_l}}] (x) g'(|[u (t)] (x)| )
\operatorname{sign}([u (t)] (x)) 1_{J_{u (t)}} (x)
- |[\psi_{h_{k_l}}] (x)| 1_{J_{u (t)}^c} (x)  \label{due}
\end{align}
for $\mathcal{H}^1$-a.e. $x \in \Gamma \cap J_{u (t)}$.
If, instead, $x \in J_{u(t)}^c$, then recalling that $0 \leq g' \leq 1$ we have 
\begin{align}
&\liminf_{j \to \infty} f_j (t) (x) = \lim_{j \to \infty} f_j (t) (x) \nonumber\\
&= \lim_{j \to \infty} - [\psi_{h_{k_l}}] (x) g'(|[u_{h_{k_j}}(t)] (x)| )
\operatorname{sign}([u_{h_{k_j}}(t)] (x)) 1_{J_{u_{h_{k_j}}(t)}} (x)
- |[\psi_{h_{k_l}}] (x)| 1_{J_{u_{h_{k_j}}(t)}^c} (x) \nonumber  \\
&\geq - |[\psi_{h_{k_l}}] (x)|
= - |[\psi_{h_{k_l}}] (x)| 1_{J_{u (t)}^c} (x).  \label{tre}
\end{align}
Combining \eqref{due} and \eqref{tre} we obtain \eqref{uno}.
Thanks to \eqref{one} and \eqref{uno} we can 
pass to the limit in \eqref{eq:crtpconddiscr+2}, obtaining
\begin{align*} 
&\int_{\Omega \setminus \Gamma} 
\nabla u (t) \cdot \nabla \psi_{h_{k_l}} \, \mathrm{d}x 
= \lim_{j \to \infty} \int_{\Omega \setminus \Gamma} \nabla u_{h_{k_j}} (t) \cdot \nabla \psi_{h_{k_l}} \, \mathrm{d}x \\
&\geq \liminf_{j \to \infty} \int_\Gamma \left ( - [\psi_{h_{k_l}}] g'(|[u_{h_{k_j}}(t)]|)
\operatorname{sign}([u_{h_{k_j}}(t)])1_{J_{u_{h_{k_j}}(t)}} 
- |[\psi_{h_{k_l}}]| 1_{J_{u_{h_{k_j}}(t)}^c} \right ) d \mathcal H^{d-1} \\
&\geq \int_\Gamma \liminf_{j \to \infty}  \left ( - [\psi_{h_{k_l}}] g'(|[u_{h_{k_j}}(t)]|)
\operatorname{sign}([u_{h_{k_j}}(t)])1_{J_{u_{h_{k_j}}(t)}} 
- |[\psi_{h_{k_l}}]| 1_{J_{u_{h_{k_j}}(t)}^c} \right ) d \mathcal H^{d-1} \\
&\geq \int_\Gamma \left ( - [\psi_{h_{k_l}}] g'(|[u (t)]|)
\operatorname{sign}([u (t)])1_{J_{u (t)}} 
- |[\psi_{h_{k_l}}]| 1_{J_{u (t)}^c} \right ) d \mathcal H^{d-1}.
\end{align*}
Finally, passing to the limit as $l \to \infty$ we have 

\begin{align*} 
\int_{\Omega \setminus \Gamma} 
\nabla u (t) \cdot \nabla \psi \, \mathrm{d}x  
\geq \int_\Gamma \left ( - [\psi] g'(|[u (t)]|)
\operatorname{sign}([u (t)])1_{J_{u (t)}} 
- |[\psi]| 1_{J_{u (t)}^c} \right ) d \mathcal H^{d-1}, 
\end{align*}
and we conclude.

\end{proof}

\section{Numerical experiments}\label{sec:exp}

The scope of this section is to practically show that the procedure illustrated in the previous sections can be effectively implemented and produces the desired quasistatic evolution, according to the one described in \cite{C}.
We refer the reader to Section \ref{appl} for the notations used here.

\subsection{Numerical simulations in $1$ dimension}

We first analyze the results obtained for a one-dimensional problem, when $\Omega\subset\R$. 
Despite its simplicity, the one-dimensional setting allows us to give a detailed comparison 
between numerical results and analytic predictions, since 
in this case the explicit solutions of \eqref{eq:crptcond} are known. 
We consider the following geometry: 
$$
\Omega = [0,2\ell], \qquad \ell = 0.5, \qquad \Gamma = \{\ell\}, \qquad \partial_D \Omega = \{0,2\ell\}.
$$
We follow the evolution in the time interval $[0,T] = [0,1]$,  
and the external load applied to the endpoints $\partial_D \Omega = \{0,2\ell\}$ is given by
$$
\omega(t) (x) = 2(x-\ell)t, \qquad \text{ for every } x \in [0,2\ell] \text{ and } t \in [0,1].
$$
We uniformly discretize the domain into $2N = 80$ intervals, so that
the spatial discretization step is given by $h=\ell/N$.
Finally, we choose a time step $\delta = 0.02$, so that the total evolution 
is concluded after $50$ time steps. 
In our specific case, by Proposition~\ref{ConditionA1} and a direct calculation we have that the parameter $\eta$ in condition (J2) can be taken as
$$
\eta = \frac{1}{2R} + \max\{4,4\sqrt{\ell}\},
$$
where the constant $R$ is the one appearing in the definition of the function $g$, see \eqref{g}.

From a practical viewpoint, the computational time needed 
to solve the minimization problem \eqref{algobis} could grow without any control.
Hence, for any $i \in \{ 0, \ldots, 50 \}$ and $j \in \mathbb{N}_0$ fixed, 
we stop the minimization loop as soon as $\|A v - f (i \delta) \|<10^{-6}$, where $A$ is the trace operator and $f=A\omega$ (we omit the dependence of $A$ on $h$ to ease notation).
That is, $v^i_j$ is chosen in such a way that $\|A v^i_j - f (i \delta) \|<10^{-6}$.
Instead, we stop the external loop (that is, the limit of $v^i_j$ as $j \to \infty$), 
as soon as $\|v^i_j-v^i_{j-1}\|<10^{-13}$.
The main reason for these choices is that the quasistatic evolution generated by the algorithm
is extremely sensitive to any perturbation. 
Thus, a larger time step $\delta$, or a too badly approximated critical point 
at each time step, could lead to nonphysical results.

We observe that the evolutions discussed in \cite[Section 9]{C} depend
on  the size of the parameter $R$. 
Therefore, in order to compare our results with those in \cite{C}, we distinguish two cases.

\vspace{.2cm}

\noindent

\textbf{Case $R \ge 2\ell$.}
When $R$ is chosen large with respect to the size $2\ell$ of the elastic body, 
the evolution found by numerical simulations evolves along \textit{global} minimizers of the energy,
and we can observe the three phases of the cohesive fracture formation: non-fractured, pre-fractured
(that is when the opening of the crack is smaller than $R$ and cohesive forces appear), and 
completely fractured (when the opening of the crack is larger than $R$ and the cohesive forces
disappear), see Figure~\ref{cohes1d}. 
\begin{figure}[!ht]
\centering
  \includegraphics[height=.29\textwidth]{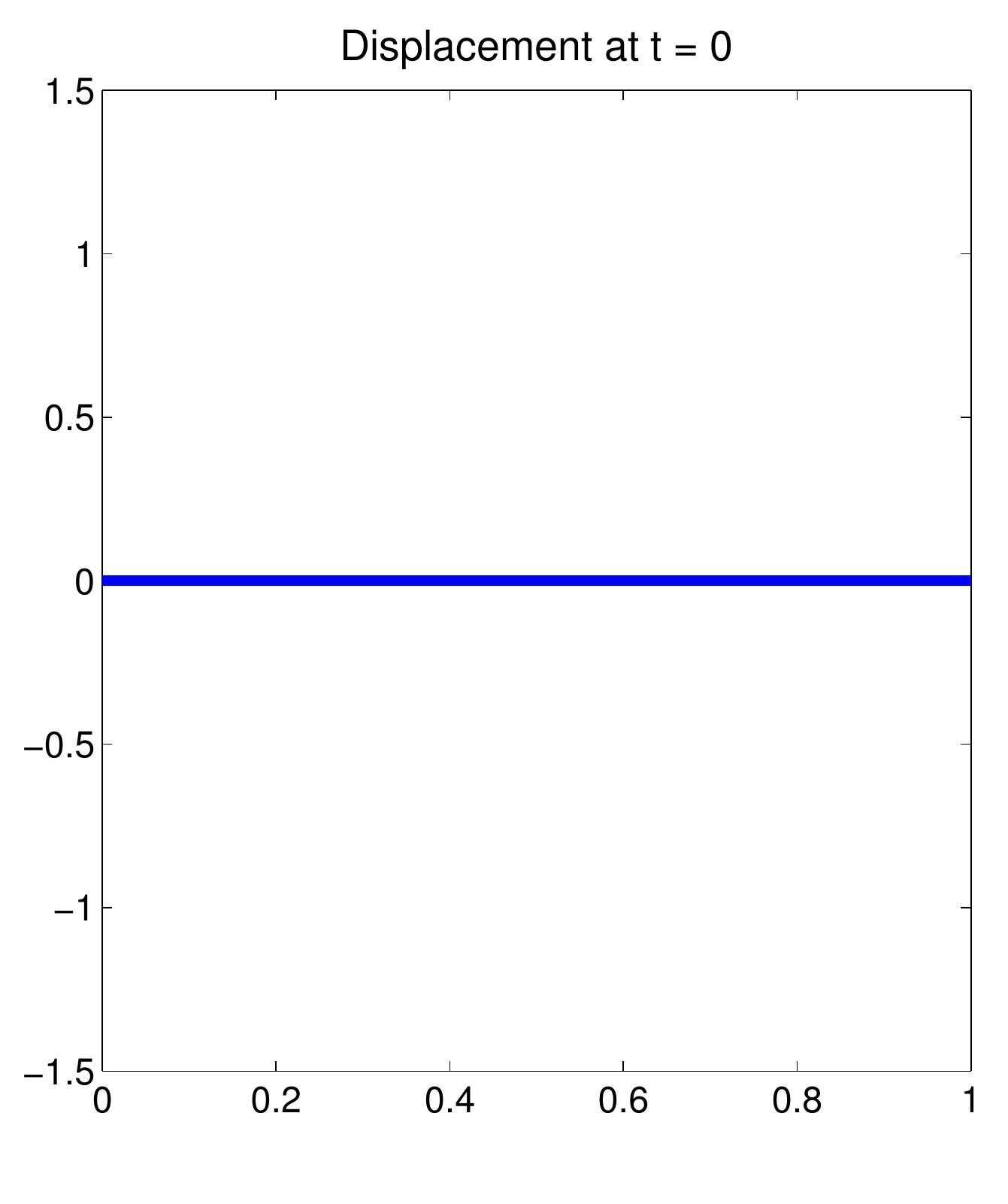}
 \includegraphics[height=.29\textwidth]{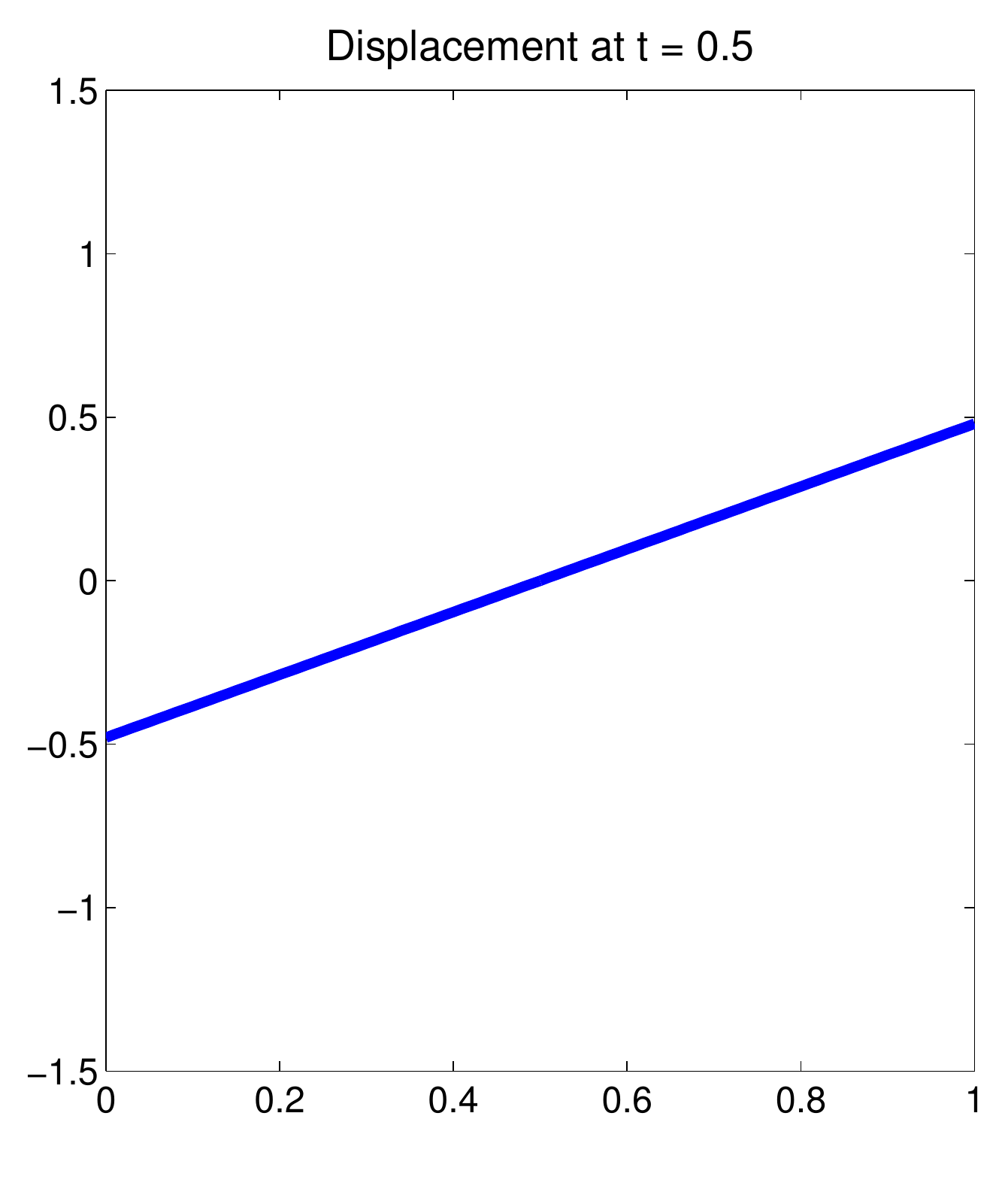}
  \includegraphics[height=.29\textwidth]{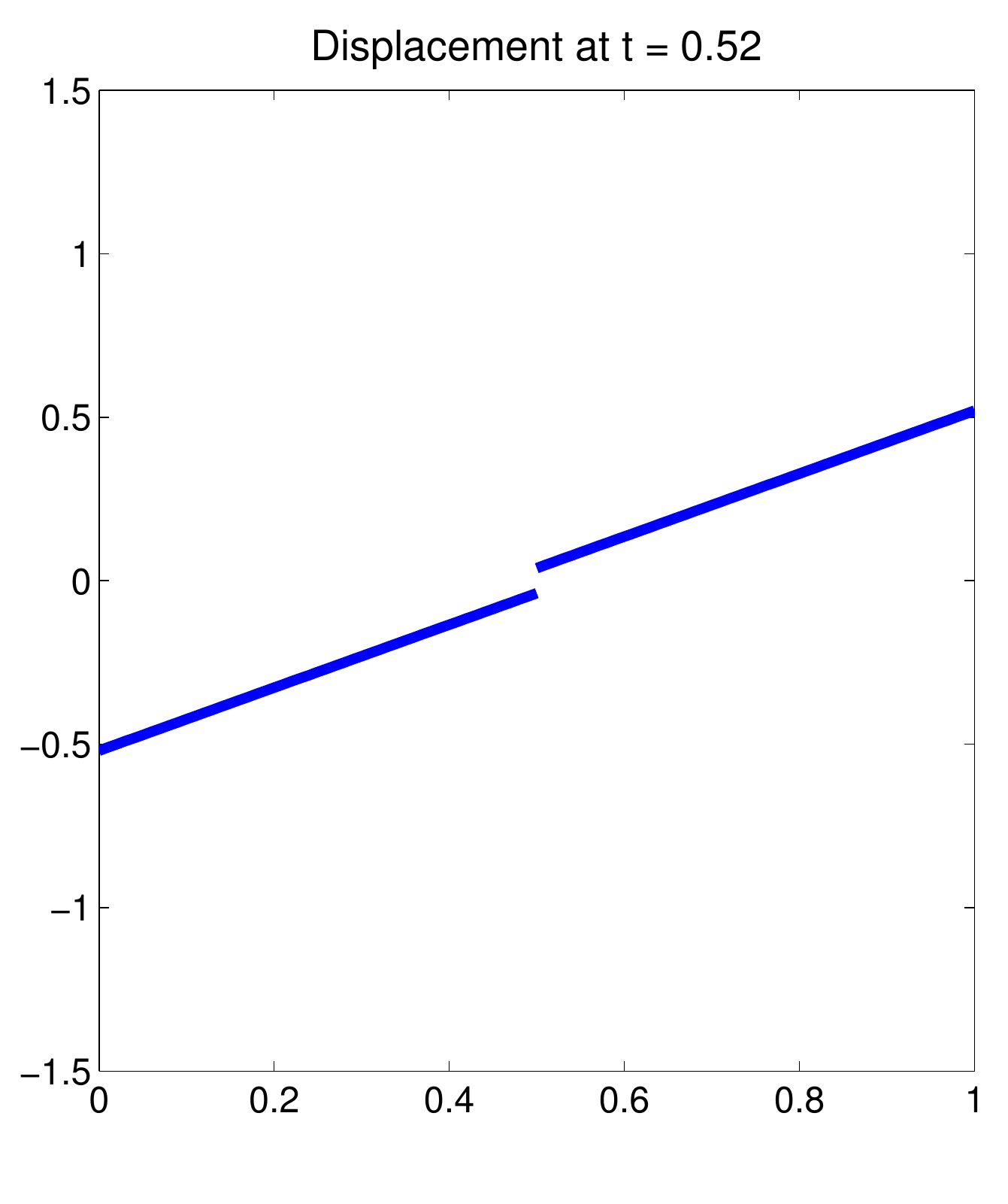}
  \includegraphics[height=.29\textwidth]{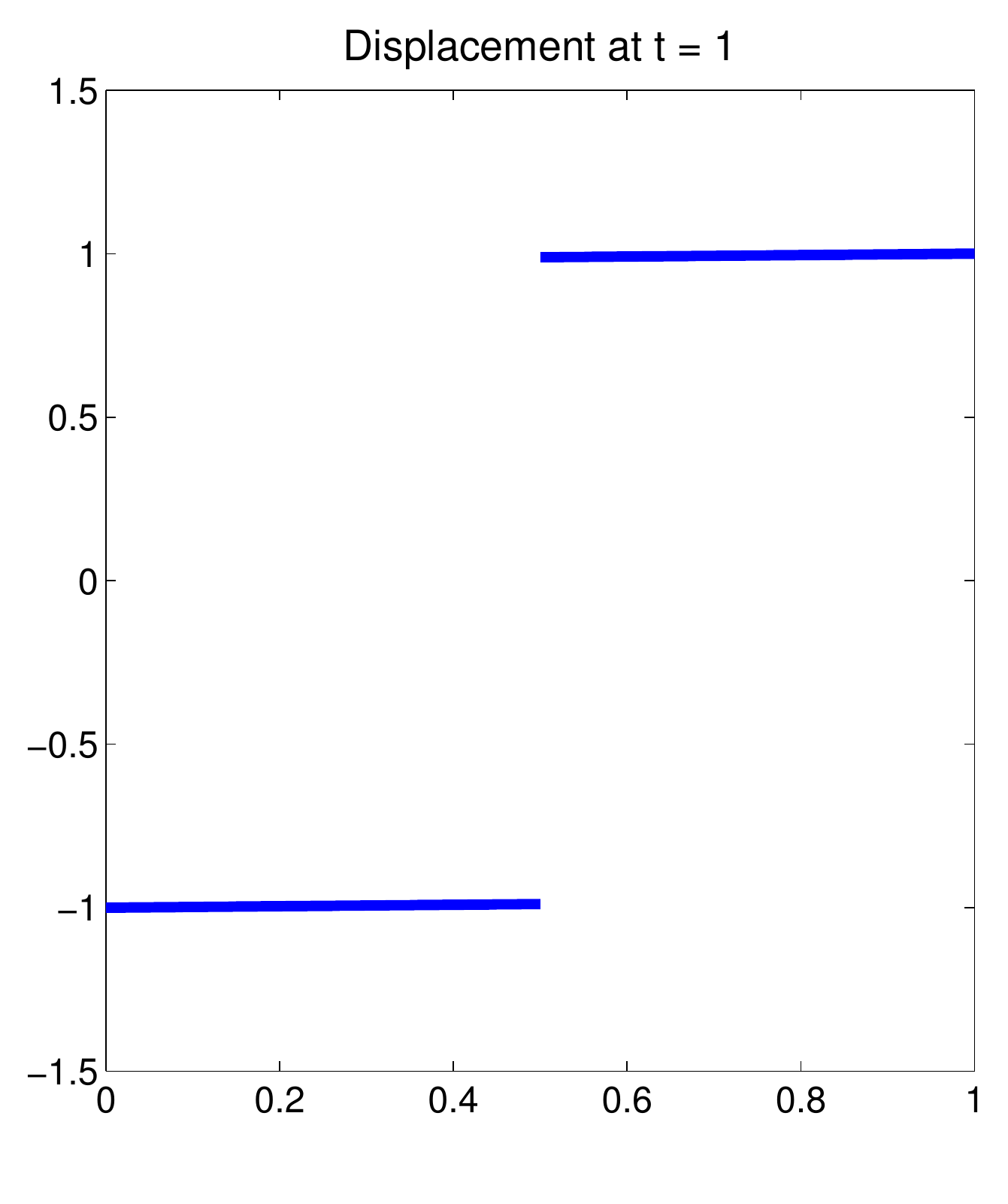}
  \caption[One dimensional cohesive crack evolution $R\ge2\ell$]
{The evolution of the quasistatic cohesive fracture for $R\ge2\ell$ at time instances $t=0, 0.5, 0.52, 1$.}\label{cohes1d}
\end{figure}
Note that in the time interval $[0, 0.5]$ the evolution follows the elastic deformation.
After $t= 0.5$ a fracture appears, since the elastic deformation 
is not any more a critical point of the energy functional (see \cite[Section 9]{C}).
Then, the pre-fracture phase starts, showing a bridging force acting on the two lips of the crack.
At time $t=1$ the cohesive energy reaches its maximum, and the body is completely fractured. 
It is worth observing that this evolution coincides with the one analytically calculated in \cite[Section 9]{C}.
\begin{figure}[!ht]
\centering
\includegraphics[height=.45\textwidth]{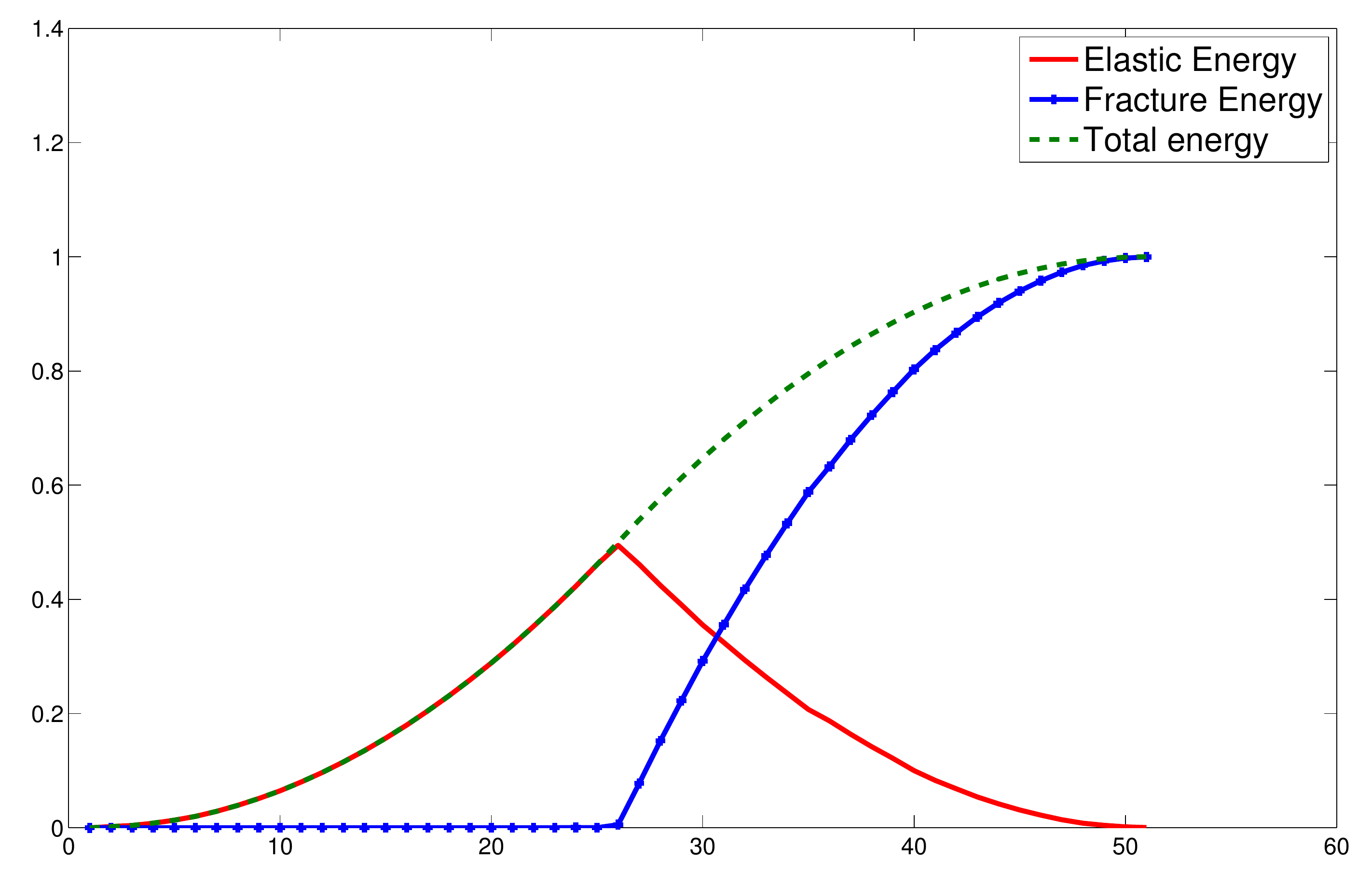}
\caption[One dimensional cohesive crack energy growth with $R\ge2\ell$]{The total, fracture, and elastic energy evolution of the quasistatic cohesive fracture for $R\ge2\ell$.}\label{cohes1d_en}
\end{figure}
We can also investigate what happens from the energy point of view, see Figure~\ref{cohes1d_en}. 
We have a smooth transition between the different phases, and the total energy has a nondecresing profile.  
The beginning of the pre-fractured phase can be observed at the $25^{th}$ time step (i.e. at time $t = 0.5$), 
when the elastic energy (in red) starts decreasing and the crack energy (in blue) starts increasing. 
The final phase of complete rupture is then attained at the final time step $t = 1$.
Although we focused on the time interval $[0,1]$, one could 
check that the three energy profiles remain constant for $t > 1$.

\begin{figure}[!ht]
\centering
 \includegraphics[height=.29\textwidth]{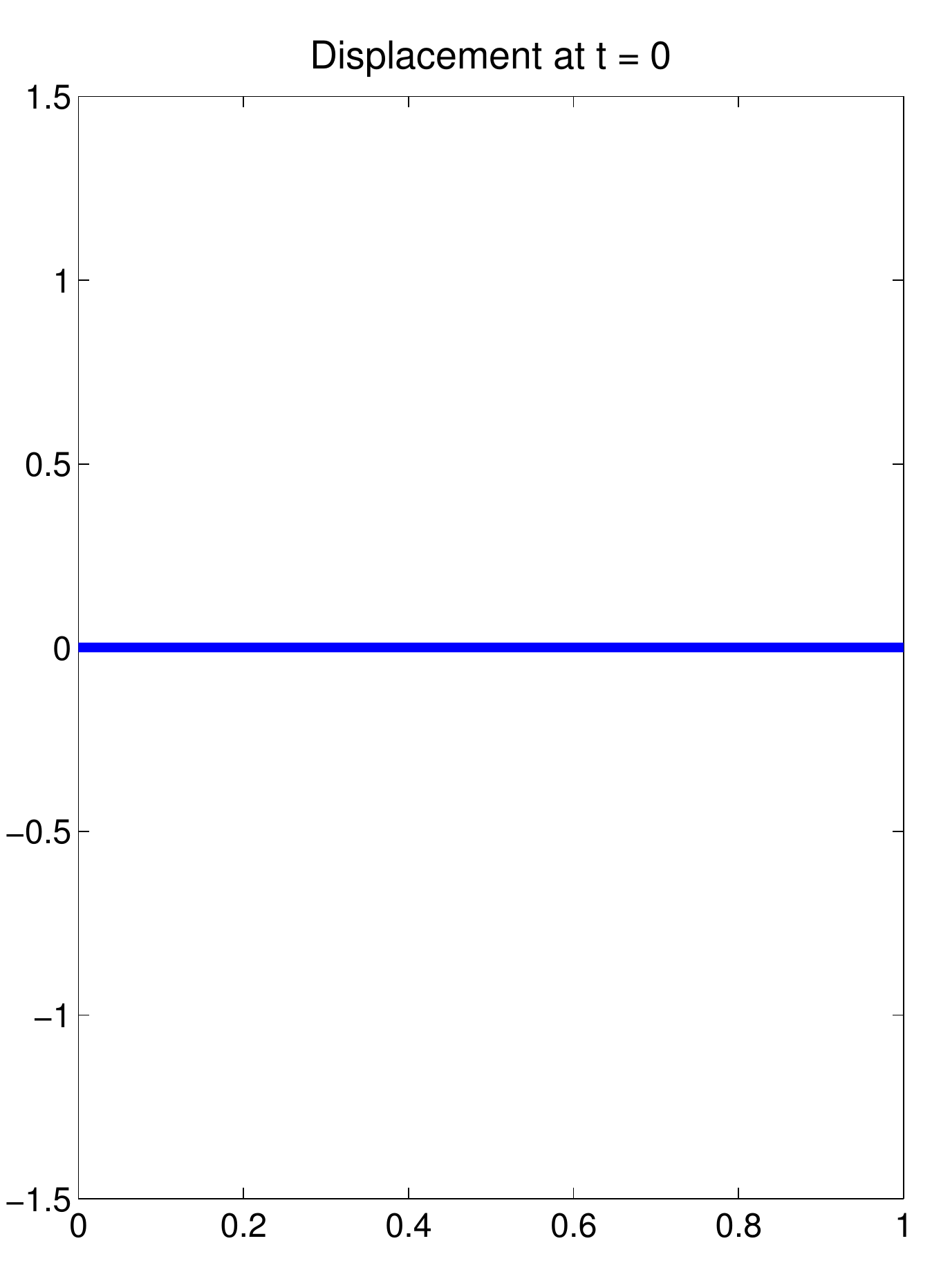}
  \includegraphics[height=.29\textwidth]{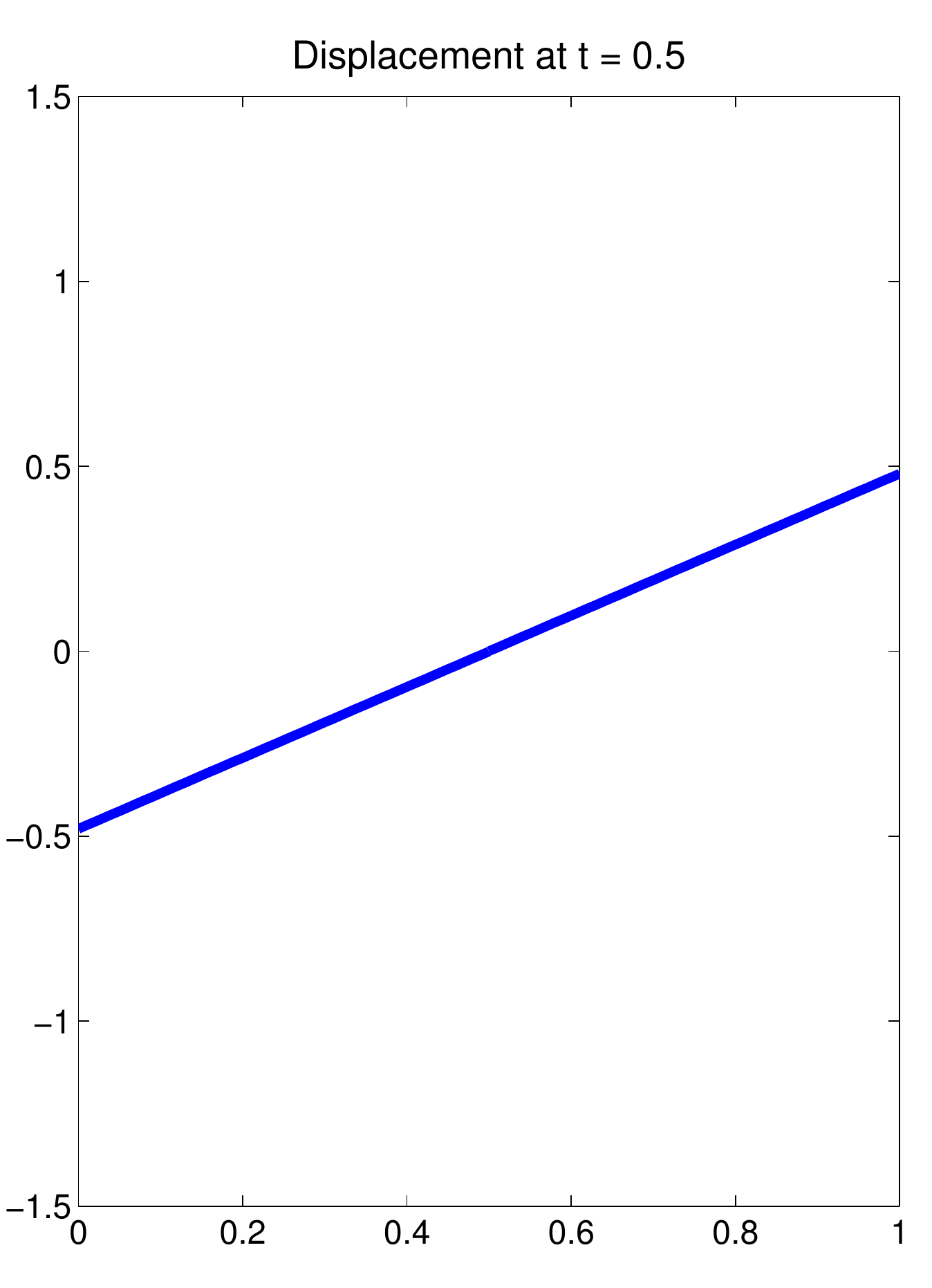}
 \includegraphics[height=.29\textwidth]{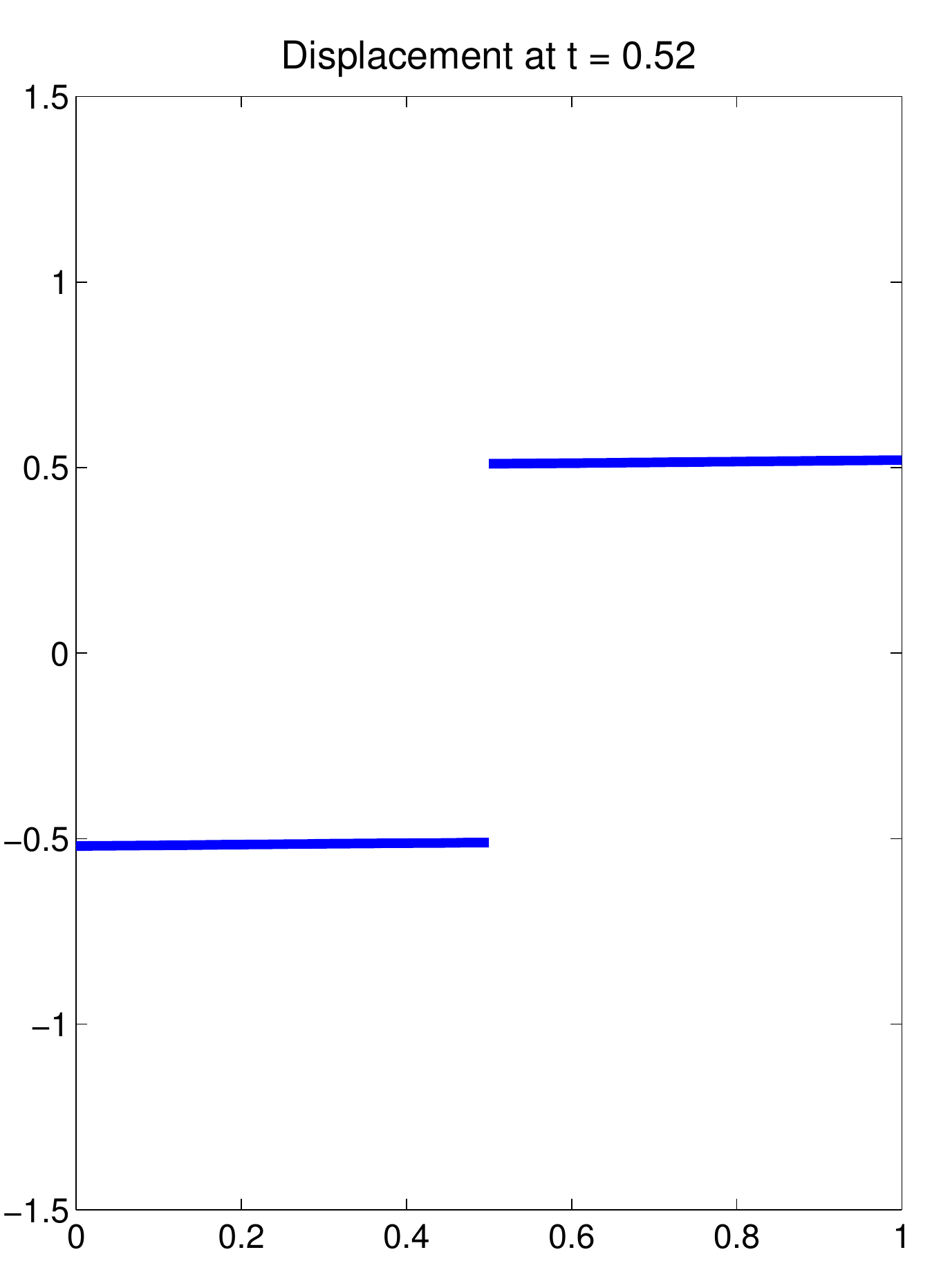}
  \includegraphics[height=.29\textwidth]{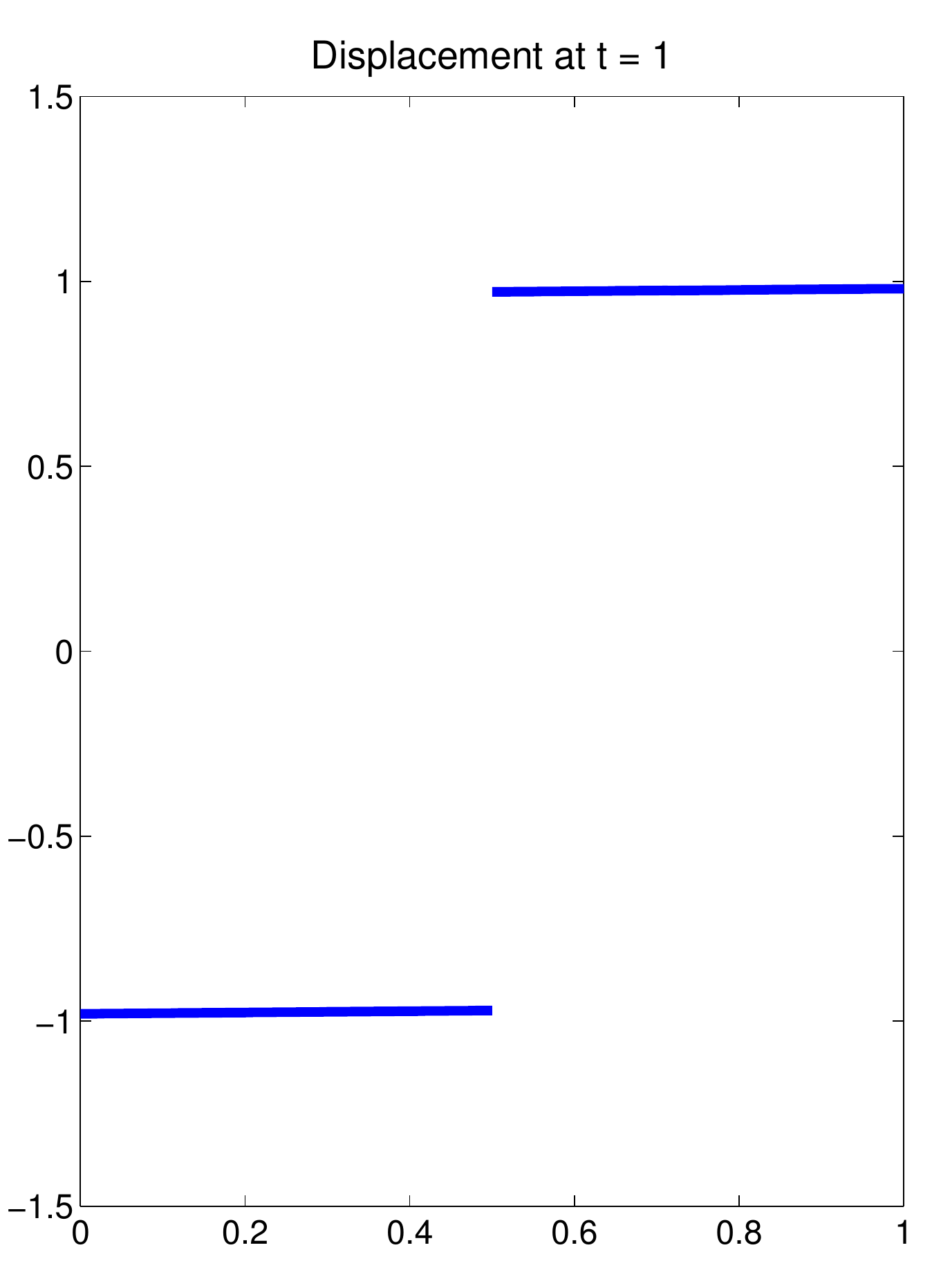}
\caption[One dimensional cohesive crack evolution $R\le2\ell$]{The evolution of the quasistatic cohesive fracture for $R < 2\ell$ at time instances $t=0, 0.5, 0.52, 1$.}\label{noncohes1d}
\end{figure}

\vspace{.2cm}

\noindent

\textbf{Case $R < 2\ell$.}
The evolution of the system changes radically when $R < 2\ell$. 
In this case (see Figure~\ref{noncohes1d}) the failure happens instantaneously, 
without a bridging phase, and thus the body exhibits 
what in literature is known as \textit{brittle} behavior. 
More precisely, in the time interval $[0, 0.5]$ the evolution follows 
again the elastic deformation, and a crack appears at $t = 0.5$.
However, immediately after $t = 0.5$ the body is completely fractured, 
and no cohesive forces appear.
It is important to observe that in this case we actually observe an evolution 
along critical points that are \textit{not global minimizers}.
Indeed, the evolution is elastic until $t = 0.5$, 
although it would be energetically convenient
to completely break the body at some earlier time $\overline{t} < 0.5$
(see \cite[Section 9]{C} for a detailed description of all critical points).
Thus, we see that the algorithm chooses the critical point which is the closest to the initial configuration, even if other options are available, which are more convenient 
from an energetic point of view.
This evolution is particularly supported by the idea that in nature a body does not completely change its configuration crossing high energetic barriers if a stable configuration can be found with less energetic effort.

\begin{figure}[!ht]
\centering
 \includegraphics[height=.45\textwidth]{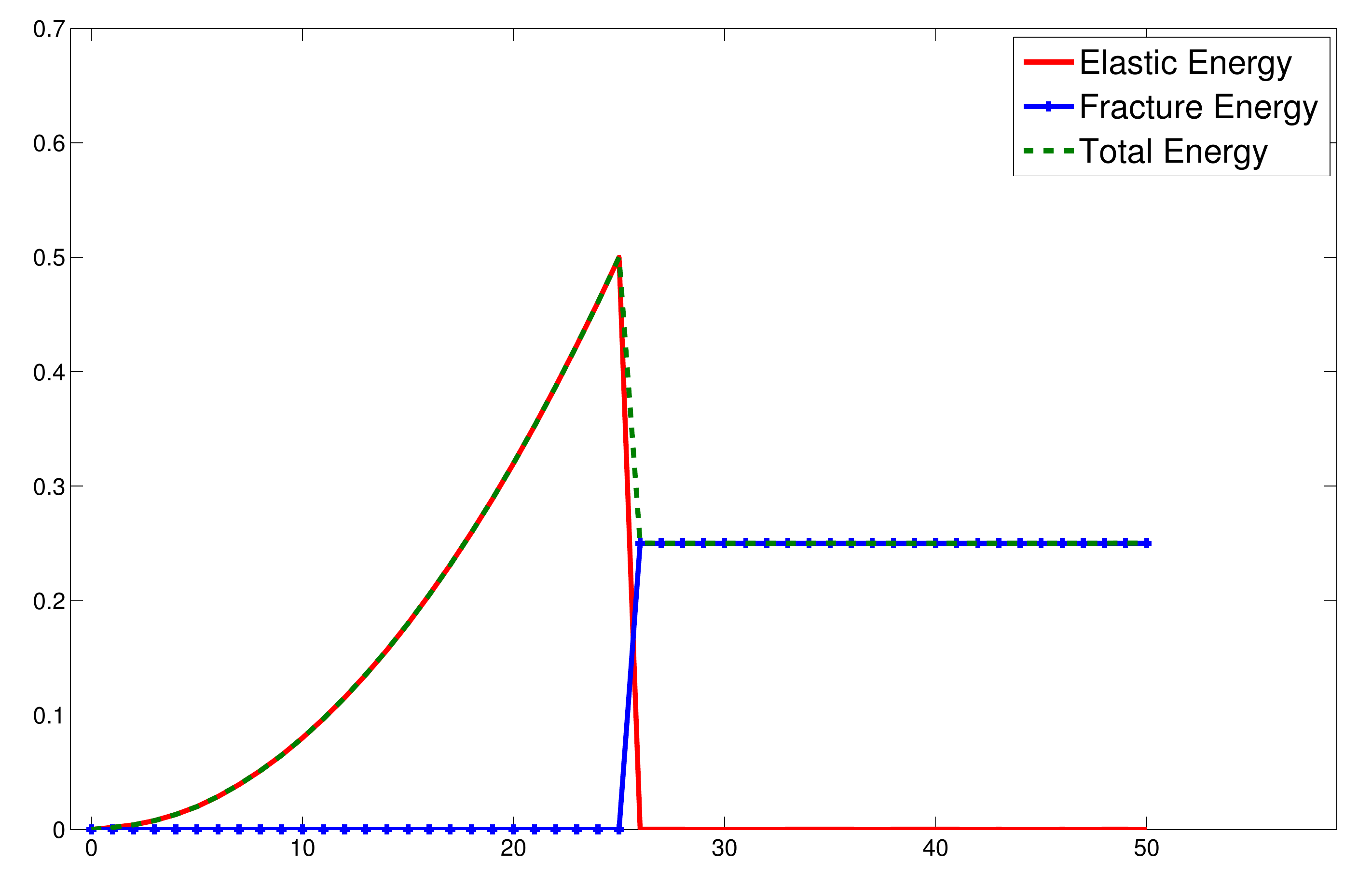}
\caption[One dimensional cohesive crack energy growth with $R\le2\ell$]{The total, fracture, and elastic energy evolution of the quasistatic cohesive fracture for $R < 2\ell$.}\label{Nocohes1d_en}
\end{figure}

Also in this case, we can observe the evolution from the energetic viewpoint, 
see Figure~\ref{Nocohes1d_en}. 
At time $t =0.5$, when the elastic deformation ceases to be a critical point, 
the domain breaks and the total energy decreases up to the value of $R/2$, 
so that no bridging force is keeping the two lips together.
As we already observed, the evolution along global minimizers would instead 
lead to a fracture way before the critical load is reached.\\
Again, the evolution found with our numerical simulation coincides with
that one given in \cite[Section 9]{C}.
In particular, our simulations agree with the \textit{crack initiation criterion}
(see \cite[Theorem 4.6]{C}), which states that a crack appears only when 
the maximum sustainable stress along $\Gamma$ is reached.
In this case, this happens at $t = 0.5$, when the slope of the elastic evolution 
reaches the value $g'(0) = 1$.

\subsection{Numerical simulations in $2$ dimensions}

Having a first analytical validation of the numerical minimization procedure, we can now 
challenge the algorithm in the simulation of two dimensional evolutions.
We now consider the domain introduced in Section~\ref{sec:discrete} 
setting $\ell = 0.5$, $2N = 8$, and $\kappa=1/2$.
Within this choice, the crack initiation time is reduced exactly of a factor $1/2$, allowing us to speed up the failure process. 
Since all the computations are performed on a MacBook Pro equipped 
with a 2.6GHz Intel Core i7 processor, 8GB of RAM, 1600MHz DDR3, 
the two dimensional simulations are performed only for a qualitative purpose. 
Indeed, we are mainly interested in showing that our algorithm produces physically 
sound evolutions also in dimension $2$, and when the external displacement $f$ is non-trivial.
The very sparse discretization of the domain $\Omega$ is due to the fact that
the minimization in \eqref{algobis} requires a huge computational effort, 
both in terms of time and memory.
Indeed, in order to implement more realistic experiments, with a finer discretization, 
we would need to modify the architecture of the minimization algorithm, in such a way 
that it may run on parallel cores.

We perform two different series of experiments, one with boundary datum
$$
\omega_1(t)(\xx) = 2(x_1-\ell)t, \qquad \text{ for every  } t \in [0,1] \text{ and } \xx \in \Omega,
$$
see Figure~\ref{cohes2d_1}, and the other one with boundary datum
$$
\omega_2 (t)(\xx) = 2t \cos\left(2\frac{x_2-\ell}{\ell}\right)(x_1-\ell) , \qquad \text{ for every  } t \in [0,1] \text{ and } \xx \in \Omega,
$$
see Figure~\ref{cohes2d_2}.
Here, we denoted by $\xx = (x_1, x_2)$ the generic point of $\Omega = (0, 1) \times (0,1)$.
We now need to reduce the tolerance of the termination condition 
of the outer loop of the Algorithm, setting it to $5\cdot10^{-14}$.
Indeed, we experimented that for bigger values of this tolerance some instabilities 
in the solution were introduced, leading to an asymmetric evolution, 
also in the case of $\omega_1$ as external displacement, 
where we expect an invariant behavior with respect to the space variable $x_2$.

\begin{figure}[!ht]
\centering
  \includegraphics[width=.24\textwidth]{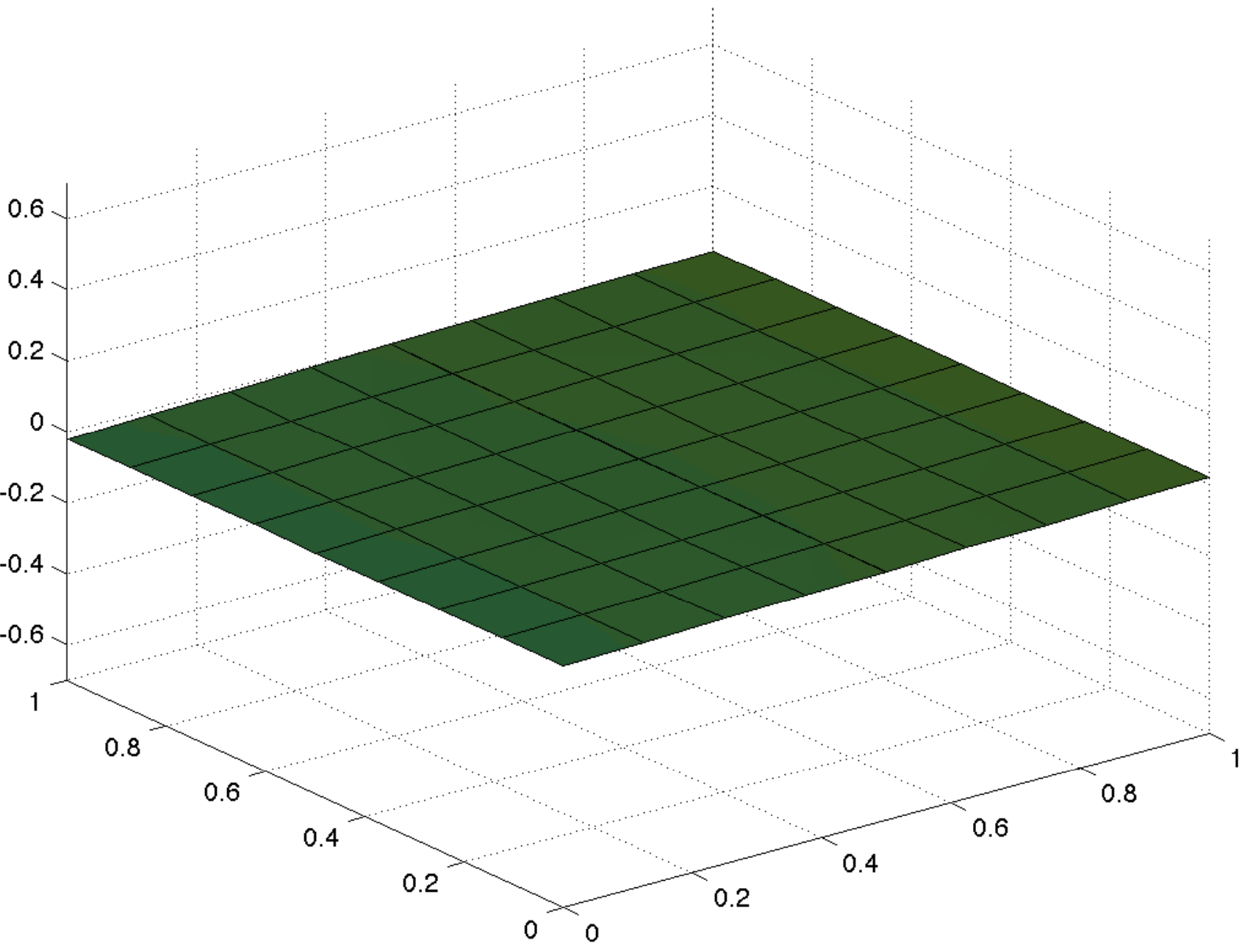}
  \includegraphics[width=.24\textwidth]{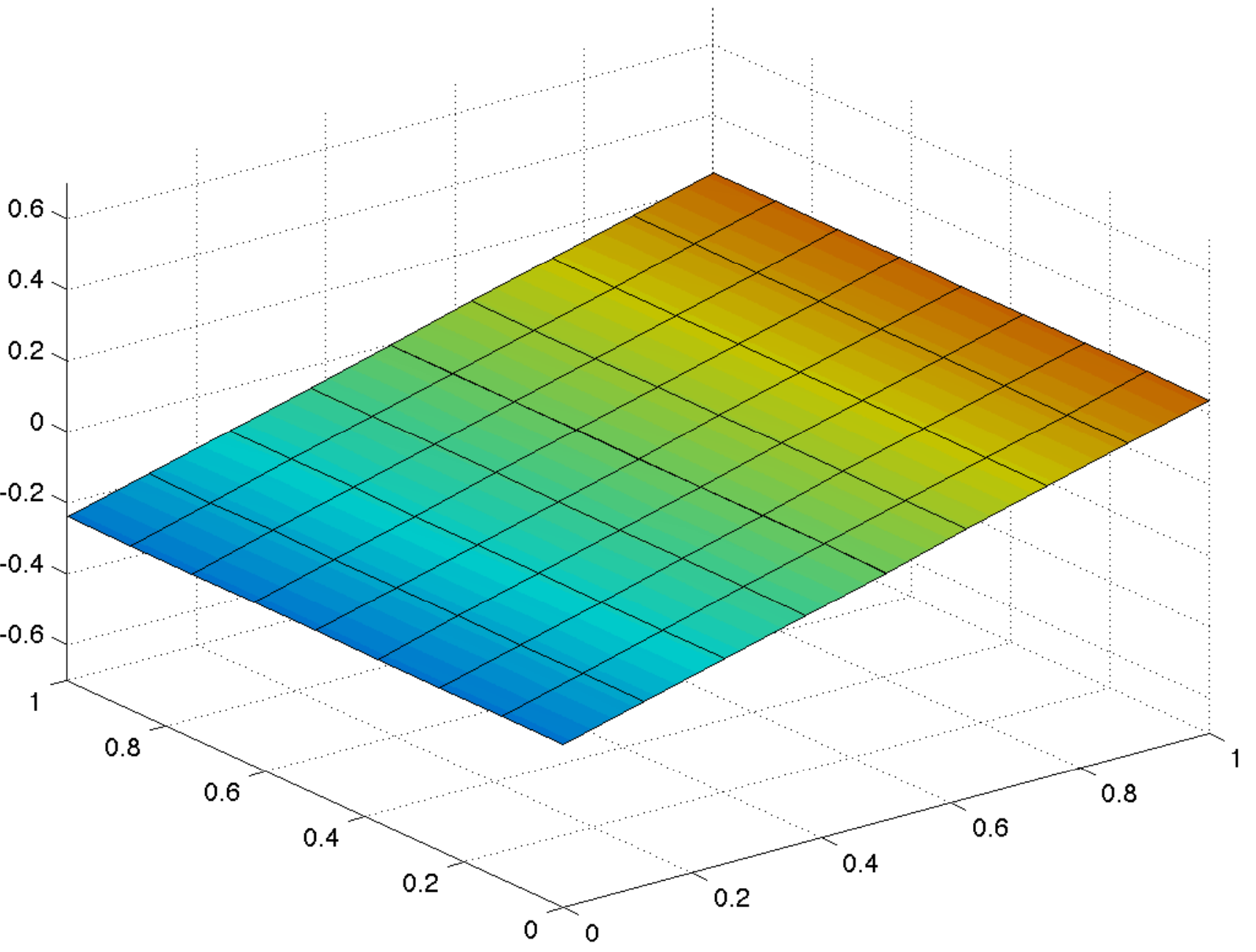}
  \includegraphics[width=.24\textwidth]{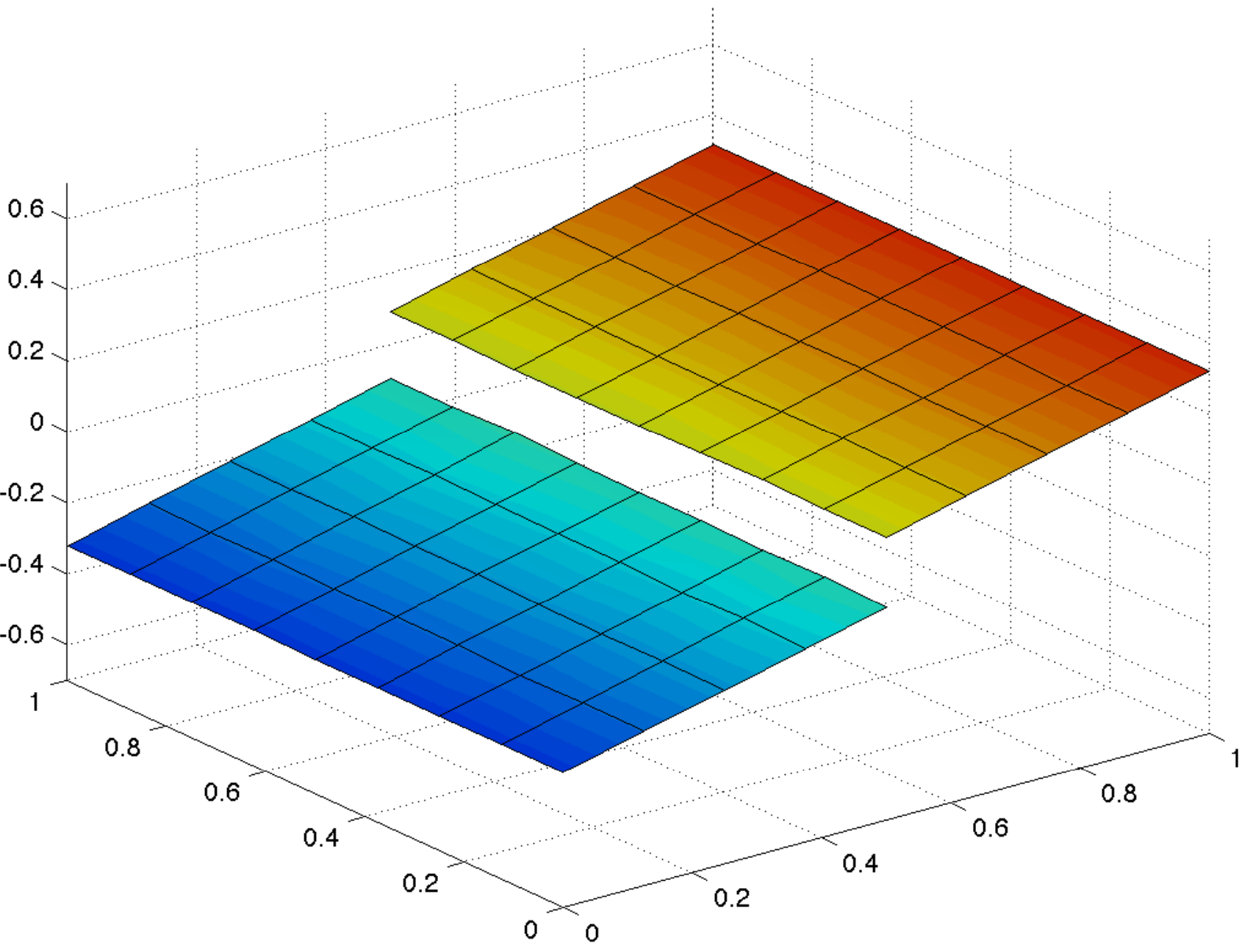}
  \includegraphics[width=.24\textwidth]{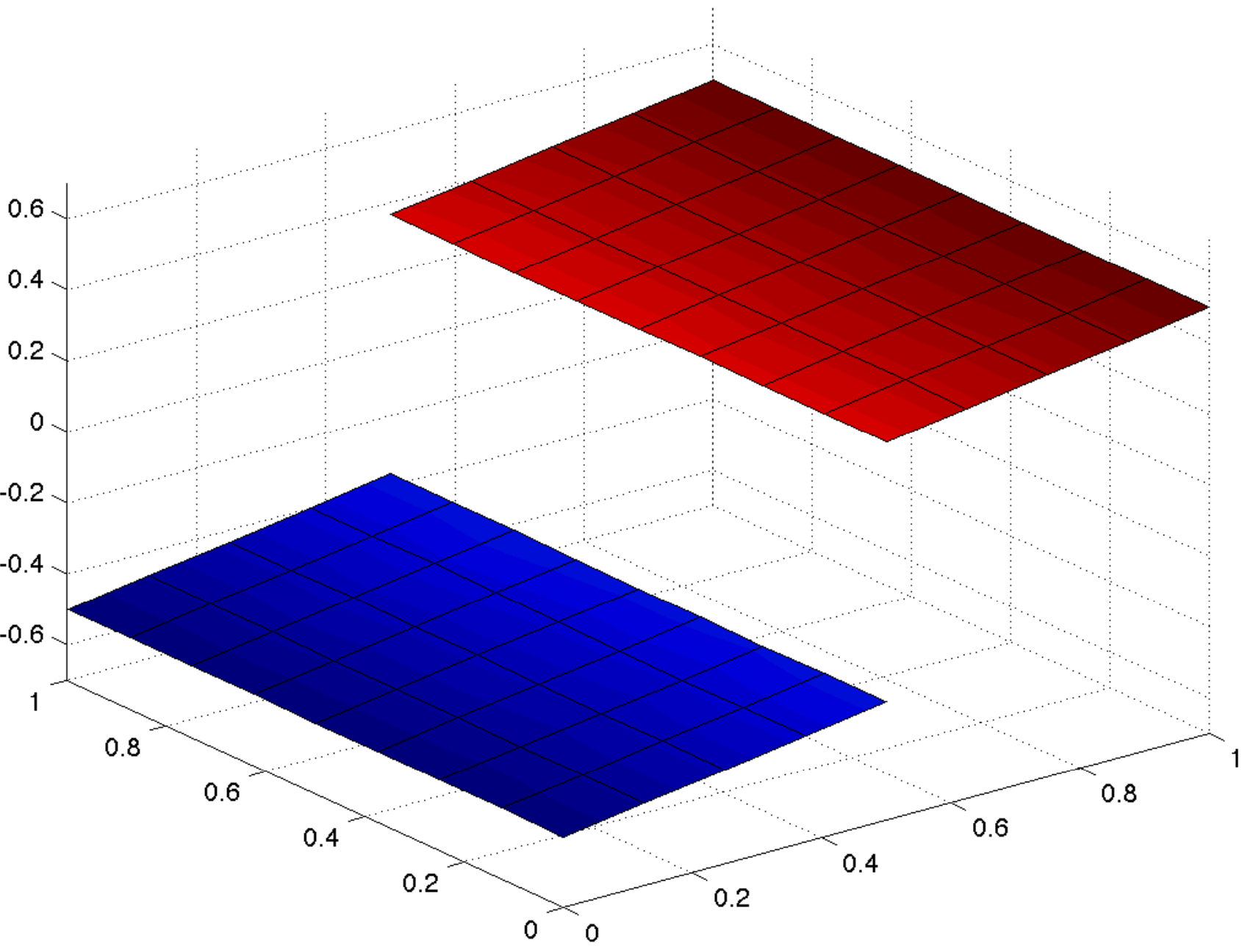}
\caption[Two dimensional cohesive crack evolution $R\ge2\ell$ (1)]{The evolution of the quasistatic cohesive fracture for $R\ge2\ell$ at time instances $t=0, 0.24, 0.32, 0.5$ 
with external displacement $\omega_1$.}\label{cohes2d_1}
\end{figure}

\begin{figure}[!ht]
\centering
  \includegraphics[width=.24\textwidth]{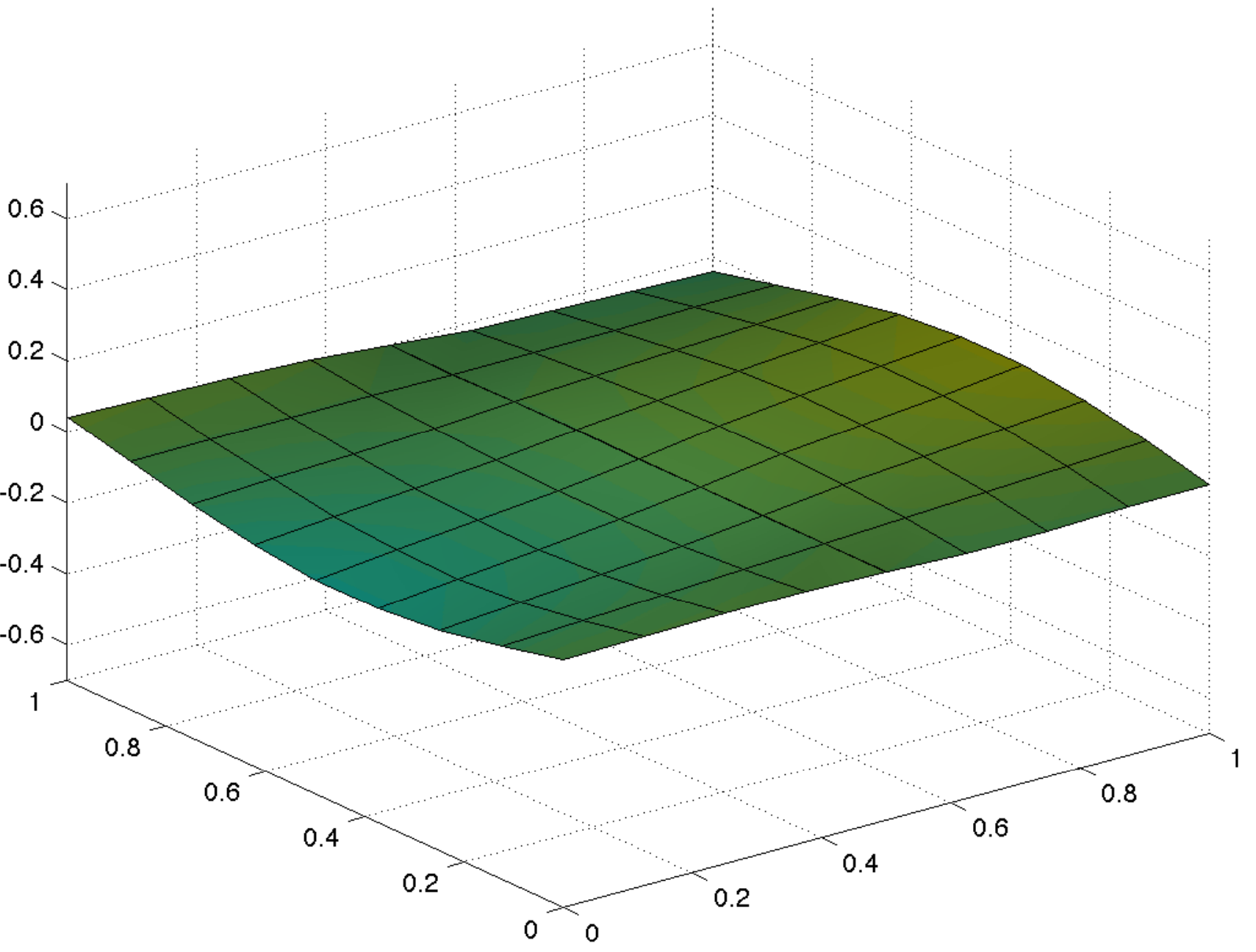}
 \includegraphics[width=.24\textwidth]{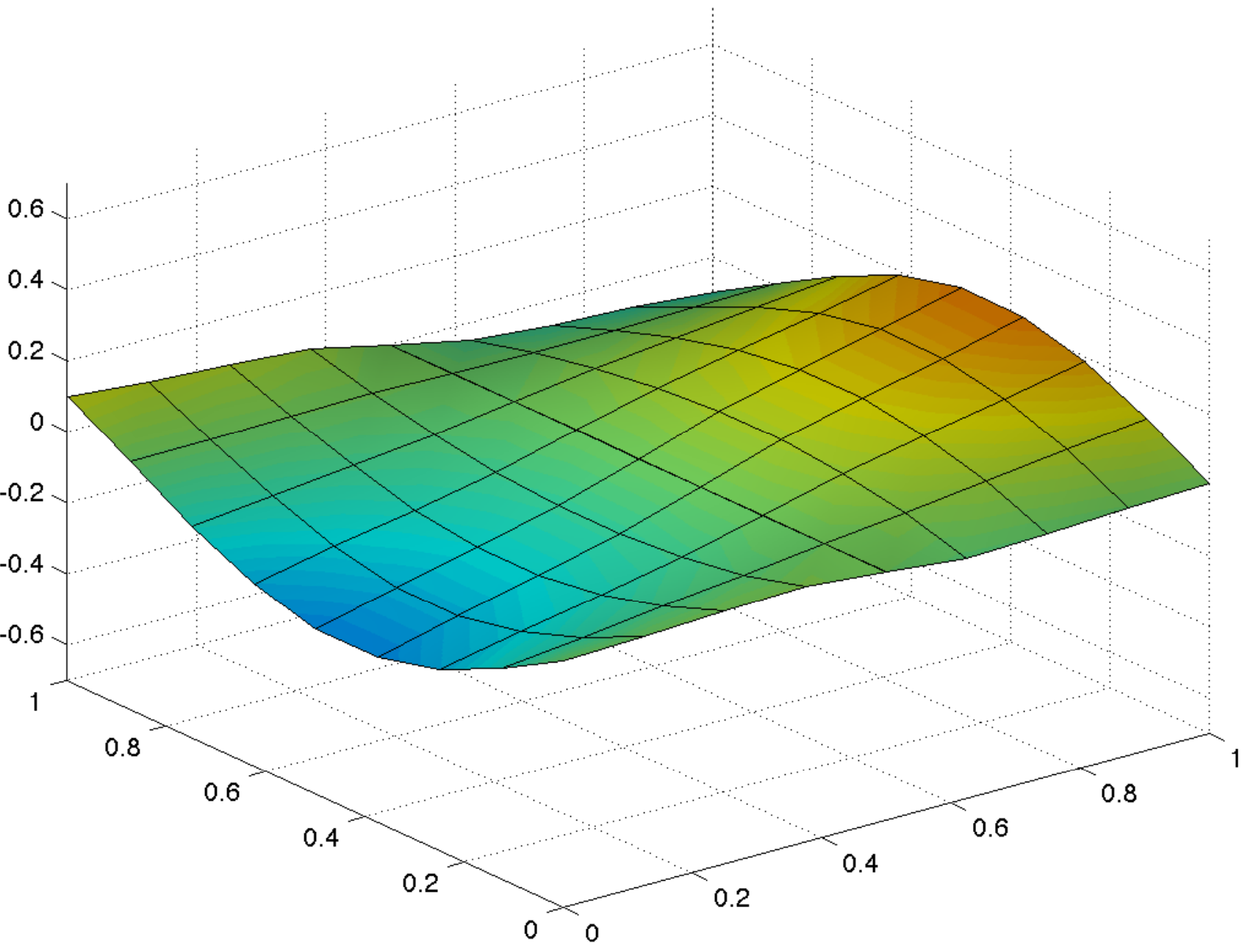}
  \includegraphics[width=.24\textwidth]{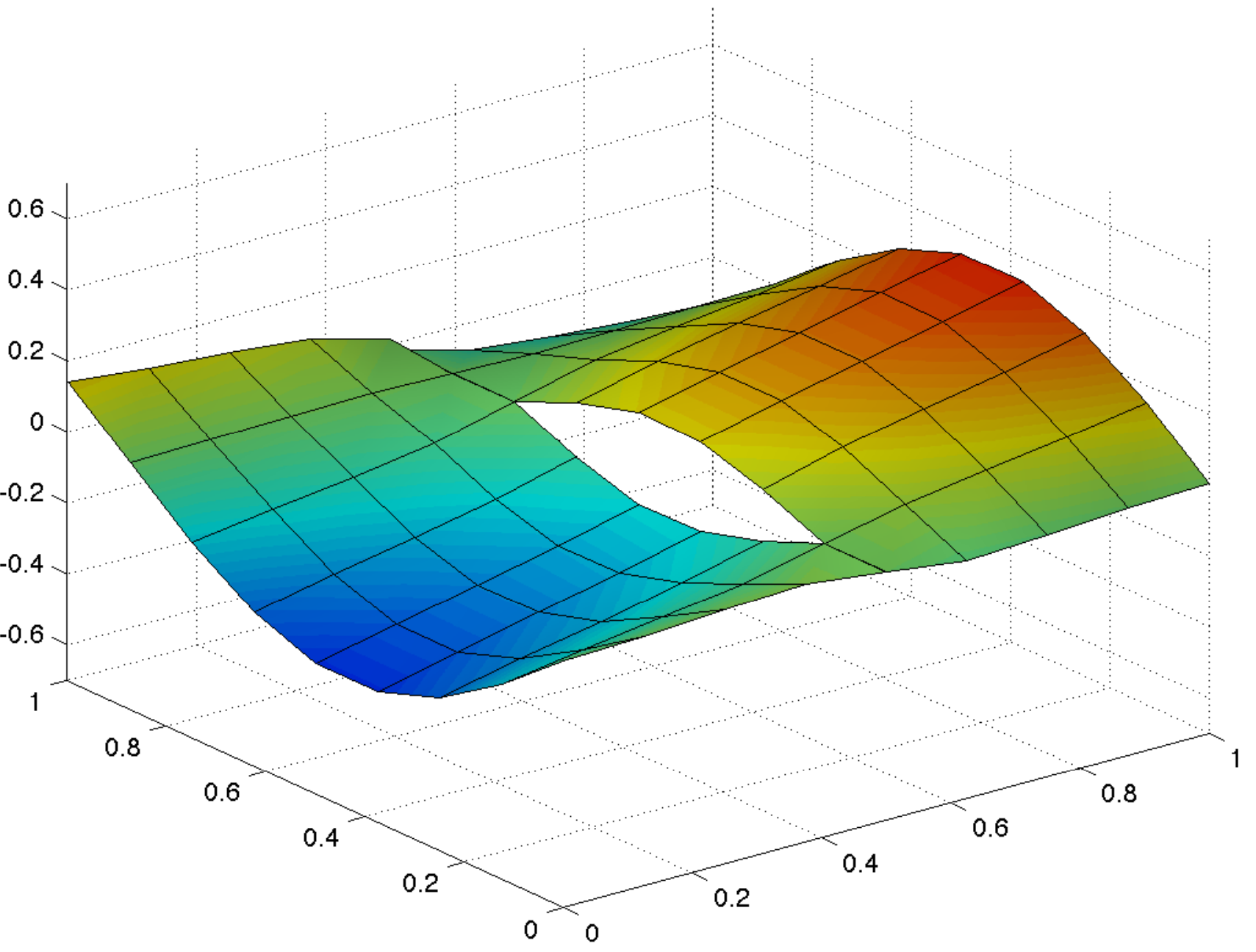}
  \includegraphics[width=.24\textwidth]{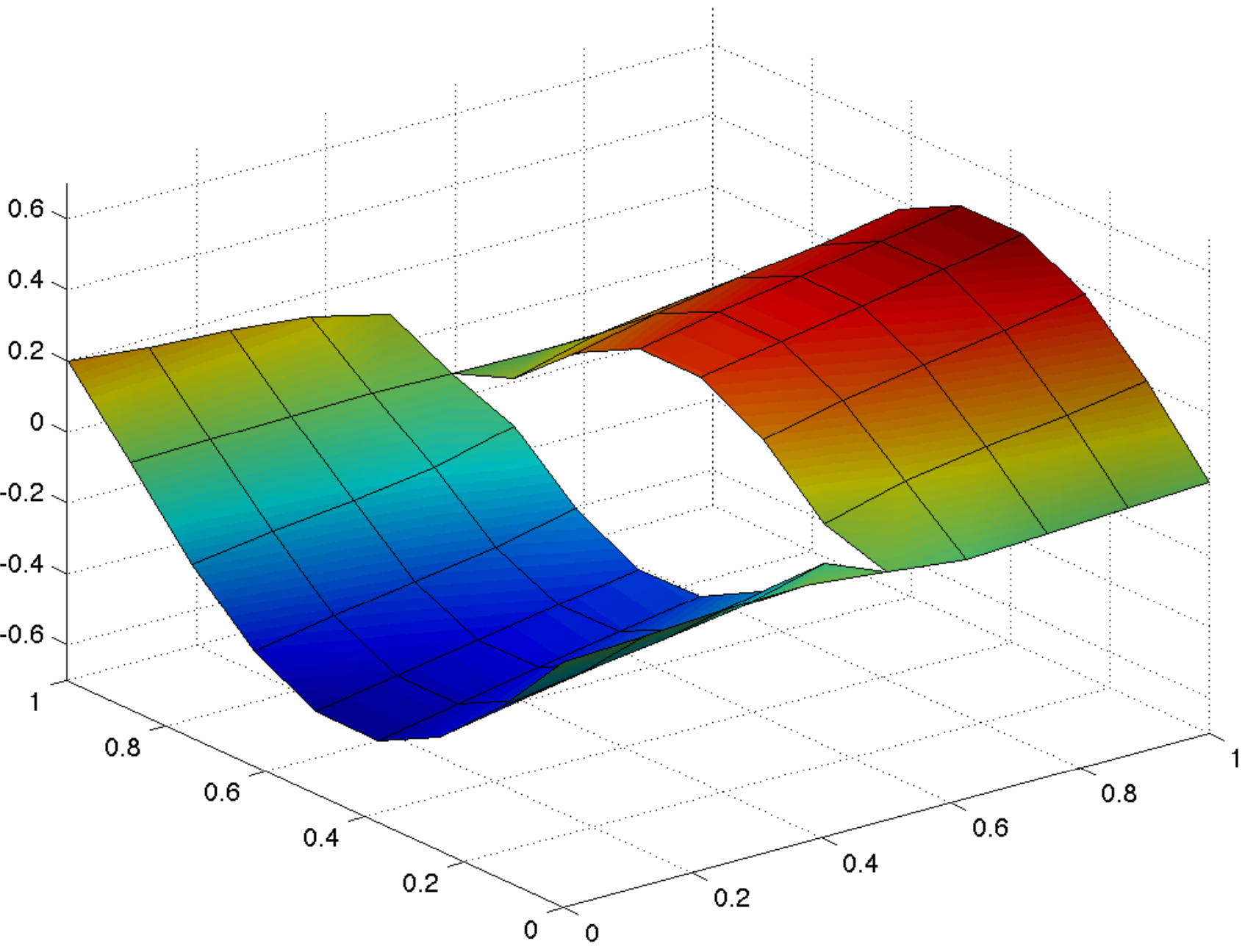}
\caption[Two dimensional cohesive crack evolution $R\ge2\ell$ (2)]{The evolution of the quasistatic cohesive fracture for $R\ge2\ell$ at time instances 
$t=0.1, 0.24, 0.34, 0.5$ with external displacement $\omega_2$.}\label{cohes2d_2}
\end{figure}

\vspace{.2cm}

\noindent
\textbf{Case $R\ge2\ell$.}
In Figure~\ref{cohes2d_1} and~\ref{cohes2d_2} we report 4 different instances of the evolution for the two different boundary data, when $R \ge2 \ell$. 
When the external displacement is $\omega_1$, 
which is constant with respect to the second coordinate $x_2$,
we observe that the evolution is also constant with respect to $x_2$.
For both boundary data, the failure of the body undergoes the three phases 
of deformation, as it happened in the one dimensional case.


\vspace{.2cm}

\noindent
\textbf{Case $R < 2 \ell$.}
When the boundary datum is $\omega_1$, see Figure~\ref{Noncohes2d_1}, 
the specimen breaks in a brittle fashion, without showing any cohesive intermediate phase.
This simulation is actually an evidence that the algorithm still characterizes the correct critical points, following the principle that the domain should not fracture 
as long as a non-fractured configuration is still a critical point.
We conclude commenting the simulation where the boundary datum is $\omega_2$ with $R < 2\ell$, 
see Figure~\ref{Noncohes2d_2}. 
By setting a displacement highly varying with respect to the $x_2$ coordinate, 
we observe that the different phases of the fracture formation can cohexist. 
At time $t=0.24$ the domain still presents no fracture, 
as expected by the previous numerical experiments. 
Then, at $t=0.34$, a pre-fracture appeares, but only 
at those points where the external load is bigger, i.e. around $x_2 = \ell$. 
In fact, even at the final time $t = 1$, the domain is not completely fractured.

Note that, when the boundary datum is $\omega_1$, 
the evolution coincides with the one obtained analytically
\cite[Section 9]{C}. In particular, the fracture appears at $t = 0.25$, when the slope of the elastic evolution reaches the value $\kappa g'(0) = 1/2$ and thus the \textit{crack initiation criterion is satisfied}.

\begin{figure}[!ht]
\centering
 \includegraphics[width=.24\textwidth]{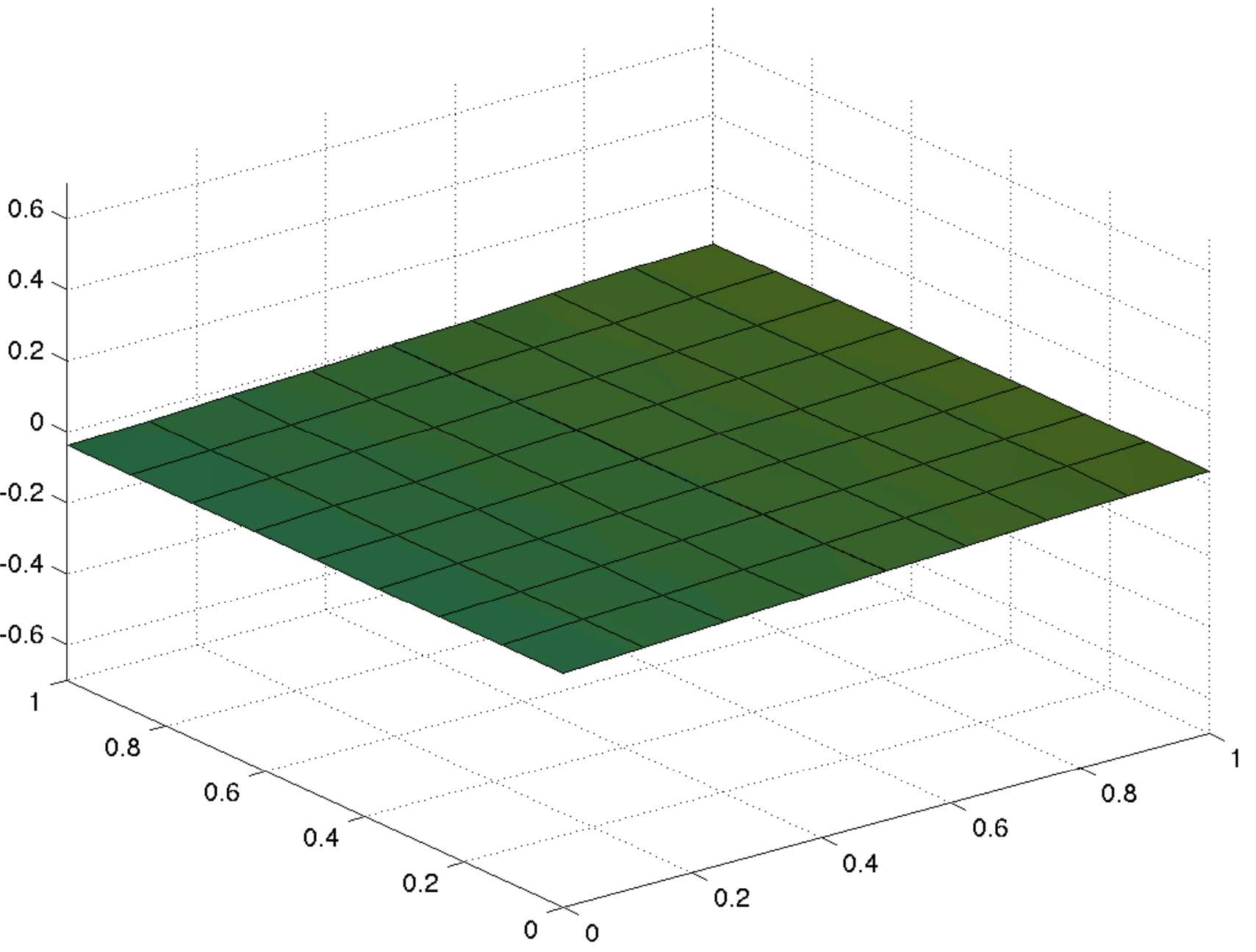}
  \includegraphics[width=.24\textwidth]{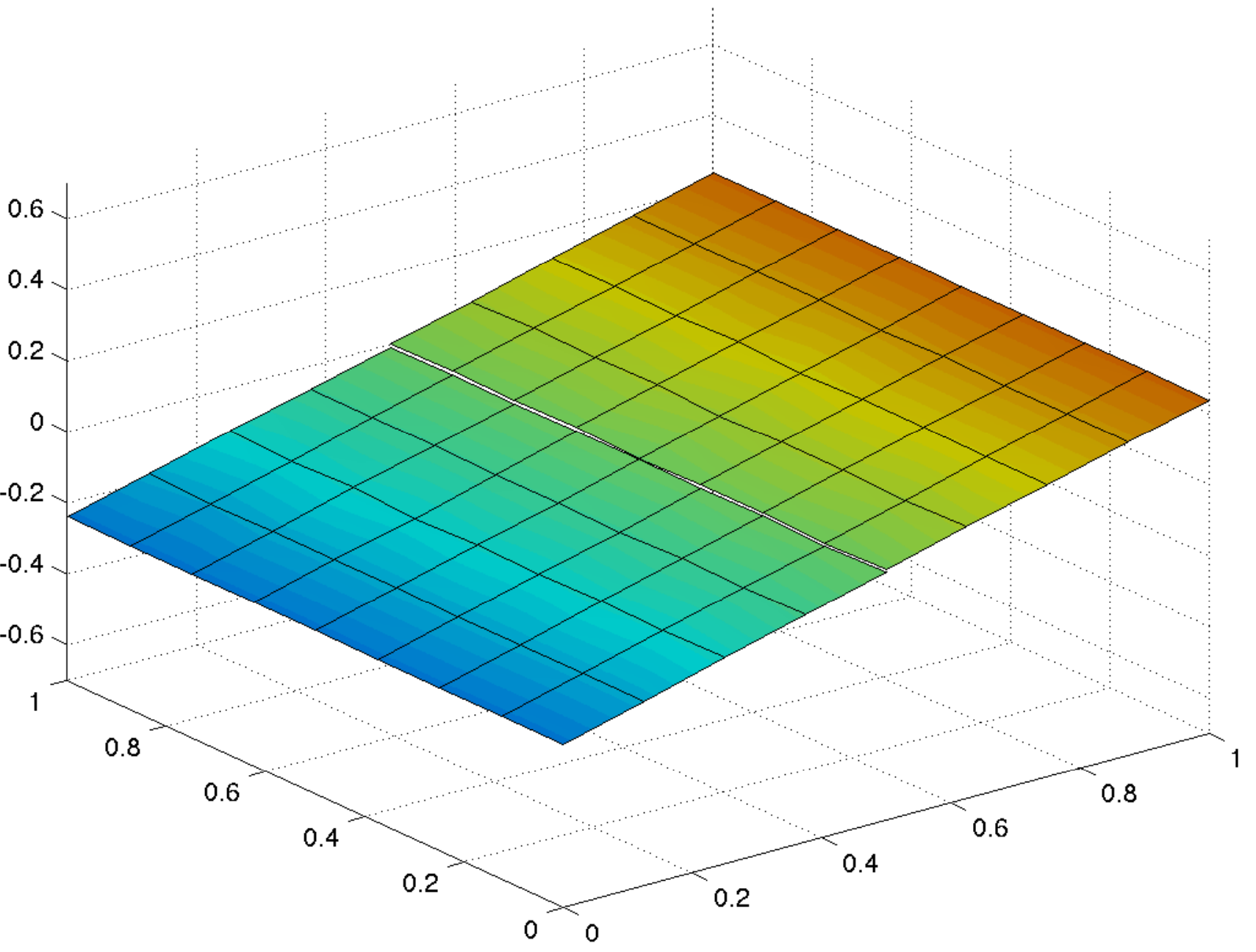}
  \includegraphics[width=.24\textwidth]{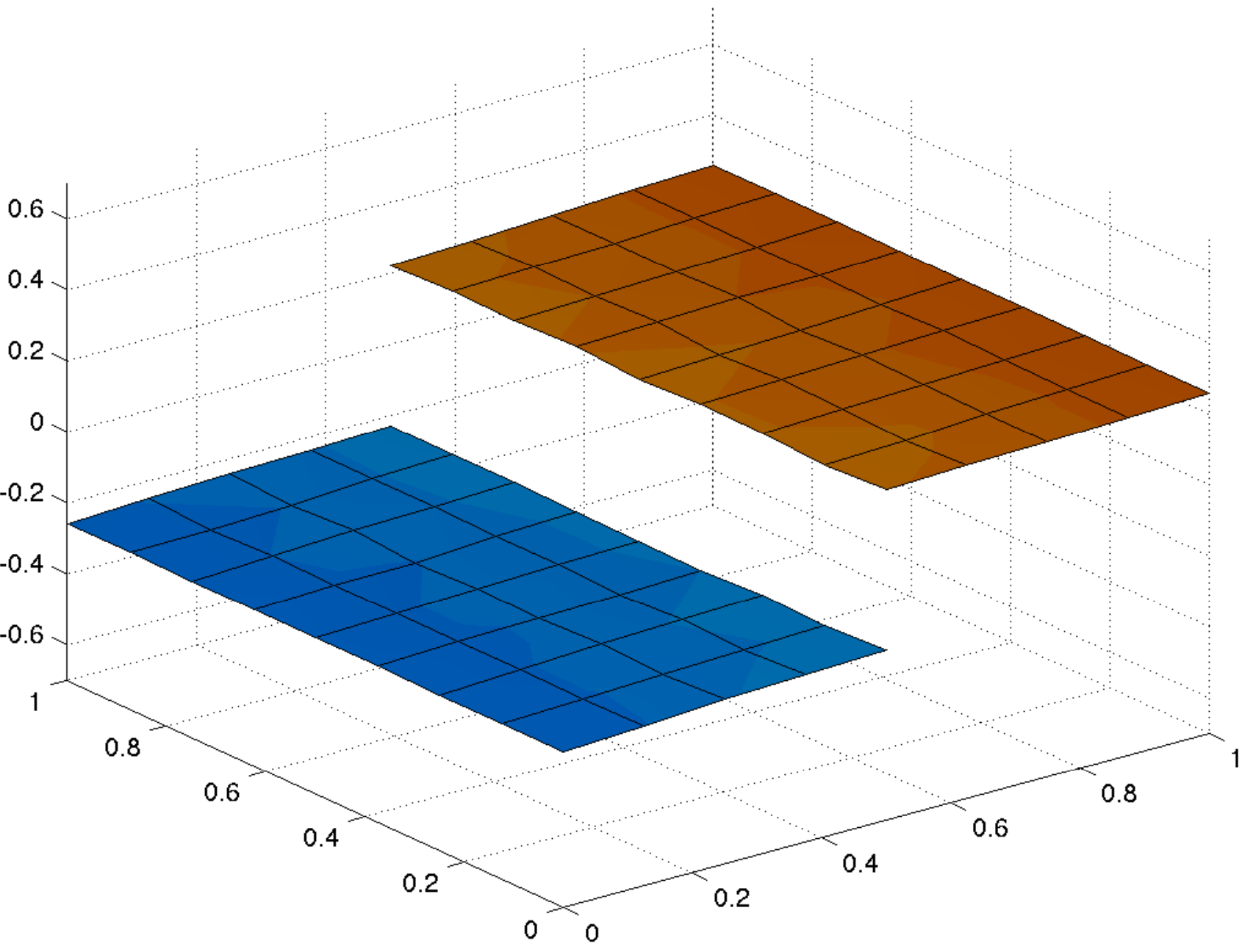}
  \includegraphics[width=.24\textwidth]{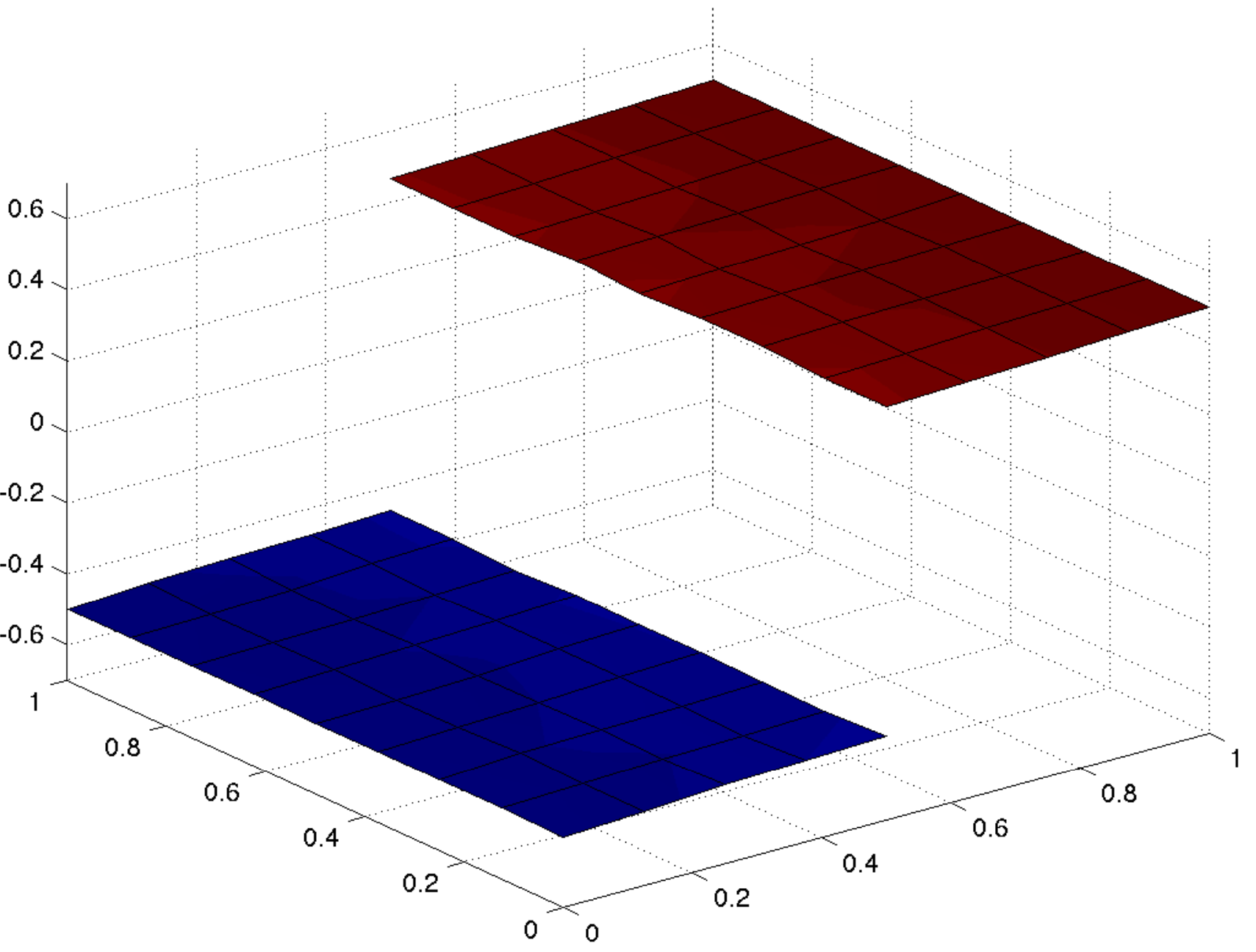}
\caption[Two dimensional cohesive crack evolution $R\le2\ell$ (1)]{The evolution of the quasistatic cohesive fracture for $R < 2\ell$ at time instances $t=0.04, 0.24, 0.26, 0.5$ 
with external displacement $\omega_1$.}\label{Noncohes2d_1}
\end{figure}

\begin{figure}[!ht]
\centering
  \includegraphics[width=.24\textwidth]{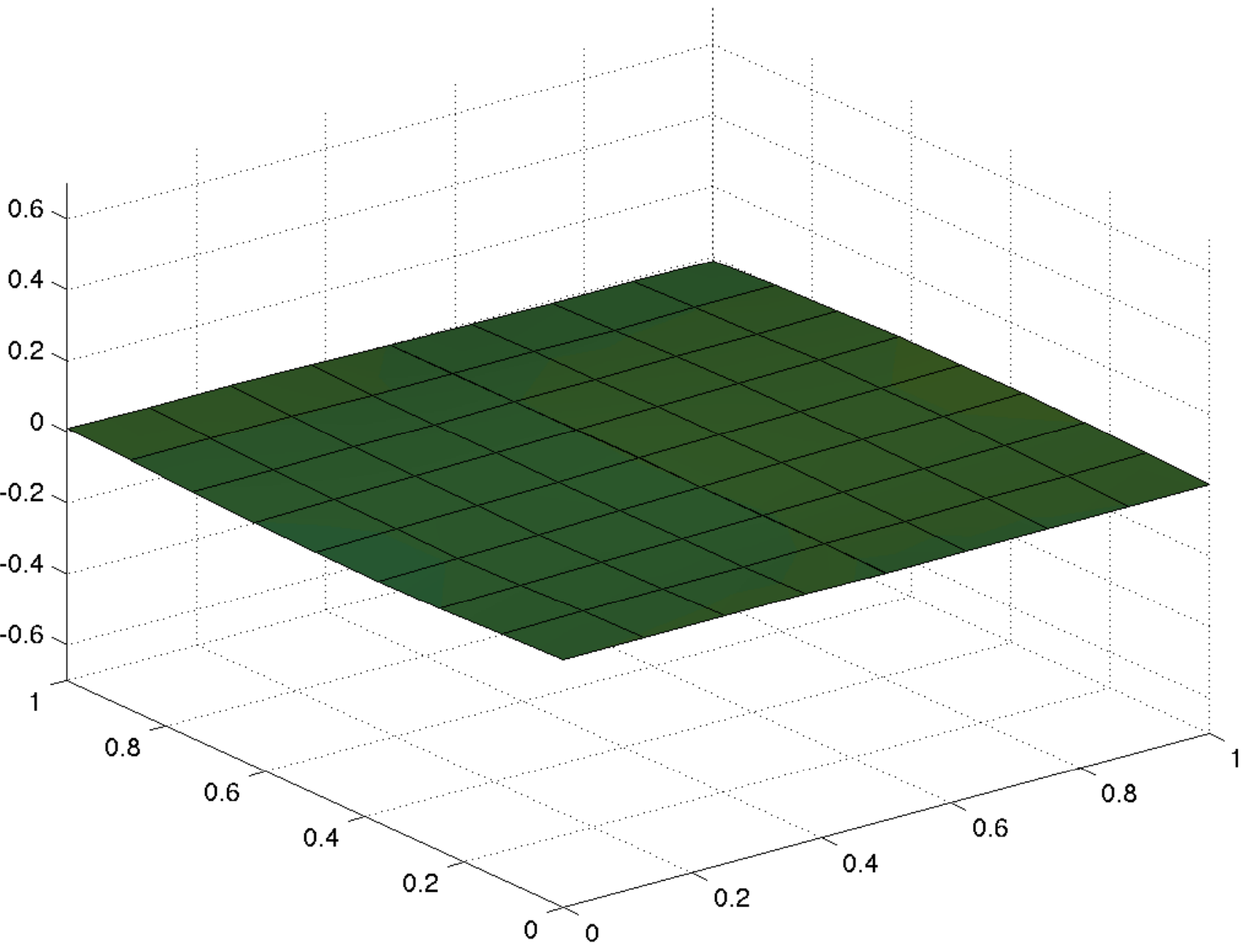}
  \includegraphics[width=.24\textwidth]{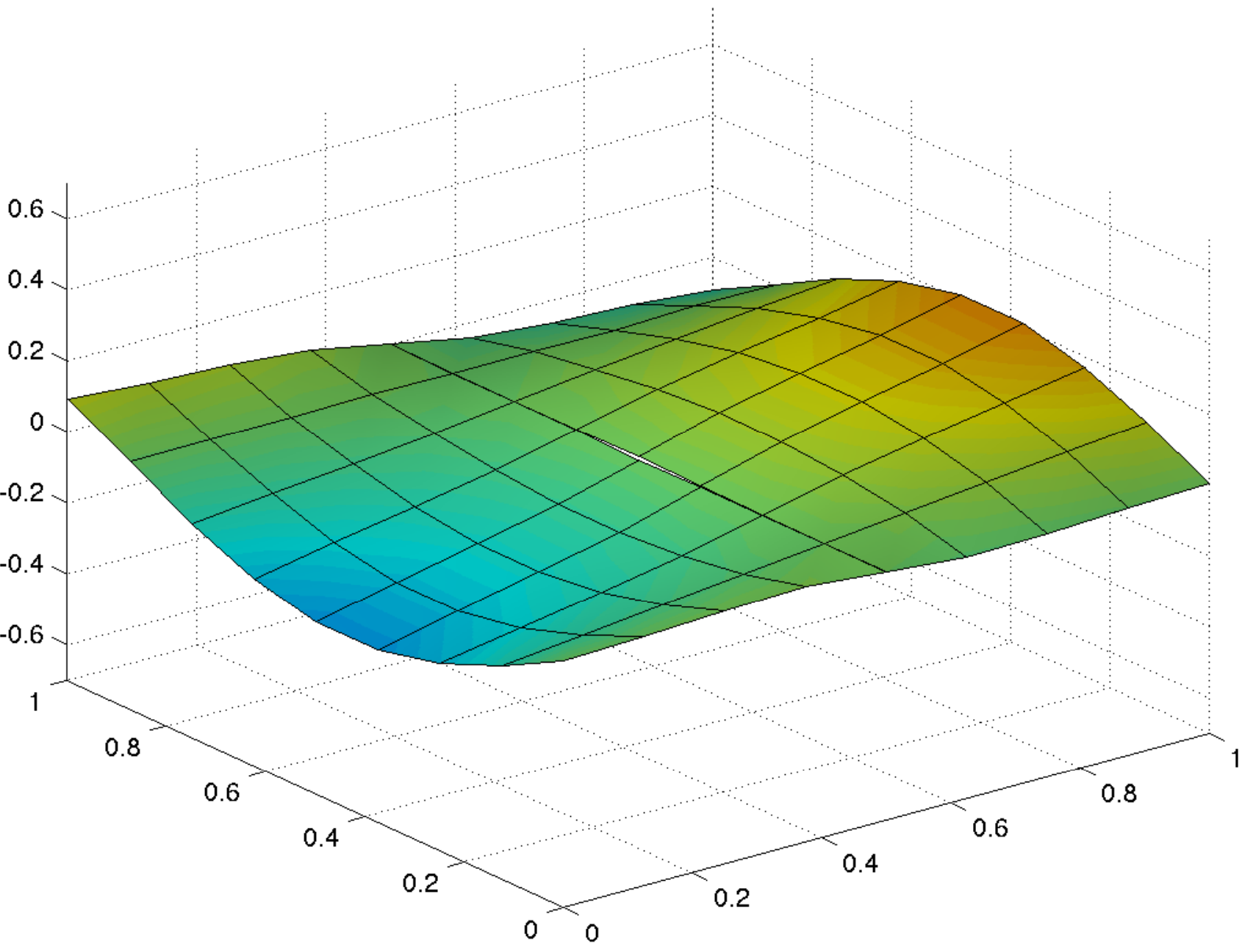}
 \includegraphics[width=.24\textwidth]{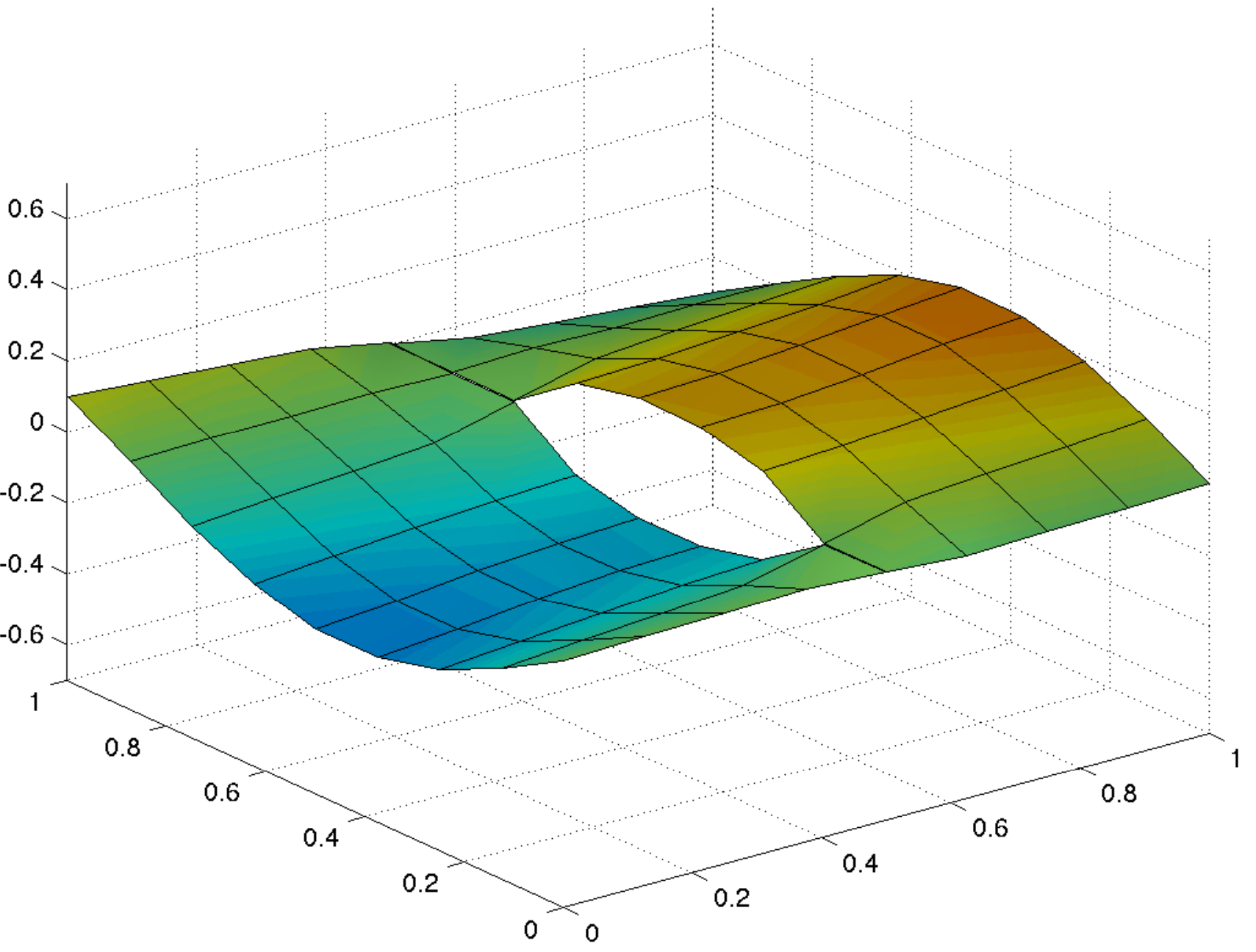}
  \includegraphics[width=.24\textwidth]{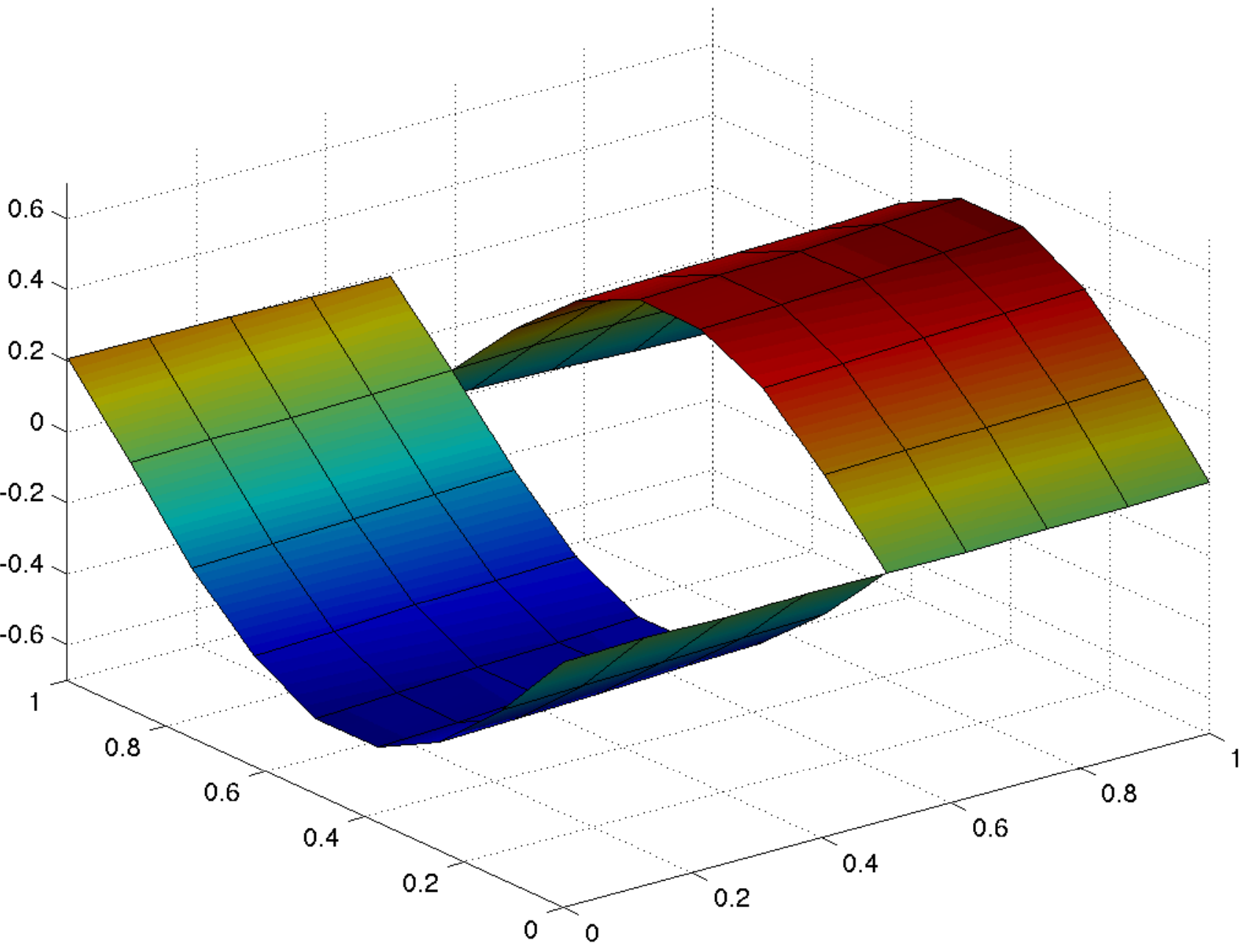}
\caption[Two dimensional cohesive crack evolution $R\le2\ell$ (2)]{The evolution of the quasistatic cohesive fracture for $R < 2\ell$ at time instances $t=0.02, 0.22, 0.24, 0.5$ with external displacement $\omega_2$.}\label{Noncohes2d_2}
\end{figure}

\section*{Acknowledgments}
Marco Artina and Massimo Fornasier acknowledge the financial support of the International Research Training Group IGDK 1754 ``Optimization and Numerical Analysis for Partial Differential Equation with Nonsmooth Structures" of the German Science Foundation.

Francesco Solombrino acknowledge the financial support of the ERC under Grant No. 290888 ``Quasistatic and Dynamic Evolution
Problems in Plasticity and Fracture" (P.I. Prof. G. Dal Maso).

\begin{thebibliography}{30}


\bibitem{Agostiniani} V.Agostiniani: 
 Second order approximations of quasistatic evolution problems in finite
dimension.
\textit{ Discrete Contin. Dyn. Syst. } 32/4 (2012), 1125--1167.

\bibitem{ARS} V.Agostiniani, R. Rossi and G. Savar\'e: 
Rate-independent limits of gradient flows:
a variational approach.
 (2016), in preparation.

\bibitem{AFS} M. Artina, M. Fornasier, F. Solombrino: 
Linearly constrained nonsmooth and nonconvex minimization.
\textit{SIAM Journal on Optimization} 23/3 (2013), 1904--1937.

\bibitem{A} S. Almi:
Energy release rate and quasistatic evolution via vanishing viscosity 
in a cohesive fracture model with an activation threshold.
\textit{ESAIM: Control Optim. Calc. Var.} (2016), published online. DOI: 10.1051/cocv/2016014.

\bibitem{Bar62} G.I. Barenblatt: The mathematical theory of
equilibrium cracks in brittle fracture. 
\textit{Adv. Appl. Mech.} \textbf{7} (1962), 55--129.

\bibitem{C} F. Cagnetti: A vanishing viscosity approach to fracture growth
in a cohesive zone model with prescribed crack path.
\textit{Math. Models Methods Appl. Sci. 18 (2008), no. 7, 1027--1071}. 

\bibitem{CT} F. Cagnetti, R. Toader:
Quasistatic crack evolution for a cohesive zone model 
with different response to loading and unloading:
a Young measures approach.
\textit{ESAIM Control Optim. Calc. Var. 17 (2011), 1--27}.

\bibitem{cava77} C. Castaing, M. Valadier:
Convex analysis and measurable multifunctions.
\textit{Lecture Notes in Mathematics},
580, Springer-Verlag, Berlin-New York, 1977.

\bibitem{DDMM} G. Dal Maso, A. DeSimone, M.G. Mora, M. Morini:
A vanishing viscosity approach to quasistatic evolution in plasticity with softening.
\textit{Arch. Ration. Mech. Anal.} 189 (2008), 469-544.

\bibitem{DM-Des-Sol}G. Dal Maso, A. De Simone, F. Solombrino: Quasistatic evolution for Cam-Clay plasticity: a weak formulation via viscoplastic regularization and time rescaling, \textit{Calc. Var. Partial Differential Equations\/} {\bf 40} (2011), 125-181.

\bibitem{DalFrToa} G. Dal Maso, G.A. Francfort,R. Toader:
Quasistatic crack growth in nonlinear elasticity.
\textit{Arch. Ration. Mech. Anal.}, 176/2 (2005), 165--225.

\bibitem{dmgipo09} G. Dal Maso, A. Giacomini, M. Ponsiglione:
A variational model for quasistatic crack growth in nonlinear elasticity: 
some qualitative properties of the solutions.
\textit{BUMI 9 (2009), Vol. 2, 371--390}.

\bibitem{DT02} G. Dal Maso, R. Toader:
A model for the quasistatic growth of brittle fractures:
existence and approximation results. 
\textit{Arch. Ration. Mech. Anal.} 162/2 (2002), 101--135.

\bibitem{DalToa02} G. Dal Maso, R. Toader:
A model for the quasistatic growth of brittle fractures
based on local minimization.
\textit{Math. Models Methods Appl. Sci.},
12/12 (2002), 1773--1799.

\bibitem{DMZ} G. Dal Maso, C. Zanini:
Quasistatic crack growth for a cohesive zone model 
with prescribed crack path.
\textit{Proc. Roy. Soc. Edinburgh Sect. A}, 
137A (2007), 253--279.

\bibitem{EM}
M.~A. Efendiev and A.~Mielke.
\newblock On the rate-independent limit of systems with dry friction and small
  viscosity.
\newblock {\em J. Convex Anal.}, 13(1) (2006), 151--167.


\bibitem{FG}
G.~A. Francfort and A.~Garroni.
\newblock A variational view of partial brittle damage evolution.
\newblock {\em Arch. Ration. Mech. Anal.}, 182(1) (2006),125--152.

\bibitem{Francfort-Larsen:2003}
 G.~A.~Francfort, C,~J.~Larsen. 
\newblock  Existence and convergence for quasistatic evolution in brittle fracture.
\newblock {\em Comm.\ Pure Appl.\ Math.\ }
\newblock {\bf 56} (2003), 1465--1500. 

\bibitem{FrMa98} G.A. Francfort, J.-J. Marigo: 
Revisiting brittle fracture as an energy minimization problem.
{\it J. Mech. Phys. Solids\/} {\bf 46} (1998), 1319-1342.

\bibitem{LT}
G.~Lazzaroni and R.~Toader.
\newblock Some remarks on the viscous approximation of crack growth.
\newblock {\em Discrete Contin. Dyn. Syst. Ser. S}, 6(1) (2013), 131--146.

\bibitem{KMZ}
D.~Knees, A.~Mielke, and C.~Zanini.
\newblock On the inviscid limit of a model for crack propagation.
\newblock {\em Math. Models Methods Appl. Sci.}, 18(9):1529--1569, 2008.

\bibitem{Neg-Knees}
D.~Knees and M.~Negri.
\newblock Convergence of alternate minimization schemes for phase field fracture and damage
\newblock {\em Preprint CVGMT},  2015.

\bibitem{KRZ}
D.~Knees, R.~Rossi, and C.~Zanini.
\newblock A vanishing viscosity approach to a rate-independent damage model.
\newblock {\em Math. Models Methods Appl. Sci.}, 23(4) (2013), 565--616.

\bibitem{Kne-Schr}
D.~Knees and A.~Schr{\"o}der.
\newblock Global spatial regularity for elasticity models with cracks, contact
  and other nonsmooth constraints.
\newblock {\em Math. Methods Appl. Sci.}, 35(15):1859--1884, 2012.

\bibitem{Kruz}
M.~Kru{\v{z}}{\'{\i}}k, C.~G. Panagiotopoulos, and T.~Roub{\'{\i}}{\v{c}}ek.
\newblock Quasistatic adhesive contact delaminating in mixed mode and its
  numerical treatment.
\newblock {\em Math. Mech. Solids}, 20(5) (2015), 582--599.

\bibitem{Mielke}
A.~Mielke.
\newblock Evolution of rate-independent systems.
\newblock In {\em Evolutionary equations. {V}ol. {II}}, Handb. Differ. Equ.,
  pages 461--559. Elsevier/North-Holland, Amsterdam, 2005.


\bibitem{MRS2}
A.~Mielke, R.~Rossi, and G.~Savar{\'e}.
\newblock Modeling solutions with jumps for rate-independent systems on metric
  spaces.
\newblock {\em Discrete Contin. Dyn. Syst.}, 25(2) (2009), 585--615.

\bibitem{MRoS}
A.~Mielke, R.~Rossi, and G.~Savar{\'e}.
\newblock B{V} solutions and viscosity approximations of rate-independent
  systems.
\newblock {\em ESAIM Control Optim. Calc. Var.}, 18(1) (2012), 36--80.


\bibitem{MRS}
A.~Mielke, T.~Roub{\'{\i}}{\v{c}}ek, and U.~Stefanelli.
\newblock {$\Gamma$}-limits and relaxations for rate-independent evolutionary
  problems.
\newblock {\em Calc. Var. Partial Differential Equations}, 31(3) (2008), 387--416.
  
\bibitem{Min-Sav}
L.~Minotti and G. Savar\'e.
\newblock Viscous corrections of the Time Incremental Minimization Scheme and Visco-Energetic Solutions to Rate-Independent Evolution Problems.
\newblock {\em Preprint CVGMT},  2016
 
\bibitem{Nardini} 
L.~Nardini. 
\newblock A note on the convergence of singularly perturbed second order potential-type equations.
\newblock {\em J. Dyn. Diff. Equat.} (2016), published online, 
DOI: 10.1007/s10884-015-9461-y. 
 
 
\bibitem{negri13}
M.~Negri.
\newblock Quasi-static rate-independent evolutions: characterization,
  existence, approximation and application to fracture mechanics.
\newblock {\em ESAIM Control Optim. Calc. Var.}, 20(4) (2014) 983--1008.

\bibitem{Neg-Ort}
M.~Negri and C.~Ortner.
\newblock Quasi-static crack propagation by {G}riffith's criterion.
\newblock {\em Math. Models Methods Appl. Sci.}, 18(11):1895--1925, 2008.

\bibitem{quarteroni08} 
A.~Quarteroni and A.~Valli.
\newblock {\em Numerical approximation of partial differential equations},
  volume~23.
\newblock Springer Science \& Business Media, 2008.

\bibitem{Roc}
Ralph~Tyrrell Rockafellar.
\newblock {\em Convex Analysis}.
\newblock Princeton {U}niversity {P}ress, 1970.

\bibitem{Serfaty}
E.~Sandier and S.~Serfaty.
\newblock Gamma-convergence of gradient flows with applications to
  {G}inzburg-{L}andau.
\newblock {\em Comm. Pure Appl. Math.}, 57(12) (2004), 1627--1672.

\bibitem{Stefanelli}
U.~Stefanelli.
\newblock A variational characterization of rate-independent evolution.
\newblock {\em Math. Nachr.}, 282(11) (2009), 1492--1512.

\bibitem{Vodicka2014}
R.~Vodi{\v{c}}ka, V.~Manti{\v{c}}, and T.~Roub{\'i}{\v{c}}ek.
\newblock Energetic versus maximally-dissipative local solutions of a
  quasi-static rate-independent mixed-mode delamination model.
\newblock {\em Meccanica}, 49(12) (2014), 2933--2963.

\bibitem{zanini} C.Zanini: 
 Singular perturbations of finite dimensional gradient flows.
\textit{ Discrete Contin. Dyn. Syst. } 18/4 (2007), 657--675.

\bibitem{Z} W.P. Ziemer, Weakly Differentiable Functions, Springer-Verlag, New York, 1989.





\end {thebibliography}

\end{document}